\documentclass[11pt]{article}
\usepackage[margin=1in]{geometry}
\usepackage{graphicx}
\usepackage{amssymb}
\usepackage{mathtools}
\usepackage{amsthm}
\usepackage{float}
\usepackage[export]{adjustbox}
\newcommand{\snapshot}[1]{\adjincludegraphics[width=0.22\linewidth, valign=m,
  trim={{0.16\width} {0.12\height} {0.015\width} {0.09\height}}, clip]{#1}}
\newcommand{\snapshotthreeD}[1]{\adjincludegraphics[width=0.23\linewidth, valign=m,
  trim={{0.01\width} {0.01\height} {0.02\width} {0.13\height}}, clip]{#1}}
\newtheorem{theorem}{Theorem}[section]
\newtheorem{lemma}{Lemma}[section]
\newtheorem{corollary}{Corollary}[section]
\newtheorem{scheme}{Scheme}
\theoremstyle{definition}
\newtheorem{remark}{Remark}[section]
\theoremstyle{plain}
\usepackage{subcaption}
\usepackage{booktabs}
\usepackage{makecell}
\usepackage{cite}
\usepackage{hyperref}

\newcommand{\diff}{\mathrm{d}}
\newcommand{\half}{\frac{1}{2}}
\newcommand{\phibar}{\overline{\phi}}
\numberwithin{equation}{section}
\newcommand{\cM}{\mathcal{M}}
\newcommand{\bx}{\mathbf{x}}
\newcommand{\cL}{\mathcal{L}}
\newcommand{\cA}{\mathcal{A}}
\newcommand{\cG}{\mathcal{G}}
\newcommand{\Csav}{C_{\mathrm{sav}}}

\newcommand{\beq}{\begin{equation}}
\newcommand{\eeq}{\end{equation}}
\newcommand{\bea}{\begin{array}}
\newcommand{\eea}{\end{array}}

\newcommand{\subfig}[3]{%
  \begin{subfigure}[b]{0.42\textwidth}%
    \centering%
    \includegraphics[width=\textwidth]{#1}%
    \caption{#2}\label{#3}%
  \end{subfigure}%
}

\begin{document}
\raggedbottom
\title{Relaxed Lagrange Multiplier (RLM) Schemes for Phase Field Models Preserving the Relaxed Original Energy Dissipation Law}
\author{Xiaobo Jing\thanks{School of Mathematics, Southeast University, Nanjing 210096, Jiangsu Province, P.R. China. Email: \texttt{xiaobo@seu.edu.cn}}
\and Jia Zhao\thanks{Department of Mathematics, University of Alabama, Tuscaloosa, AL 35487, USA. Email: \texttt{jia.zhao@ua.edu} (corresponding author)}}
\date{}
\maketitle

\begin{abstract}
Phase-field models are typically derived from variational principles for a free-energy functional and are widely used to simulate complex multiphase phenomena in science and engineering. A central goal in designing numerical schemes for these models is to preserve the underlying energy-dissipation law. In this paper, we propose a class of relaxed Lagrange multiplier (RLM) schemes for phase field models. In contrast to popular scalar auxiliary variable (SAV) and invariant energy quadratization (IEQ) methods, which dissipate a modified energy involving auxiliary variables, the RLM schemes dissipate a relaxed version of the original energy and closely track the original energy dissipation rate. Compared with the classical Lagrange multiplier (LM) approach, the RLM schemes ensure that the resulting discrete system is uniquely solvable over a broad range of time steps. The key idea is to augment the LM formulation with a relaxation term, yielding a scalar quadratic equation for the multiplier with an explicit closed-form solution. The resulting schemes are linear and efficient because each time step requires solving only two linear systems with constant coefficients, at a cost comparable to that of SAV schemes. We construct both first-order and second-order variants and prove their energy stability. Numerical experiments verify the expected convergence rates and demonstrate that the RLM schemes accurately capture interface dynamics.
\end{abstract}

\medskip

\section{Introduction}\label{sec:introduction}

Phase-field methods are widely used to model multiphase problems across many areas of science and engineering. In materials science, they are used to simulate microstructural evolution processes such as solidification, grain growth, and phase transformations~\cite{AC, Cahn2013Free, Karma98, McFadden1993}. In fluid dynamics, they are used to model multiphase flows, including droplet dynamics and bubble formation~\cite{Lowengrub1998Quasi, Anderson1998Diffuse, Liu2003A}. Some other examples include fracture mechanics, biological systems, and pattern formation~\cite{Chen98}.

Phase-field models are usually derived from variational principles for a free-energy functional that governs the system dynamics via an energy-dissipation law~\cite{Furihata2011, McLachlan&Quispel1999}. The core idea of phase-field methods is to introduce a diffuse interface variable whose thickness is controlled by an artificial parameter $\varepsilon$. Instead of explicitly tracking sharp interfaces, we solve PDEs for the phase variable over the entire domain, and the interface is recovered from level sets of the phase variable. Although this formulation is conceptually simple, it presents significant numerical challenges. The small parameter $\varepsilon$ yields stiff systems that require sufficiently fine spatial and temporal resolution, and the nonlinear potential terms should be carefully treated to ensure stability~\cite{Shen2010Numerical}. Moreover, numerical methods should preserve the underlying energy dissipation principle. Schemes that preserve this property are called energy-stable, and those that maintain stability regardless of the time step are termed unconditionally energy-stable.

Several classes of energy-stable numerical schemes have been developed for phase-field models. The convex splitting method~\cite{Eyre1998, Shen2012Second, Elliott1993The} separates the free energy into convex and concave components, treating them implicitly and explicitly, respectively. Stabilization techniques add artificial damping terms to enable explicit treatment of nonlinear terms, yielding linear schemes~\cite{Li2016Characterizing, Wang&YuJSC2018, Xu2006Stability}. Exponential time differencing (ETD) methods provide another class of numerical schemes for phase-field models. These methods treat the linear stiff terms exactly using matrix exponentials while handling the nonlinear terms explicitly or semi-implicitly, enabling large time steps without sacrificing accuracy~\cite{QiaoZhangTang2011mbe, ZhangQiao2012ch_adapt}. Recent advances have established maximum bound principles and energy stability for ETD schemes applied to Allen--Cahn and epitaxial growth models~\cite{DuJuLiQiao2019nac, JuLiQiaoZhang2018etd_mbe, ChengQiaoWang2019etd3, DuJuLiQiao2021mbp}. However, existing approaches in the literature still have limitations. Some of them are linear but require restrictive time-step constraints, while others require the solution of nonlinear systems at each time step, which can be computationally expensive~\cite{Wise2009AN, Hu2009Stable}.

The Invariant Energy Quadratization (IEQ) method~\cite{Yang2017Numerical, Yang2016Linear, Zhao2017Numerical, Zhao2018EQreview, GongEnergy} and the Scalar Auxiliary Variable (SAV) method~\cite{Shen2018The} have recently been developed by introducing an auxiliary variable that transforms the original PDEs into an equivalent system with a quadratic energy structure. This reformulation enables the construction of unconditionally energy-stable linear schemes. These methods have been widely adopted due to their simplicity and broad applicability~\cite{Xiaofeng2017Numerical, Guillen&Tierra2014, Li2017On}.
However, the energy dissipation laws in the SAV and IEQ methods are formulated in terms of auxiliary variables rather than the original phase variables. As truncation errors accumulate in the auxiliary variables during time stepping, the discrete auxiliary variables deviate from their continuous definitions. Consequently, while these schemes satisfy a modified energy dissipation law, they do not guarantee dissipation of the original free energy. This discrepancy has been partially addressed by the relaxed SAV (RSAV) method~\cite{jiang2022improving} and the relaxed IEQ approach~\cite{ZhaoIEQrelax}, both of which use projection techniques to restore consistency.

An alternative approach is the Lagrange multiplier (LM) method~\cite{LM-Cheng}, which introduces a scalar multiplier to enforce the original energy dissipation law directly. This method achieves dissipation of the original energy without requiring the potential to be bounded from below~\cite{EggerNM2019}. However, the Lagrange multiplier approach requires solving a nonlinear algebraic equation for the multiplier at each time step, which may have multiple solutions, complicating the analysis~\cite{Cheng2025_SIAM}. Several variants have been proposed to address this solvability issue~\cite{hou2023efficient,huang2025weighted}.  
The dynamically regularized Lagrange multiplier (DRLM) method introduces regularization parameters to improve solvability and has been applied to incompressible Navier--Stokes, Cahn--Hilliard--Navier--Stokes, and magnetohydrodynamic equations~\cite{doan2025dynamically, doan2025dynamically-1, wang2025efficient}.

Meanwhile, the supplementary variable method (SVM)~\cite{SVM-Gong} provides a general framework for structure-preserving schemes, with the Lagrange multiplier method as a special case. However, SVM still faces solvability challenges. The existence and uniqueness of solutions are not guaranteed, and the solution typically requires solving a nonlinear algebraic equation subject to time-step constraints. By combining the SAV and Lagrange multiplier methods, a weighted SAV method was introduced to address the shortcomings of both~\cite{huang2025weighted}. This approach is applicable to any time-step size, provided the weight coefficients are sufficiently large, thereby yielding a discrete energy that is closer to the original energy.

More broadly, several recent works improve the solvability and energy fidelity of auxiliary-variable and Lagrange multiplier schemes through various relaxation techniques, including a relaxed Lagrange multiplier approach for gradient flows~\cite{liu2024relaxedLM}, linear relaxation schemes for Allen--Cahn-type and Cahn--Hilliard-type models~\cite{jiang2023linearRelax}, a class of unconditionally energy-stable relaxation schemes for gradient flows~\cite{zhang2024relaxGF}, and multi-step SAV approaches that recover the original dissipation law~\cite{chen2025multiSAV}. We emphasize, however, that the term ``relaxation'' carries different meanings across these works, which differ from its meaning in the present method. In the relaxed SAV/IEQ and relaxed Lagrange multiplier methods~\cite{jiang2022improving, ZhaoIEQrelax, liu2024relaxedLM}, relaxation refers to a correction step applied to the auxiliary variable or multiplier after the main solve to restore consistency with the original energy. In the linear relaxation schemes~\cite{jiang2023linearRelax, zhang2024relaxGF}, it refers to introducing an auxiliary relaxation variable governed by its own relaxation dynamics. In the DRLM method~\cite{doan2025dynamically-1}, regularization parameters are introduced to stabilize the multiplier equation.

In summary, energy-stable schemes have progressed from conditionally energy-stable schemes (e.g., the stabilization method) to unconditionally energy-stable schemes (e.g., convex splitting), from schemes that dissipate a modified energy (e.g., the IEQ and SAV methods) to schemes that dissipate the original energy (e.g., the LM method). As the dissipated energy approaches the original energy, however, the resulting algebraic system tends to be harder to solve, often requiring smaller time steps and reducing computational efficiency. In the design of numerical schemes, maintaining mathematical elegance, preserving the physical properties of the PDE, and improving computational efficiency are not mutually exclusive. In schemes that conserve a modified energy, both the energy and the governing equations are reformulated; whereas in schemes that conserve the original energy, the governing equations are also moderately adjusted. From the standpoint of PDE well-posedness and numerical solutions, the stability of the algorithm is precisely a reflection of PDE well-posedness at the discrete level. Once the governing equations are modified during algorithm design, the well-posedness of the PDE must be reassessed. From a modeling perspective, any PDE arises from the desire to maintain mathematical elegance and physical fidelity, and thus, the modeling error may be substantial. If an appropriate modification to the PDE can unify these three objectives—mathematical elegance, physical structure-preserving, and computational efficiency—then such a modification is permissible. Therefore, we argue that, assuming no round-off error, modeling error and algorithmic error must be considered together. In fact, the numerical schemes cited in the paper have all taken modeling error into account, either deliberately or imperceptibly.

In this paper, we propose a class of energy-stable linear schemes for phase-field equations, the Relaxed Lagrange Multiplier (RLM) scheme. The RLM method applies to a general class of phase-field equations, including Allen--Cahn-type and Cahn--Hilliard-type equations as special cases. We make several contributions. First, we introduce a relaxation mechanism that augments the Lagrange multiplier formulation with a relaxation term parameterized by $\alpha \geq 0$. This yields a quadratic algebraic equation for the scaling factor $q^{n+1}$, which admits a closed-form solution. Compared with the LM scheme, the RLM scheme is solvable for a broad range of time step sizes once $\alpha$ is chosen sufficiently large. Moreover, for sufficiently large $\alpha$, the quadratic equation admits exactly one positive root, so the update of the scaling factor is uniquely determined. Second, the RLM scheme dissipates a relaxed original energy $\tilde{E} = E + \alpha (q^2-1)$, where $E$ is the original free energy. For suitable $\alpha$, we have $q \approx 1$, so the relaxation term $\alpha(q^2-1)$ is close to zero and $\tilde{E}$ closely tracks $E$. This differs from SAV/IEQ methods whose modified energies can deviate significantly from $E$. In particular, the formulation tracks the original energy dissipation rate when $q \approx 1$, adding only a controllable relaxation penalty. Third, the scheme is linear and requires solving only two linear systems per time step, with a computational cost comparable to that of SAV methods. We develop both first-order and second-order variants.
These properties make the RLM schemes attractive for long-time simulations of phase-field models while faithfully preserving the energy dissipation law. 
 Table~\ref{tab:method-comparison} summarizes the key differences between SAV, IEQ, RSAV, LM, WSAV, SVM, and the proposed RLM methods.

 \begin{table}[H]
\centering
 \caption{Comparison of energy-stable schemes for phase-field equations.}
 \label{tab:method-comparison}
 \begin{tabular*}{\textwidth}{@{\extracolsep{\fill}}lcccc@{}}
 \toprule
 Method & Energy Preserved & Cost per Step & Solvability & Requirement on $F$ \\
 \midrule
 SAV~\cite{Shen2019SAVreview} & Modified & 2 linear solves & Guaranteed & Bounded below \\
 IEQ~\cite{Yang2017Numerical} & Modified & 1 linear solve & Guaranteed & Bounded below \\
 RSAV~\cite{jiang2022improving} & Modified & 2 linear solves & Guaranteed & Bounded below \\
 LM~\cite{LM-Cheng} & Original & Nonlinear solve & May fail & None \\
 WSAV~\cite{huang2025weighted} & Modified & 2 linear solves & Guaranteed & Bounded below \\
 SVM~\cite{SVM-Gong} & Original & Nonlinear solve & May fail & None \\
 RLM (this work) & Relaxed original & 2 linear solves & Guaranteed & None \\
 \bottomrule
 \end{tabular*}
 \vspace{0.4em}
 \noindent\begin{minipage}{\linewidth}
 \footnotesize\raggedright
 \end{minipage}
 \end{table}

The remainder of this paper is organized as follows. In Section~\ref{sec:background}, we review the SAV and Lagrange multiplier methods for context. In Section~\ref{sec:rlm-methods}, we present the general RLM framework for phase-field models. In Section~\ref{sec:numerical}, we apply the RLM scheme to the Allen--Cahn, Cahn--Hilliard, and conserved Allen-Cahn equations with the double-well potential for binary flow, the Flory--Huggins potential for polymer solutions, and the Lennard--Jones-type potential for liquid thin films, and present numerical experiments demonstrating the convergence rates and comparing the RLM schemes with the SAV method. Finally, we conclude in Section~\ref{sec:conclusion}.

\section{Background and Preliminaries}\label{sec:background}

\subsection{General phase-field model framework}
We begin by establishing notation for general phase-field models. Let $\Omega \subset \mathbb{R}^d$ be a bounded domain. For $f, g \in L^2(\Omega)$, we define the standard inner product and norm as
\[
(f, g) = \int_\Omega f(\bx) g(\bx) \,\diff\bx, \qquad \|f\| = \sqrt{(f, f)}.
\]
Given a linear, self-adjoint, non-negative operator $\cA$, we denote the induced inner product and semi-norm by
\[
(f, g)_\cA = (\cA f, g), \qquad \| f\|_\cA = \sqrt{(f,f)_\cA}.
\]
The self-adjointness of $\cA$ ensures $(f,g)_\cA = (g,f)_\cA$, so $(\cdot,\cdot)_\cA$ is a symmetric nonnegative bilinear form with associated semi-norm $\|\cdot\|_\cA$. We use this notation with $\cA$ taken to be the operators $\cL$, $\cL_0$, and $\cM$ introduced below.

We consider a general free energy functional of the form
\beq \label{eq:general-free-energy}
E(\phi) = \half \|\phi\|_\cL^2 + \int_\Omega F(\phi) \,\diff\bx,
\eeq
where $\cL$ is a linear self-adjoint elliptic operator (e.g., $\cL = -\varepsilon^2 \Delta$) and $F(\phi)$ is a nonlinear potential function. We denote $f(\phi)\coloneqq F'(\phi)$ as the derivative of the potential. The phase-field evolution equation associated with the free energy~\eqref{eq:general-free-energy} is
\beq \label{eq:general-gradient-flow}
\left\{
\begin{aligned}
\partial_t \phi &= - \cM \mu, \\
\mu &\coloneqq\frac{\delta E}{\delta \phi} = \cL \phi + f(\phi),
\end{aligned}
\right.
\eeq
where $\mu$ is the chemical potential and $\cM \geq 0$ is the mobility operator. Two canonical choices of the mobility operator $\cM$ recover the Allen--Cahn (AC) type and Cahn--Hilliard (CH) type dynamics. For the AC type, we take $\cM = M I$ with $M>0$ a constant and $I$ the identity operator, which yields
\beq
\partial_t \phi = -M (\cL \phi + f(\phi)), \label{eq:AC-type-phi}
\eeq
For the CH type, we take $\cM = -M\Delta$ with mobility constant $M > 0$, which can be equivalently written as
\beq
\partial_t \phi = M \Delta (\cL \phi + f(\phi)). \label{eq:CH-type-phi}
\eeq
Moreover, under homogeneous Neumann or periodic boundary conditions for $\mu$, the CH-type system conserves mass,
\begin{equation}
\frac{d}{dt} \int_\Omega \phi \,\diff\bx = 0. \label{eq:CH-type-mass-conservation}
\end{equation}
This follows from integrating~\eqref{eq:CH-type-phi} over $\Omega$ and applying the divergence theorem with the boundary condition.

A key property of the phase-field equation~\eqref{eq:general-gradient-flow} associated with the energy~\eqref{eq:general-free-energy} is the energy dissipation law
\begin{equation} \label{eq:general-gradient-flow-energy-law}
\frac{d}{dt} E(\phi) = \left(\frac{\delta E}{\delta \phi}, \partial_t \phi\right) = -(\mu, \cM \mu) = -\|\mu\|_\cM^2 \leq 0.
\end{equation}
In other words, the free energy decreases monotonically over time, a fundamental thermodynamic principle that numerical schemes should preserve in phase-field models. 

The identity in Eq.~\eqref{eq:general-gradient-flow-energy-law} relies on the self-adjointness of $\cL$ and $\cM$ in $L^2(\Omega)$ under the chosen boundary conditions.
For simplicity, we assume homogeneous Neumann boundary conditions throughout this paper, though the approach extends naturally to other boundary conditions. We assume that $\cL$ and $\cM$ are linear, self-adjoint, and non-negative with respect to the $L^2$ inner product, with boundary conditions chosen so that all integration-by-parts identities used in the energy estimates are valid. In particular, for $\cL = -\varepsilon^2\Delta$ we impose the homogeneous Neumann condition $\mathbf{n}\cdot\nabla\phi = 0$ on $\partial\Omega$, and for the CH-type mobility $\cM = -M\Delta$ we additionally impose $\mathbf{n}\cdot\nabla\mu = 0$ on $\partial\Omega$. Periodic boundary conditions work equally well.

\paragraph{Notation for time discretization.}
To make the later explanation clearer and more consistent, we introduce some notation for numerical discretization. Consider the time interval $[0, T]$ and discretize it into equally spaced points $0=t_0 < t_1 <\cdots < t_N=T$ with step size $\Delta t = T/N$, so that $t_n = n\Delta t$. The numerical solution for a function $\phi(\bx, t)$ at time $t_n$ is denoted by $\phi^{n}(\bx)$. When there is no ambiguity, we write simply $\phi^{n}$ to represent $\phi^{n}(\bx)$.

We use the following extrapolation notation. For functions $\phi(\bx, t)$ and $q(t)$, we have
\begin{align}
& \phi^{n+\half} = \frac{1}{2}(\phi^{n} + \phi^{n+1}), \quad \phibar^{n+\half} = \frac{3}{2}\phi^{n} - \frac{1}{2}\phi^{n-1}, \quad \phibar^{\,n+1}=2\phi^{n}-\phi^{n-1}, \\
& q^{n+\half} = \frac{1}{2}(q^{n} + q^{n+1}).
\end{align}
Whenever an extrapolation involves values at earlier time levels (e.g., $\phi^{n-1}$), we assume the required starting values are available and set $q^0 = 1$, consistent with $q(0) = 1$.

\subsection{Review of the SAV method}
We briefly review the classical SAV method~\cite{Shen2018The, Shen2019SAVreview}. The key is to introduce auxiliary variables and reformulate the original problem in~\eqref{eq:general-gradient-flow} into an equivalent PDE with a quadratic energy structure.
For the classical SAV method, we introduce the scalar auxiliary variable
\beq \label{eq:sav-q}
q (t) = H(\phi(\bx,t))\coloneqq\sqrt{\int_\Omega \Big( F(\phi) - \frac{1}{2}\gamma_0\phi^2+ \Csav \Big) \,\diff\bx},
\eeq 
where $\gamma_0$ is a regularization parameter~\cite{ChenZhaoYang2018} and $\Csav>0$ is chosen so that $q(t)$ is well-defined, i.e., so that the quantity under the square root remains strictly positive for all relevant states $\phi$. Otherwise, $h(\phi)$ below becomes singular. Denote
\[
\cL_0 = \cL + \gamma_0 I, \quad h(\phi)\coloneqq \frac{\delta H}{\delta \phi}= \frac{F'(\phi)-\gamma_0\phi}{2H(\phi)}.
\]
Then Eq.~\eqref{eq:general-gradient-flow} is reformulated into an equivalent system
\beq \label{eq:SAV-general-gradient-flow}
\left\{
\begin{aligned}
&\partial_t \phi = -\cM \mu, \\
&\mu = \cL_0 \phi + 2h(\phi) q(t), \\
&\frac{d q(t)}{dt} =  \int_\Omega h(\phi) \partial_t \phi \,\diff\bx,
\end{aligned}
\right.
\eeq 
with a consistent initial condition for the auxiliary variable $q(0) = H(\phi(\bx,0))$. Here $h(\phi)=\delta H/\delta\phi$ is the variational derivative of the functional $H[\phi]=\sqrt{\int_\Omega(F(\phi)-\frac12\gamma_0\phi^2+\Csav)\,\diff\bx}$, and the evolution of $q$ in~\eqref{eq:SAV-general-gradient-flow} follows from the chain rule $\frac{dq}{dt}=\frac{dH}{dt}=\int_\Omega \frac{\delta H}{\delta\phi}\,\partial_t\phi\,\diff\bx$.
If we define the reformulated free energy as
\beq \label{eq:SAV-general-free-energy}
\hat E(\phi, q) = \half \|\phi\|_{\cL_0}^2 + q^2 - \Csav|\Omega|,
\eeq
then the energy dissipation law for Eq.~\eqref{eq:SAV-general-gradient-flow} is
\beq \label{eq:SAV-general-gradient-flow-energy-law}
\frac{d}{dt} \hat E(\phi, q)  = -\|\cL_0 \phi +2h(\phi)q(t)\|_\cM^2 \leq 0.
\eeq
Here $|\Omega|$ denotes the measure of $\Omega$. In particular, $\hat E(\phi, q) = E(\phi)$ when $q = H(\phi)$.

At the continuous level, $q(t) = H(\phi)$. The reformulated phase-field model in Eq.~\eqref{eq:SAV-general-gradient-flow} is equivalent to the original phase-field model in Eq.~\eqref{eq:general-gradient-flow}, and the modified energy law in Eq.~\eqref{eq:SAV-general-gradient-flow-energy-law} is equivalent to the original energy law in Eq.~\eqref{eq:general-gradient-flow-energy-law} when $q(t) = H(\phi)$. Thus, numerical schemes designed for Eq.~\eqref{eq:SAV-general-gradient-flow} are well-suited for solving Eq.~\eqref{eq:general-gradient-flow}. If we apply the backward difference formula (BDF) to the time derivative with a semi-implicit time discretization, we obtain the first-order SAV-BDF scheme as follows.

\begin{scheme}[SAV-BDF1 Scheme]\label{sch:SAV-BDF}
\beq\label{eq:SAV-BDF}
\left\{
\begin{aligned}
&\frac{\phi^{n+1} - \phi^n}{\Delta t}= -\cM \mu^{n+1}, \\
&\mu^{n+1} =\cL_0 \phi ^{n+1} +2h(\phi^n)q^{n+1} ,\\
&q^{n+1} - q^n = \int_\Omega h(\phi^n) (\phi^{n+1} - \phi^n) \,\diff\bx.
\end{aligned}
\right.
\eeq
\end{scheme}

Meanwhile, if we apply the Crank--Nicolson temporal discretization, we obtain the second-order SAV-CN scheme as follows.
\begin{scheme}[SAV-CN Scheme] \label{sch:SAV-CN}
\beq\label{eq:SAV-CN}
\left\{
\begin{aligned}
&\frac{\phi^{n+1} - \phi^n}{\Delta t} = -\cM \mu^{n+\half}, \\
&\mu^{n+\half} =\cL_0 \phi ^{n+\half} + 2h(\phibar^{n+\half}) q^{n+\half},\\
&q^{n+1} - q^n=\int_\Omega h(\phibar^{n+\half}) (\phi^{n+1} - \phi^n) \,\diff\bx.
\end{aligned}
\right.
\eeq
\end{scheme}
Scheme~\ref{sch:SAV-BDF} and Scheme~\ref{sch:SAV-CN} are both unconditionally energy-stable in the sense that 
$\hat E(\phi^{n+1}, q^{n+1})-\hat E(\phi^n, q^n)\leq0$.

\subsection{Review of the Lagrange Multiplier Method}
We briefly revisit the Lagrange multiplier (LM) framework for phase field models~\cite{LM-Cheng}. For the general phase-field equation in Eq.~\eqref{eq:general-gradient-flow}, the LM method introduces a scalar multiplier $q(t)$ to enforce dissipation of the original energy. The reformulated system reads
\beq \label{eq:LM-continuous}
\left\{
\begin{aligned}
& \partial_t\phi = -\cM \mu, \\
& \mu = \cL\phi + q(t) f(\phi), \quad f(\phi) = F'(\phi), \\
& \frac{d}{dt}\int_\Omega F(\phi)\,\diff\bx = q(t)\int_\Omega f(\phi)\,\partial_t\phi\,\diff\bx,
\end{aligned}
\right.
\eeq
with the consistent initial condition $q(0)=1$. Taking the $L^2$ inner product of the first relation in~\eqref{eq:LM-continuous} with $\mu$ and using the last identity yields
\[
\frac{dE}{dt}=-\|\mu\|_\cM^2 = - \|\cL\phi + q(t) f(\phi)\|_\cM^2  \leq 0,
\]
so the original energy decays monotonically without assuming that $F$ is bounded from below~\cite{LM-Cheng}. We note that the branch $q(t)\equiv 1$ recovers the original gradient flow~\eqref{eq:general-gradient-flow}. Indeed, combining the scalar constraint in~\eqref{eq:LM-continuous} with the chain rule $\frac{d}{dt}\int_\Omega F(\phi)\,\diff\bx=\int_\Omega f(\phi)\,\partial_t\phi\,\diff\bx$ shows that it reduces to $(q(t)-1)\int_\Omega f(\phi)\,\partial_t\phi\,\diff\bx = 0$, so $q(t)\equiv 1$ except at states where $\int_\Omega f(\phi)\,\partial_t\phi\,\diff\bx = 0$, at which the constraint does not uniquely determine $q(t)$. Compared to the SAV approach, the LM method dissipates the original energy and does not rely on a modified energy. However, solvability issues arise as discussed below.

Using the same notation, we can obtain the LM-BDF1 and LM-CN schemes as follows.
\begin{scheme}[LM-BDF1 Scheme] \label{sch:LM-BDF1}
\begin{align}
& \frac{\phi^{n+1}-\phi^n}{\Delta t} = -\cM\,\mu^{n+1}, \label{eq:LM-BDF1-1}\\
& \mu^{n+1} = \cL\,\phi^{n+1} + q^{n+1}\, f\big(\phi^{n}\big), \label{eq:LM-BDF1-2}\\
& \left( F(\phi^{n+1}) - F(\phi^{n}),\, 1 \right) = q^{n+1}\, \left( f \big( \phi^{n}\big),\, \phi^{n+1}-\phi^{n} \right). \label{eq:LM-BDF1-3}
\end{align}
\end{scheme}

\begin{scheme}[LM-CN Scheme] \label{sch:LM-CN}
\begin{align}
& \frac{\phi^{n+1}-\phi^n}{\Delta t} = -\cM\,\mu^{n+\half}, \label{eq:LM-CN-1}\\
& \mu^{n+\half} = \cL\,\phi^{n+\half} + q^{n+\half}\, f\big(\phibar^{n+\half}\big), \label{eq:LM-CN-2}\\
& \left( F(\phi^{n+1}) - F(\phi^{n}),\, 1 \right) = q^{n+\half}\, \left( f \big( \phibar^{n+\half}\big),\, \phi^{n+1}-\phi^{n} \right). \label{eq:LM-CN-3}
\end{align}
\end{scheme}

Scheme~\ref{sch:LM-BDF1} and Scheme~\ref{sch:LM-CN} are energy-stable in the sense that
\beq \label{eq:LM-CN-energy-stability}
E(\phi^{n+1})-E(\phi^n) \leq 0.
\eeq
However, the existence and uniqueness of solutions are not guaranteed, and solvability may depend on the time step size $\Delta t$~\cite{Cheng2025_SIAM}.

In practice, one time step of the LM--CN scheme proceeds as follows. Denote 
\beq \label{eq:LM-CN-affine-representation}
\phi^{n+1} = \phi_1^{n+1} + q^{n+1} \phi_2^{n+1}.
\eeq
\begin{enumerate}
\item Define the linear operator
\beq \label{eq:LM-CN-linear-operator}
\cG \coloneqq \frac{1}{\Delta t} I + \frac{1}{2}\cM\,\cL.
\eeq
We assume throughout that $\cG$ is invertible. This holds for the canonical AC and CH choices under the stated boundary conditions, since $\cM\cL \geq 0$ in both cases (for AC, $\cM\cL = -M\varepsilon^2\Delta \geq 0$, and for CH, $\cM\cL = M\varepsilon^2\Delta^2 \geq 0$), so that $\cG \geq \frac{1}{\Delta t} I > 0$.
Given \(\phi^n\), compute \(\phi_1^{n+1}\) and \(\phi_2^{n+1}\) by solving the two linear problems with the operator \(\cG\) in~\eqref{eq:LM-CN-linear-operator},
\[
\cG\, \phi_1^{n+1} = \frac{\phi^n}{\Delta t}-\frac{1}{2}\cM\,\cL\phi^n-\frac{1}{2}\cM\,q^n f\big(\phibar^{n+\half}\big),
\qquad
\cG\, \phi_2^{n+1} = -\frac{1}{2}\cM\,  f\big(\phibar^{n+\half}\big),
\]
\item Get $q^{n+1}$ by solving the scalar nonlinear equation
\begin{equation}\label{eq:eta_scalar}
\Big( F(\phi_1^{n+1} + q^{n+1} \phi_2^{n+1}) - F(\phi^n),\, 1 \Big)
\;=\;
q^{n+\half}\, \Big( f \big(\phibar^{n+\half}\big),\, \phi_1^{n+1} + q^{n+1} \phi_2^{n+1}-\phi^n \Big).
\end{equation}
Here $q^{n+\half}=\frac{1}{2}(q^n+q^{n+1})$, so Eq.~\eqref{eq:eta_scalar} is a scalar equation for the single unknown $q^{n+1}$. This can be solved, for example, by Newton's method with initial guess \(q^{n+1}=1\).
\item Recover the new solution \(\phi^{n+1}\) from the affine representation in Eq.~\eqref{eq:LM-CN-affine-representation},
and proceed to the next time step.
\end{enumerate}

\begin{remark}
Eq.~\eqref{eq:eta_scalar} is a scalar nonlinear algebraic equation for $q^{n+1}$. In the proposed RLM scheme, this relation is relaxed and, with a suitable quadratic approximation of $F^{n+1}$, reduces to a quadratic equation for the scaling factor that admits a closed-form solution.
\end{remark}

The SAV and LM methods each have advantages and limitations. A significant advantage of the LM method is that it dissipates the original energy $E$ without requiring $F$ to be bounded from below. On the other hand, at each time step, we must solve a nonlinear scalar equation for $q$, which may admit multiple solutions, making the stability and error analysis more involved.

\section{Relaxed Lagrange Multiplier (RLM) Methods}\label{sec:rlm-methods}

We now present our proposed Relaxed Lagrange Multiplier (RLM) schemes, which combine the advantages of both approaches while avoiding their drawbacks. The key idea is to augment the LM formulation with a relaxation term involving a parameter $\alpha \geq 0$. This relaxation serves two purposes. First, combined with a suitable quadratic approximation of $F^{n+1}$, it yields a quadratic equation for the multiplier that admits a closed-form solution, with the relaxation term $\alpha [(q^{n+1})^2-(q^n)^2]$ ensuring solvability and root-separation. Second, it provides additional flexibility in controlling the deviation of $q(t)$ from 1.
Unlike SAV methods, the RLM schemes preserve the energy dissipation law in the original variables. Unlike the standard LM method, they avoid solving a nonlinear equation and ensure that the algebraic system is uniquely solvable with a suitable relaxation parameter $\alpha$.

\subsection{RLM Reformulation}
For the general phase-field equation~\eqref{eq:general-gradient-flow}, we reformulate it as
\begin{align}
& \partial_t \phi = -\cM \mu, \label{eq:general-gradient-flow-discretized} \\
& \mu = \cL \phi + q(t)\,f(\phi), \quad f(\phi) = F'(\phi), \label{eq:general-gradient-flow-discretized-2} \\
& \frac{d}{dt} \int_\Omega F(\phi) \,\diff\bx + \alpha \,\frac{d \big(q(t)\big)^2 }{dt}= \int_\Omega q(t)\, f(\phi) \partial_t \phi \,\diff\bx, \label{eq:general-gradient-flow-scaled-2}
\end{align}
where $\alpha\ge 0$ is a relaxation parameter and $q(t)$ is a scaling factor with the consistent initial condition $q(0)=1$.
Taking the $L^2$ inner product of~\eqref{eq:general-gradient-flow-discretized} with $\mu$, expanding $\mu$ via~\eqref{eq:general-gradient-flow-discretized-2}, and using~\eqref{eq:general-gradient-flow-scaled-2} together with the time-independence and self-adjointness of $\cL$ and $\cM$, we obtain the relaxed energy dissipation law
\beq \label{eq:RLM-continuous-energy-law}
\frac{d}{dt}\Big[ E(\phi) + \alpha \big(q(t)^2-1\big) \Big] = -\|\cL\phi + q(t) f(\phi)\|_\cM^2 \leq 0.
\eeq
We refer to $\tilde{E}(\phi,q) := E(\phi) + \alpha (q^2-1)$ as the relaxed (original) energy. With the consistent initial condition $q(0)=1$, it satisfies $\tilde{E}(\phi(0),q(0)) = E(\phi(0))$.
Moreover, the reformulation is consistent with the original model~\eqref{eq:general-gradient-flow}. If $\phi$ solves~\eqref{eq:general-gradient-flow}, then the pair $(\phi, q\equiv 1)$ satisfies~\eqref{eq:general-gradient-flow-discretized}--\eqref{eq:general-gradient-flow-scaled-2}, since for $q\equiv 1$ Eq.~\eqref{eq:general-gradient-flow-scaled-2} reduces to the chain rule $\frac{d}{dt}\int_\Omega F(\phi)\,\diff\bx = \int_\Omega f(\phi)\,\partial_t\phi\,\diff\bx$. In this case $\tilde{E}(\phi,q\equiv 1) = E(\phi)$ exactly, so~\eqref{eq:RLM-continuous-energy-law} reduces to the original energy dissipation law~\eqref{eq:general-gradient-flow-energy-law}.

\begin{remark}
For $\alpha>0$, combining~\eqref{eq:general-gradient-flow-scaled-2} with the chain rule $\frac{d}{dt}\int_\Omega F(\phi)\,\diff\bx=\int_\Omega f(\phi)\,\partial_t\phi\,\diff\bx$ yields the explicit scalar ODE
\beq
\dot q(t)=\frac{\big(q(t)-1\big)\int_\Omega f(\phi)\,\partial_t\phi\,\diff\bx}{2\alpha\,q(t)}.
\eeq
When $q(0)=1$, it admits $q\equiv 1$ as a solution branch. This is the well-posedness improvement that the relaxation provides over the bare LM constraint.
\end{remark}

For the Allen--Cahn (AC) type, applying the RLM reformulation to the original AC-type equation~\eqref{eq:AC-type-phi} with $\cM = M I$ (where $M>0$ is a constant and $I$ is the identity operator), we obtain
\begin{align}
& \partial_t \phi = -M \mu, \label{eq:RLM-AC-type-phi} \\
& \mu = \cL \phi + q(t)\,f(\phi), \quad f(\phi) = F'(\phi), \label{eq:RLM-AC-type-mu} \\
& \frac{d}{dt} \int_\Omega F(\phi) \,\diff\bx + \alpha \,\frac{d \big(q(t)\big)^2 }{dt}= \int_\Omega q(t)\, f(\phi) \partial_t \phi \,\diff\bx. \label{eq:RLM-AC-type-q}
\end{align}

For the Cahn--Hilliard (CH) type, applying the RLM reformulation to the original CH-type equation~\eqref{eq:CH-type-phi} with $\cM = -M\Delta$ (where $M > 0$ is the mobility constant), we obtain
\begin{align}
& \partial_t \phi = M \Delta \mu, \label{eq:RLM-CH-type-phi} \\
& \mu = \cL \phi + q(t)\,f(\phi), \quad f(\phi) = F'(\phi), \label{eq:RLM-CH-type-mu} \\
& \frac{d}{dt} \int_\Omega F(\phi) \,\diff\bx + \alpha \,\frac{d \big(q(t)\big)^2 }{dt}= \int_\Omega q(t)\, f(\phi) \partial_t \phi \,\diff\bx. \label{eq:RLM-CH-type-q}
\end{align}
As with the original CH-type system, the RLM reformulation preserves mass conservation~\eqref{eq:CH-type-mass-conservation} under homogeneous Neumann boundary conditions ($\mathbf{n} \cdot \nabla \mu = 0$ on $\partial\Omega$),
\begin{equation}
\frac{d}{dt} \int_\Omega \phi \,\diff\bx = 0. \label{eq:RLM-CH-mass-conservation}
\end{equation}
This follows from integrating~\eqref{eq:RLM-CH-type-phi} over $\Omega$ and applying the divergence theorem with the boundary condition, just as in the original CH-type case. The introduction of the scaling factor $q(t)$ does not affect mass conservation, as it appears only in the chemical potential $\mu$ and not in the evolution equation for $\phi$.

\begin{remark}
If $\alpha=0$, such a strategy reduces to the Lagrange multiplier method~\cite{LM-Cheng} and one of the special cases of the supplementary variable method~\cite{SVM-Gong}.
\end{remark}

\begin{remark}
For simplicity, and to avoid obscuring the effects of the stabilization technique, we choose the regularization parameter $\gamma_0 = 0$~\cite{ChenZhaoYang2018} in the RLM scheme. A non-zero treatment, as in the SAV approach in Eq.~\eqref{eq:sav-q}, can be carried out similarly. Positive $\gamma_0$ increases the numerical stability by introducing a stabilization term as in~\cite{Wang&YuJSC2018}.
\end{remark}

\subsection{RLM Numerical Schemes}

Discretizing Eqs.~\eqref{eq:general-gradient-flow-discretized}--\eqref{eq:general-gradient-flow-scaled-2} using BDF1 or CN semi-implicit schemes yields the first-order RLM-BDF1 scheme and the second-order RLM-CN scheme as follows.

\begin{scheme}[RLM-BDF1 Scheme] \label{sch:RLM-BDF}
\begin{align}
& \frac{\phi^{n+1} - \phi^n}{\Delta t} = -\cM \mu^{n+1}, \label{eq:RLM-BDF-1} \\
& \mu^{n+1} = \cL \phi^{n+1} + q^{n+1} f(\phi^{n}), \label{eq:RLM-BDF-2} \\
& \int_\Omega \frac{F^{n+1} - F^n}{\Delta t} \,\diff\bx + \alpha \frac{ (q^{n+1})^2 - (q^n)^2}{\Delta t}= q^{n+1} \int_\Omega f(\phi^{n}) \frac{ \phi^{n+1} - \phi^n}{\Delta t} \,\diff\bx, \label{eq:RLM-BDF-3}
\end{align}
where $F^{n+1}$ and $F^n$ are approximations of $F(\phi)$ at the times $t^{n+1}$ and $t^n$, respectively. 
\end{scheme}

\begin{scheme}[RLM-CN Scheme] \label{sch:RLM-CN}
\begin{align}
& \frac{\phi^{n+1} - \phi^n}{\Delta t} = -\cM \mu^{n+\half}, \label{eq:RLM-CN-1} \\
& \mu^{n+\half} = \cL \phi^{n+\half} + q^{n+\half} f(\phibar^{n+\half}), \label{eq:RLM-CN-2} \\
& \int_\Omega \frac{F^{n+1} - F^n}{\Delta t} \,\diff\bx + \alpha \frac{ (q^{n+1})^2 - (q^n)^2}{\Delta t}= q^{n+\half} \int_\Omega f(\phibar^{n+\half}) \frac{ \phi^{n+1} - \phi^n}{\Delta t} \,\diff\bx, \label{eq:RLM-CN-3}
\end{align}
where $F^{n+1}$ and $F^n$ are approximations of $F(\phi)$ at the times $t^{n+1}$ and $t^n$, respectively. 
\end{scheme}

\begin{remark}
For instance, $F^{n+1} = F(\phi^{n+1})$ and $F^n = F(\phi^n)$. Special treatments of $F^{n+1}$ and $F^n$ will be discussed in Sections \ref{sec:rlm-quadratization} and \ref{sec:rlm-PC}.
\end{remark}

\begin{theorem}\label{thm:RLM-energy-stability}
Assume that at each time step the discrete system admits a solution $(\phi^{n+1}, q^{n+1})$ (see the solvability results in Section~\ref{sec:rlm-quadratization}). Schemes~\ref{sch:RLM-BDF} and~\ref{sch:RLM-CN} are energy-stable in the sense that
\begin{equation}
\tilde{E}(\phi^{n+1}, q^{n+1})-\tilde{E}(\phi^n, q^n) \leq 0,
\end{equation}
where $\tilde{E}(\phi^n, q^n) = \frac{1}{2} \|\phi^n\|_\cL^2 + \int_\Omega F^n \,\diff\bx + \alpha \big[(q^n)^2-1\big]$.
\end{theorem}

\begin{proof}
We first prove the result for Scheme~\ref{sch:RLM-CN} and then indicate the modification for Scheme~\ref{sch:RLM-BDF}.

Taking the $L^2$ inner products of~\eqref{eq:RLM-CN-1} and~\eqref{eq:RLM-CN-2} with $\mu^{n+\half}$ and $\frac{\phi^{n+1}-\phi^n}{\Delta t}$, respectively, we obtain
\begin{equation}\label{eq:RLM-CN-energy-balance}
\left(\phi^{n+\half},\frac{\phi^{n+1}-\phi^n}{\Delta t}\right)_\cL
+\left(q^{n+\half}f(\phibar^{n+\half}),\frac{\phi^{n+1}-\phi^n}{\Delta t}\right)
=-\|\mu^{n+\half}\|_\cM^2.
\end{equation}
Using the symmetry of $(\cdot,\cdot)_\cL$ and $\phi^{n+\half}=\frac12(\phi^{n+1}+\phi^n)$,
\[
\left(\phi^{n+\half},\frac{\phi^{n+1}-\phi^n}{\Delta t}\right)_\cL
=\frac{1}{2\Delta t}\Big[\|\phi^{n+1}\|_\cL^2-\|\phi^n\|_\cL^2\Big].
\]
Substituting this identity and~\eqref{eq:RLM-CN-3} into~\eqref{eq:RLM-CN-energy-balance}, multiplying by $\Delta t$, and rearranging yields
\[
\tilde{E}(\phi^{n+1}, q^{n+1})-\tilde{E}(\phi^n, q^n)=-\Delta t\,\|\mu^{n+\half}\|_\cM^2\le 0,
\]
where the inequality holds since $\|\cdot\|_\cM^2 \geq 0$ by the non-negativity of $\cM$.

For Scheme~\ref{sch:RLM-BDF}, we take the $L^2$ inner products of~\eqref{eq:RLM-BDF-1} and~\eqref{eq:RLM-BDF-2} with $\mu^{n+1}$ and $\frac{\phi^{n+1}-\phi^n}{\Delta t}$, respectively, and use the identity
\[
\big(\phi^{n+1},\phi^{n+1}-\phi^n\big)_\cL
=\frac{1}{2}\Big[\|\phi^{n+1}\|_\cL^2-\|\phi^n\|_\cL^2+\|\phi^{n+1}-\phi^n\|_\cL^2\Big].
\]
Substituting~\eqref{eq:RLM-BDF-3} and multiplying by $\Delta t$ yields
\[
\tilde{E}(\phi^{n+1}, q^{n+1})-\tilde{E}(\phi^n, q^n)
=-\Delta t\,\|\mu^{n+1}\|_\cM^2-\frac{1}{2}\|\phi^{n+1}-\phi^n\|_\cL^2\le 0,
\]
which carries an additional non-negative dissipation term $\frac{1}{2}\|\phi^{n+1}-\phi^n\|_\cL^2$ from the non-negativity of $\cL$.
\end{proof}

We emphasize that Theorem~\ref{thm:RLM-energy-stability} holds for any choice of the approximations $F^{n+1}$ and $F^n$, provided the scalar relation~\eqref{eq:RLM-BDF-3} or~\eqref{eq:RLM-CN-3} is enforced exactly with the increment $\phi^{n+1}-\phi^n$. This is the case for the RLM-Q variant introduced in the next subsection.

\begin{remark}[Relaxed Energy vs.\ Original Energy]
The relaxed energy
\[
\tilde{E}(\phi^n, q^n) = \frac{1}{2}\|\phi^n\|_\cL^2 + \int_\Omega F^n\,\diff\bx + \alpha \big[(q^n)^2-1\big]
\]
differs from the original energy
\[
E(\phi^n) = \frac{1}{2}\|\phi^n\|_\cL^2 + \int_\Omega F(\phi^n)\,\diff\bx
\]
by the exact decomposition
\[
\tilde{E}(\phi^n, q^n) - E(\phi^n) = \int_\Omega \big(F^n - F(\phi^n)\big)\,\diff\bx + \alpha\big[(q^n)^2 - 1\big].
\]
For a consistent discretization, the first term satisfies 
$$
\int_\Omega (F^n - F(\phi^n))\,\diff\bx = O(\Delta t^k),
$$
where $k$ is the order of the surrogate approximation of $F$. 
The relaxation term $\alpha((q^n)^2-1)$ is observed numerically to be of the same order, $O(\Delta t^{k})$, so that
\[
|\tilde{E}(\phi^n, q^n) - E(\phi^n)| \lesssim O(\Delta t^k).
\]
A rigorous bound on $|q^n-1|$ in terms of $\alpha$ and $\Delta t$, which would make this estimate precise, is left for future work.
\end{remark}

Both RLM-BDF1 and RLM-CN satisfy the mass/volume conservation law with the phase variable $\phi$, which usually represents the mass fraction or volume fraction. Here, we present only the proof for RLM-CN, since the proofs are analogous. 
\begin{theorem}[Mass/Volume Conservation]\label{thm:RLM-mass-conservation}
The RLM-CN scheme (Scheme~\ref{sch:RLM-CN}) for the Cahn--Hilliard equation preserves mass/volume, i.e.,
\[
\int_\Omega \phi^{n+1}\,\diff\bx = \int_\Omega \phi^n\,\diff\bx
\]
for all $n \geq 0$, under homogeneous Neumann boundary conditions $\mathbf{n}\cdot\nabla\mu = 0$ on $\partial\Omega$ (or periodic boundary conditions).
\end{theorem}
\begin{proof}
Integrating equation~\eqref{eq:RLM-CN-1} over the domain $\Omega$, we obtain
\[
\int_\Omega \frac{\phi^{n+1} - \phi^n}{\Delta t}\,\diff\bx = M \int_\Omega \Delta \mu^{n+\half}\,\diff\bx.
\]
Applying the divergence theorem to the right-hand side (with $\mu^{n+\half}$ taken sufficiently regular for the divergence theorem to apply, consistent with the standing boundary-condition assumptions) yields
\[
\int_\Omega \Delta \mu^{n+\half}\,\diff\bx = \int_\Omega \nabla \cdot (\nabla \mu^{n+\half})\,\diff\bx = \int_{\partial\Omega} \mathbf{n} \cdot \nabla \mu^{n+\half}\,\diff S.
\]
Under homogeneous Neumann boundary conditions, $\mathbf{n} \cdot \nabla \mu^{n+\half} = 0$ on $\partial\Omega$, so the boundary integral vanishes. Under periodic boundary conditions, the contributions from opposite faces of $\partial\Omega$ cancel, so the boundary integral likewise vanishes. Then, we have
\[
\int_\Omega \Delta \mu^{n+\half}\,\diff\bx = 0.
\]
Substituting back and multiplying by $\Delta t$ gives
\[
\int_\Omega \phi^{n+1}\,\diff\bx - \int_\Omega \phi^n\,\diff\bx = 0,
\]
which establishes the conservation property. 
\end{proof}

The procedures to solve Scheme~\ref{sch:RLM-BDF} and Scheme~\ref{sch:RLM-CN} are similar. Taking Scheme~\ref{sch:RLM-CN} as an example, we express $\phi^{n+1}$ as an affine function of $q^{n+1}$,
\[
\phi^{n+1} = \phi_1^{n+1} + q^{n+1} \phi_2^{n+1},
\]
by solving the following system of equations.
\begin{itemize}
\item Step 1: Solve $\phi_1^{n+1}$ and $\phi_2^{n+1}$. Denote the operator
\[
\cG \coloneqq \frac{1}{\Delta t} I + \frac{1}{2}\cM \,\cL.
\]
Then, solve $\phi_1^{n+1}$ and $\phi_2^{n+1}$ from the following linear system,
\begin{align}
& \cG \phi_1^{n+1} = \frac{1}{\Delta t}\phi^n-\frac{1}{2}\cM \cL\phi^n - \frac{1}{2}\cM q^n f(\phibar^{n+\half});
\label{eq:RLM-CN-step1} \\
& \cG \phi_2^{n+1} = -\frac{1}{2} \cM f(\phibar^{n+\half}).   \label{eq:RLM-CN-step2}
\end{align}
\item Step 2: Solve $q^{n+1}$. Using the discrete energy equation~\eqref{eq:RLM-CN-3}, we solve $q^{n+1}$ from the following algebraic equation,
\begin{equation}\label{eq:RLM-CN-step3}
\int_\Omega \frac{F^{n+1} - F^n}{\Delta t} \,\diff\bx + \alpha \frac{ (q^{n+1})^2 - (q^n)^2}{\Delta t}= q^{n+\half} \int_\Omega f(\phibar^{n+\half}) \frac{ \phi^{n+1} - \phi^n}{\Delta t} \,\diff\bx.
\end{equation}
With a suitable approximation of $F^{n+1}$.
\item Step 3: Update $\phi^{n+1}=\phi_1^{n+1}+q^{n+1}\phi_2^{n+1}$.
\end{itemize}

\begin{remark}
Eq.~\eqref{eq:RLM-CN-step3} is an algebraic equation for $q^{n+1}$.
For instance, if we approximate $F^{n+1}$ and $F^n$ by $F^{n+1} = F(\phi^{n+1})$ and $F^n = F(\phi^n)$, then Eq.~\eqref{eq:RLM-CN-step3} can be rewritten as
\begin{align}
& \int_\Omega \Big( F(\phi_1^{n+1} + q^{n+1} \phi_2^{n+1}) - F(\phi^n) \Big) \,\diff\bx
+ \alpha \Big[ (q^{n+1})^2 - (q^n)^2 \Big] \notag \\
&\qquad = q^{n+\half} \int_\Omega f(\phibar^{n+\half}) \big(\phi_1^{n+1} + q^{n+1} \phi_2^{n+1} - \phi^n\big)\,\diff\bx. \label{eq:RLM-CN-step3-algebra}
\end{align}
\end{remark}

Once $\phi_1^{n+1}$ and $\phi_2^{n+1}$ are obtained from Eq.~\eqref{eq:RLM-CN-step1} and Eq.~\eqref{eq:RLM-CN-step2}, respectively, we can determine $q^{n+1}$ by solving the algebraic Eq.~\eqref{eq:RLM-CN-step3}.
Unfortunately, due to the nonlinearity of $F(\phi)$, the existence and uniqueness of a solution for $q^{n+1}$ in Eq.~\eqref{eq:RLM-CN-step3} are not guaranteed, as in the LM and SVM approaches \cite{LM-Cheng, SVM-Gong}. We can expect the solvability to depend on the time step size $\Delta t$.

\subsection{RLM Schemes with Quadratization}\label{sec:rlm-quadratization}
The solvability issue identified above motivates the development of RLM schemes that guarantee the existence of a closed-form solution for $q^{n+1}$. The key idea is to approximate $F^{n+1}$ in such a way that Eq.~\eqref{eq:RLM-CN-step3} becomes a quadratic equation in $q^{n+1}$, which admits an explicit solution via the quadratic formula. 

In our first approach, we approximate $F^{n+1}$ by a quadratic function of $\phi^{n+1}$ in a semi-implicit manner. We refer to this approach as relaxed Lagrange multiplier with quadratization (RLM-Q). 
This approach is practical for a broad class of potentials. For instance, we can introduce a stabilization constant $S > 0$ such that 
\beq \label{eq:F-quadratize-stabilization}
F^{n+1} = S (\phi^{n+1})^2 + F(\phibar^{n+1}) - S (\phibar^{n+1})^2.
\eeq
Therefore, Eq.~\eqref{eq:RLM-CN-step3} is a quadratic equation for $q^{n+1}$.

In particular, we can refine this approximation by properly exploring $F(\phi)$. For instance, when $F(\phi) = \frac{1}{4}(\phi^2 -1)^2$, we can approximate $F^{n+1}$ and $F^n$ by
\beq \label{eq:F-quadratize}
F^{n+1} = \frac{1}{4}(\phi^{n+1} \phibar^{n+1} - 1)^2, \quad F^n = \frac{1}{4} (\phi^n \phibar^n - 1)^2.
\eeq
After multiplying through by $\Delta t$ to clear the common $1/\Delta t$ factor, Eq.~\eqref{eq:RLM-CN-step3} is simplified as
\begin{align}
& \int_\Omega \Big( \frac{1}{4}(\phi^{n+1} \phibar^{n+1} - 1)^2 - \frac{1}{4} (\phi^n \phibar^n - 1)^2 \Big) \diff\bx
+ \alpha \Big[ (q^{n+1})^2 - (q^n)^2 \Big] \notag \\
&\qquad = q^{n+\half} \int_\Omega f(\phibar^{n+\half}) \big(\phi^{n+1} - \phi^n\big)\,\diff\bx. \label{eq:RLM-CN-Q-algebra}
\end{align}
This expression is quadratic in $\phi^{n+1}$ and, through the affine representation $\phi^{n+1}=\phi_1^{n+1}+q^{n+1}\phi_2^{n+1}$, in the unknown $q^{n+1}$. For Eq.~\eqref{eq:F-quadratize}, we can write
\[
\int_\Omega F^{n+1} \,\diff\bx = A_0 (q^{n+1})^2 + B_0 q^{n+1} + C_0.
\]
with the coefficients calculated as
\[
A_0  = \int_\Omega \frac{1}{4}(\phibar^{n+1} \phi_2^{n+1})^2 \,\diff\bx, \quad
B_0  = \int_\Omega \frac{1}{2} (\phibar^{n+1} \phi_1^{n+1} - 1) \phibar^{n+1} \phi_2^{n+1} \,\diff\bx,
\]
\[
C_0  = \int_\Omega \frac{1}{4} (\phibar^{n+1} \phi_1^{n+1} -1)^2 \,\diff\bx.
\]

Therefore, Eq.~\eqref{eq:RLM-CN-Q-algebra} is a quadratic equation for $q^{n+1}$. We can write it as
\begin{equation} \label{eq:RLM-Q-quadratic}
A_1 (q^{n+1})^2+B_1 q^{n+1}+C_1=0, 
\end{equation}
where the coefficients for the quadratic equations are derived as
\begin{align*}
A_1  & = \tilde{A}_0 + \alpha, \quad \tilde{A}_0 =  A_0 - \half \int_\Omega f(\phibar^{n+\half}) \phi_2^{n+1} \,\diff\bx\\
B_1 & = \tilde{B}_0, \quad \tilde{B}_0 = B_0 - \half \int_\Omega f(\phibar^{n+\half}) (\phi_1^{n+1} - \phi^n) \,\diff\bx - \half q^n \int_\Omega f(\phibar^{n+\half}) \phi_2^{n+1}\,\diff\bx, \\
C_1 & = \tilde{C}_0 - \alpha (q^n)^2, \quad \tilde{C}_0 = C_0 - \int_\Omega F^n \,\diff\bx - \half q^n \int_\Omega f(\phibar^{n+\half}) (\phi_1^{n+1} - \phi^n)\,\diff\bx.
\end{align*}
Here $\int_\Omega F^n\,\diff\bx$ denotes the same quadratized surrogate as in~\eqref{eq:F-quadratize}, namely $\int_\Omega F^n\,\diff\bx=\int_\Omega \tfrac14(\phi^n\phibar^n-1)^2\,\diff\bx$.
The following lemma formalizes this solvability result.

\begin{lemma}[Solvability of RLM-Q Quadratic Equation]\label{lem:RLM-Q-solvability}
Assume $q^n>0$ (consistent with the root-selection criterion $q^{n+1}\approx q^n\approx 1$).
Then Eq.~\eqref{eq:RLM-Q-quadratic} has at least one real solution in the following cases:
\begin{enumerate}
\item If $(\tilde{A}_0 (q^n)^2+\tilde{C}_0)^2 < (q^n)^2\tilde{B}_0^2$, then the equation is solvable for any $\alpha \geq 0$.
\item If $(\tilde{A}_0 (q^n)^2+\tilde{C}_0)^2 \geq (q^n)^2\tilde{B}_0^2$, then the equation is solvable when
\[
\alpha \geq \alpha^* \coloneqq \max\left\{0, \frac{\tilde{C}_0 - \tilde{A}_0(q^n)^2+\sqrt{(\tilde{A}_0 (q^n)^2+\tilde{C}_0)^2-(q^n)^2\tilde{B}_0^2} }{2(q^n)^2}\right\}.
\]
\end{enumerate}
\end{lemma}

\begin{proof}
Assume $q^n>0$.
If $A_1 = \tilde{A}_0 + \alpha = 0$, Eq.~\eqref{eq:RLM-Q-quadratic} degenerates to a linear equation, which is solvable whenever $B_1 \neq 0$ (and if $A_1 = B_1 = 0$, a solution exists only when $C_1 = 0$, in which case every $q^{n+1}$ is a solution). We henceforth assume $A_1 \neq 0$ (this is guaranteed, e.g., when $\alpha > -\tilde{A}_0$, cf.\ Corollary~\ref{cor:RLM-Q-unique}).
The discriminant of equation~\eqref{eq:RLM-Q-quadratic} is
\begin{align*}
\Delta_1 &= B_1^2 - 4A_1C_1 \\
&= \tilde{B}_0^2 - 4(\tilde{A}_0 + \alpha)(\tilde{C}_0 - \alpha(q^n)^2) \\
&= \tilde{B}_0^2 - 4\tilde{A}_0\tilde{C}_0 + 4\alpha\tilde{A}_0(q^n)^2 - 4\alpha\tilde{C}_0 + 4\alpha^2(q^n)^2 \\
&= 4(q^n)^2\alpha^2 + 4\big(\tilde{A}_0(q^n)^2 - \tilde{C}_0\big)\alpha + \big(\tilde{B}_0^2 - 4\tilde{A}_0\tilde{C}_0\big).
\end{align*}
This is a quadratic polynomial in $\alpha$ with leading coefficient $4(q^n)^2 > 0$. Therefore, $\Delta_1 \to +\infty$ as $\alpha \to \infty$.

For the discriminant $\Delta_1 \geq 0$, we require the quadratic in $\alpha$ to be non-negative. The discriminant of this quadratic in $\alpha$ is
\[
D_\alpha = 16\big[(\tilde{A}_0(q^n)^2 + \tilde{C}_0)^2 - (q^n)^2\tilde{B}_0^2\big].
\]

If $D_\alpha < 0$, i.e., $(\tilde{A}_0(q^n)^2 + \tilde{C}_0)^2 < (q^n)^2\tilde{B}_0^2$, then the quadratic in $\alpha$ has no real roots. Since its leading coefficient is positive and it has no real roots, the quadratic in $\alpha$ is strictly positive, so $\Delta_1 > 0$ for all $\alpha \geq 0$, and the equation \eqref{eq:RLM-Q-quadratic} is solvable for any $\alpha \geq 0$.

If $D_\alpha \geq 0$, i.e., $(\tilde{A}_0(q^n)^2 + \tilde{C}_0)^2 \geq (q^n)^2\tilde{B}_0^2$, then the quadratic in $\alpha$ has two real roots. Since the leading coefficient is positive, $\Delta_1 \geq 0$ when $\alpha$ lies outside the interval between these roots. The larger root gives the threshold $\alpha^*$,
\[
\alpha^* = \max\left\{0, \frac{\tilde{C}_0 - \tilde{A}_0(q^n)^2 + \sqrt{(\tilde{A}_0(q^n)^2 + \tilde{C}_0)^2 - (q^n)^2\tilde{B}_0^2}}{2(q^n)^2}\right\}.
\]
We note that $\Delta_1 \geq 0$ also holds for $\alpha$ at or below the smaller root. We state only the large-$\alpha$ branch, which is the regime of practical interest.
When two real roots exist for $q^{n+1}$, we select the root closest to $q^n$ to ensure continuity of the scaling factor.
\end{proof}

The solvability conditions in Lemma~\ref{lem:RLM-Q-solvability} guarantee the existence of a real root. The following corollary shows that the quadratic admits exactly one positive root. For sufficiently large $\alpha$, the two roots approach $\pm q^n$, so the closest-to-$q^n$ selection rule returns this positive root, and the update of the scaling factor is uniquely determined.

\begin{corollary}[Unique Positive Root for RLM-Q]\label{cor:RLM-Q-unique}
Assume $q^n > 0$. If
$
\alpha > \max\left\{ -\tilde{A}_0,\; \frac{\tilde{C}_0}{(q^n)^2} \right\},
$
then $A_1 > 0$ and $C_1 < 0$ in~\eqref{eq:RLM-Q-quadratic}. Consequently, the discriminant satisfies $\Delta_1 = B_1^2 - 4A_1C_1 > 0$, and the two real roots of~\eqref{eq:RLM-Q-quadratic} have opposite signs. Hence Eq.~\eqref{eq:RLM-Q-quadratic} admits exactly one positive root. For sufficiently large $\alpha$, the two roots approach $\pm q^n$, so the closest-to-$q^n$ selection rule returns this positive root, and $q^{n+1}$ is uniquely determined.
\end{corollary}

\begin{proof}
The condition $\alpha > -\tilde{A}_0$ gives $A_1 = \tilde{A}_0 + \alpha > 0$, and $\alpha > \tilde{C}_0/(q^n)^2$ gives $C_1 = \tilde{C}_0 - \alpha (q^n)^2 < 0$. Then $\Delta_1 = B_1^2 - 4A_1C_1 \geq -4A_1C_1 > 0$, so~\eqref{eq:RLM-Q-quadratic} has two distinct real roots. The product of the roots equals $C_1/A_1 < 0$, so the roots have opposite signs, and exactly one of them is positive.
\end{proof}

\begin{remark}
In particular, when $\alpha$ is sufficiently large ($\alpha \gg |\tilde A_0|, |\tilde B_0|, |\tilde C_0| $ ), the terms involving $\alpha$ dominate the coefficients in~\eqref{eq:RLM-Q-quadratic}. In this limit,~\eqref{eq:RLM-Q-quadratic} reduces to $\alpha[(q^{n+1})^2 - (q^n)^2] = 0$, and the discriminant satisfies $\Delta_1=B_1^2 - 4A_1C_1 = 4\alpha^2(q^n)^2 > 0$, which means that Eq. \eqref{eq:RLM-Q-quadratic} is always solvable. 
Combined with Theorem~\ref{thm:RLM-energy-stability}, the RLM-Q-CN scheme is unconditionally energy-stable provided $\alpha$ is sufficiently large and $|\tilde A_0|$, $|\tilde B_0|$, and $|\tilde C_0|$ are bounded.
\end{remark}

\begin{remark}[Root Selection Criterion]
When the quadratic equation for $q^{n+1}$ admits two real roots, we select the root closest to $q^n$. This choice is justified by three considerations. At the continuous level, $q(t)$ evolves continuously with $q(0)=1$, so the discrete approximation $q^{n+1}$ should remain close to $q^n$ for small time steps. The criterion also ensures temporal stability and prevents spurious jumps in the scaling factor. Finally, for sufficiently large $\alpha$, the two roots approach $\pm q^n$, and selecting the root closest to $q^n$ naturally yields the positive root that maintains $q^{n+1} \approx 1$.
In particular, under the conditions of Corollary~\ref{cor:RLM-Q-unique}, exactly one of the two roots is positive. For sufficiently large $\alpha$, the two roots approach $\pm q^n$, so the closest-to-$q^n$ selection criterion picks precisely this positive root.
\end{remark}
\subsection{RLM with Prediction-Correction (RLM-PC)}\label{sec:rlm-PC}
In our second approach, we approximate $F^{n+1}$ using the predicted solution (obtained by setting $q=1$). This leads to a more easily solvable quadratic equation for $q^{n+1}$ than the RLM-Q method. We refer to this approach as RLM-PC (relaxed Lagrange multiplier with prediction-correction).

In particular, for a general bulk potential $F(\phi)$, we can update 
\beq \label{eq:F-PC}
F^{n+1} = F(\phi^{n+1,*}), \quad F^n=F(\phi^{n,*}), \qquad
\phi^{n+1,*}\coloneqq \phi_1^{n+1} + \phi_2^{n+1},\ \ \phi^{n,*}\coloneqq \phi_1^{n} + \phi_2^{n},
\eeq
in Eq.~\eqref{eq:RLM-CN-step3}. Here $\phi^{n+1,*}$ is the predicted state corresponding to setting $q^{n+1}=1$. The surrogate $F^n=F(\phi^{n,*})$ uses the predicted state $\phi^{n,*}=\phi_1^n+\phi_2^n$ (corresponding to $q=1$), formed from the affine fields of the previous step, which differs from $F(\phi^n)$ unless $q^n=1$. Since $q^n\approx 1$, the two agree to high order. To obtain a quadratic equation with explicit coefficients, we further approximate the inner-product term in~\eqref{eq:RLM-CN-step3} using the predicted increment $\phi^{n+1,*}-\phi^n$. Thus, Eq.~\eqref{eq:RLM-CN-step3} is simplified as
\begin{align}
& \int_\Omega \Big( F(\phi^{n+1,*}) - F(\phi^{n,*}) \Big) \,\diff\bx
+ \alpha \Big[ (q^{n+1})^2 - (q^n)^2 \Big] \notag \\
&\qquad = q^{n+\half} \int_\Omega f(\phibar^{n+\half}) \big(\phi^{n+1,*} - \phi^n\big)\,\diff\bx. \label{eq:RLM-CN-PC-algebra}
\end{align}
Then we need to solve another (but simpler) quadratic equation, which is
 \begin{equation}
 A_2 (q^{n+1})^2+B_2q^{n+1}+C_2=0, \label{eq:RLM-PC-quadratic}
 \end{equation}
 where the coefficients are
 \[
A_2  = \alpha, \quad
B_2  = -\frac{1}{2} \int_\Omega f(\phibar^{n+\half})(\phi^{n+1,*}-\phi^n)\,\diff\bx, \quad
C_2  = \int_\Omega \big(F^{n+1}-F^n \big) \,\diff\bx+q^nB_2-\alpha (q^n)^2,
\]
with $F^{n+1}$ and $F^n$ defined by Eq.~\eqref{eq:F-PC}.

\begin{remark}[Comparison of RLM-Q and RLM-PC]
The key difference between RLM-Q and RLM-PC lies in how $F^{n+1}$ (and the related inner-product term in~\eqref{eq:RLM-CN-step3}) is approximated. In the RLM-Q approach, we replace $F^{n+1}$ by a semi-implicit surrogate that is quadratic in $\phi^{n+1}$, with the remaining factor frozen at the extrapolation $\phibar^{n+1}$. This yields a quadratic equation for $q^{n+1}$ with coefficients that depend on the extrapolated values. In the RLM-PC approach, we first form the predicted state $\phi^{n+1,*}=\phi_1^{n+1}+\phi_2^{n+1}$ (corresponding to $q^{n+1}=1$), evaluate $F^{n+1}=F(\phi^{n+1,*})$, and approximate the inner-product term using the predicted increment $\phi^{n+1,*}-\phi^n$. This leads to the simpler quadratic equation~\eqref{eq:RLM-PC-quadratic}. In this sense, RLM-PC can be viewed as a predictor--correction for $q^{n+1}$, predicting with $q^{n+1}=1$ and then correcting by solving for $q^{n+1}$.
\end{remark}

\section{Numerical Results}\label{sec:numerical}
In this section, we apply the RLM framework to several classical phase-field models with homogeneous Neumann boundary conditions to illustrate the accuracy and efficiency of the proposed schemes. For spatial discretization, we use the discrete cosine transform method, which is second-order accurate and naturally enforces homogeneous Neumann boundary conditions.

\subsection{AC-type Equation with Double-Well Potential}\label{sec:AC-double-well}
The free energy of the system with the double-well potential is given as
\begin{equation} \label{eq:free-energy-AC}
E(\phi) = \int_\Omega \left( \frac{\varepsilon^2}{2} |\nabla \phi|^2 + F(\phi) \right) \,\diff\bx, \quad F(\phi) = \frac{1}{4}(\phi^2-1)^2.
\end{equation}
where $\phi$ is the phase-field variable and $\varepsilon$ is the interfacial thickness parameter. Note that $f(\phi) = F'(\phi) = \phi(\phi^2 - 1)$. 

Based on the RLM reformulation in Section~\ref{sec:rlm-methods}, the Allen--Cahn-type equation with the double-well potential is given by
\begin{align}
& \partial_t \phi = -M \mu, \label{eq:AC-double-well-phi} \\
& \mu = -\varepsilon^2 \Delta \phi + q(t) \phi(\phi^2 - 1), \label{eq:AC-double-well-mu} \\
& \frac{d}{dt} \int_\Omega F(\phi) \,\diff\bx + \alpha \,\frac{d \big(q(t)\big)^2 }{dt}= \int_\Omega q(t) \phi(\phi^2 - 1) \partial_t \phi \,\diff\bx, \label{eq:AC-double-well-q}
\end{align}
where $M > 0$ is the mobility constant and $q(0)=1$.

For the RLM-Q scheme, we set $F^{n+1} = \frac{1}{4}(\phi^{n+1} \phibar^{n+1} - 1)^2$. We assess temporal convergence, energy accuracy, and computational efficiency.

In the first example, we verify the convergence rates of the proposed RLM-Q and RLM-PC schemes for the Allen--Cahn equation. We set the domain as $\Omega=[0,1]^2$ and use the initial condition
\[
\phi_0(x,y)=\cos(2\pi x)\cos(2\pi y).
\]
We take the interface thickness $\varepsilon=10^{-2}$, $\alpha=1\times 10^{-2}$, and the mobility coefficient $M=1$.

For the temporal convergence test, we compute the Cauchy difference $\phi_k-\phi_{k-1}$ for $k=1,2,\ldots,7$ in the $L^2$ norm at final time $T_{\mathrm{final}}=0.1$ with $\Delta t=0.1/2^{k-1}$ and spatial step size $h=1/256$. We compute the convergence order by
$
\text{order}=\log_2 \frac{L^2(k-1)}{L^2(k)}$, where $L^2(k):=\|\phi_k-\phi_{k-1}\|_2$ with the discrete norm $\|v\|_2=(h^2\sum_{i,j}v_{ij}^2)^{1/2}.$ The same definition applies to the Cahn--Hilliard convergence test in Section~\ref{sec:CH-double-well}.

We perform convergence tests for the RLM-Q-BDF1, RLM-Q-CN, and RLM-PC-CN schemes. The results are summarized in Table~\ref{tab:CRofRLM-AC}.
From Table~\ref{tab:CRofRLM-AC}, we observe that the RLM-Q-BDF1 scheme achieves first-order temporal accuracy, as expected for a first-order BDF scheme. Moreover, both the RLM-Q-CN and RLM-PC-CN schemes achieve approximately second-order temporal accuracy, consistent with Crank--Nicolson-type schemes.

\begin{table}[H]
\centering
\caption{Temporal convergence test for the Allen--Cahn equation with the double-well potential: Cauchy-difference $L^2$ errors and observed convergence orders ($\varepsilon=10^{-2}$, $\alpha=10^{-2}$, $M=1$, $h=1/256$, $T_{\mathrm{final}}=0.1$).}
\begin{subtable}{0.48\textwidth}
\centering
\subcaption{RLM-Q-BDF1 scheme}
\begin{tabular}{c c r c}
\toprule
Coarse $\Delta t$ & Fine $\Delta t$ & $L^2$ error & Order \\
\midrule
0.1 & $0.1/2$ & $1.04\times 10^{-2}$ & -- \\
$0.1/2$ & $0.1/2^2$ & $5.23\times 10^{-3}$ & 0.985 \\
$0.1/2^2$ & $0.1/2^3$ & $2.63\times 10^{-3}$ & 0.992 \\
$0.1/2^3$ & $0.1/2^4$ & $1.32\times 10^{-3}$ & 0.996 \\
$0.1/2^4$ & $0.1/2^5$ & $6.60\times 10^{-4}$ & 0.998 \\
$0.1/2^5$ & $0.1/2^6$ & $3.30\times 10^{-4}$ & 0.999 \\
\bottomrule
\end{tabular}
\label{tab:CRofRLM-Q-BDF1_AC}
\end{subtable}

\vspace{1em}
\begin{minipage}{0.48\textwidth}
\begin{subtable}{\textwidth}
\centering
\subcaption{RLM-Q-CN scheme}
\begin{tabular}{c c r c}
\toprule
Coarse $\Delta t$ & Fine $\Delta t$ & $L^2$ error & Order \\
\midrule
0.1 & $0.1/2$ & $1.10\times 10^{-3}$ & -- \\
$0.1/2$ & $0.1/2^2$ & $2.91\times 10^{-4}$ & 1.92 \\
$0.1/2^2$ & $0.1/2^3$ & $7.48\times 10^{-5}$ & 1.96 \\
$0.1/2^3$ & $0.1/2^4$ & $1.90\times 10^{-5}$ & 1.98 \\
$0.1/2^4$ & $0.1/2^5$ & $4.78\times 10^{-6}$ & 1.99 \\
$0.1/2^5$ & $0.1/2^6$ & $1.20\times 10^{-6}$ & 1.99 \\
\bottomrule
\end{tabular}
\label{tab:CRofRLM-Q_AC}
\end{subtable}
\end{minipage}
\begin{minipage}{0.48\textwidth}
\begin{subtable}{\textwidth}
\centering
\subcaption{RLM-PC-CN scheme}
\begin{tabular}{c c r c}
\toprule
Coarse $\Delta t$ & Fine $\Delta t$ & $L^2$ error & Order \\
\midrule
0.1 & $0.1/2$ & $6.99\times 10^{-4}$ & -- \\
$0.1/2$ & $0.1/2^2$ & $1.80\times 10^{-4}$ & 1.96 \\
$0.1/2^2$ & $0.1/2^3$ & $4.57\times 10^{-5}$ & 1.98 \\
$0.1/2^3$ & $0.1/2^4$ & $1.15\times 10^{-5}$ & 1.99 \\
$0.1/2^4$ & $0.1/2^5$ & $2.89\times 10^{-6}$ & 1.99 \\
$0.1/2^5$ & $0.1/2^6$ & $7.23\times 10^{-7}$ & 2.00 \\
\bottomrule
\end{tabular}
\label{tab:CRofRLM-PC-CN_AC}
\end{subtable}
\end{minipage}
\label{tab:CRofRLM-AC}
\end{table}

Next, we present dynamic simulations of the Allen--Cahn equation using the proposed RLM-CN schemes and compare them with the SAV-CN scheme. We set the computational domain as $\Omega=[0,1]^2$, the interface thickness $\varepsilon=10^{-2}$, the time step size $\Delta t=10^{-3}$, the mobility coefficient $M=1$, the constant $\Csav=1\times 10^3$ in the SAV scheme, and the spatial step size $h=1/128$. The initial condition is given by
\begin{equation}
\phi_0(x,y) = 
\tanh\left(\frac{0.2-r}{\delta_0}\right), \quad r=\sqrt{(x-0.5)^2+(y-0.5)^2}, \quad \delta_0=0.01.
\end{equation}

We apply the SAV-CN, RLM-Q-CN, and RLM-PC-CN schemes and compare snapshots of the phase variable at $T=24, 120, 192, 240$. Figure~\ref{fig:AC-snapshots-CN} summarizes the results. The three schemes produce visually indistinguishable snapshots at these times.

\begin{figure}[H]
\centering
\renewcommand{\arraystretch}{0.9}
\setlength{\tabcolsep}{2pt}
\begin{tabular}{c c c c c}
\toprule
{{\footnotesize Method}} & $T=24$ & $T=120$ & $T=192$ & $T=240$ \\
\midrule
{\footnotesize SAV} &
\snapshot{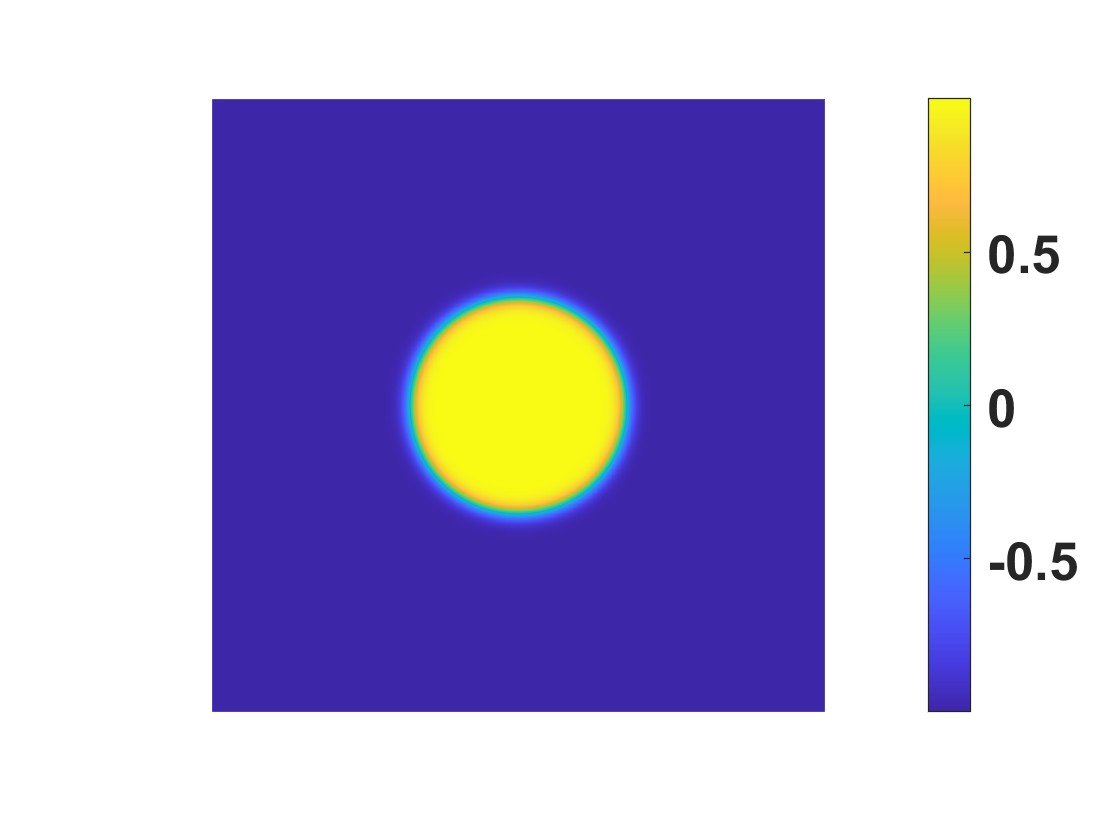} &
\snapshot{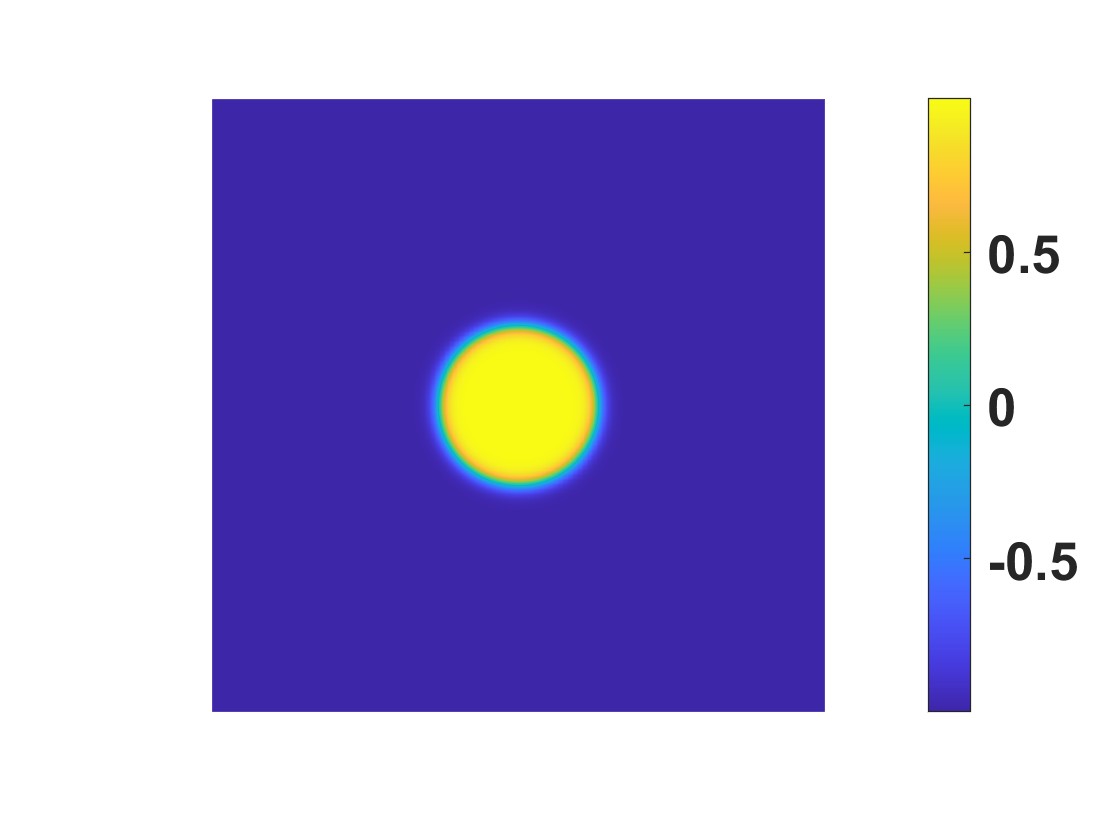} &
\snapshot{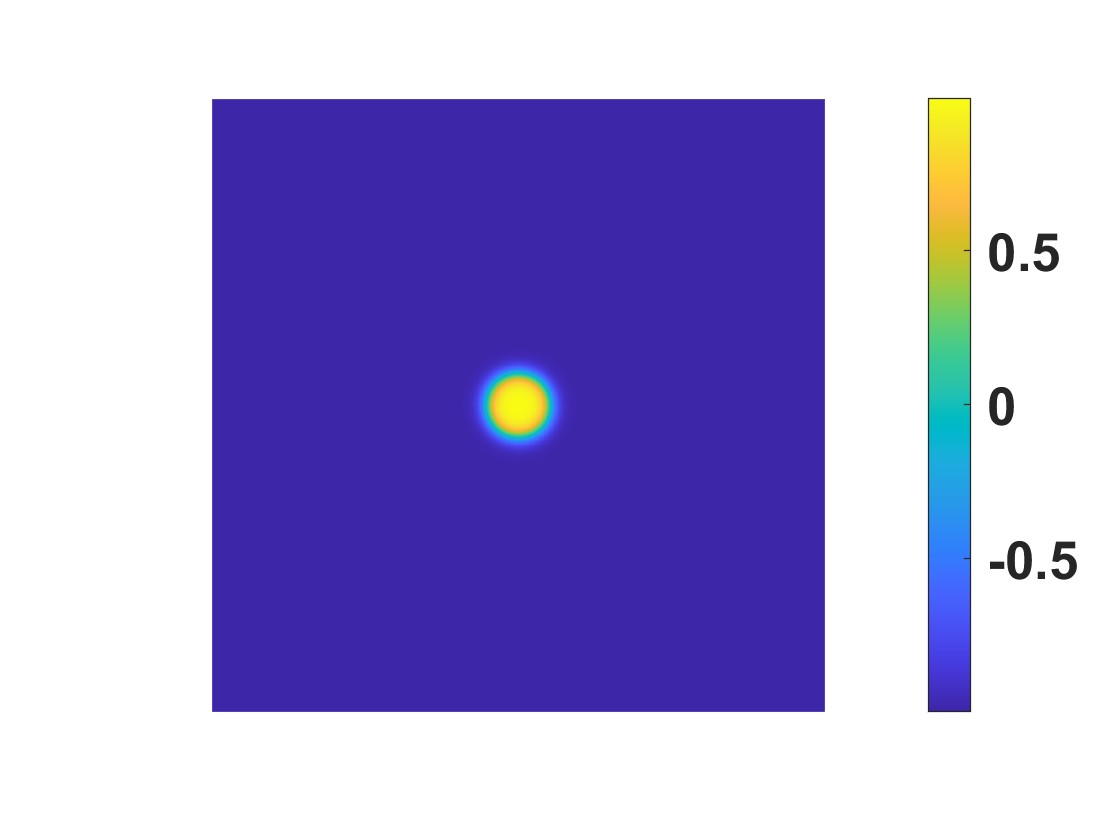} &
\snapshot{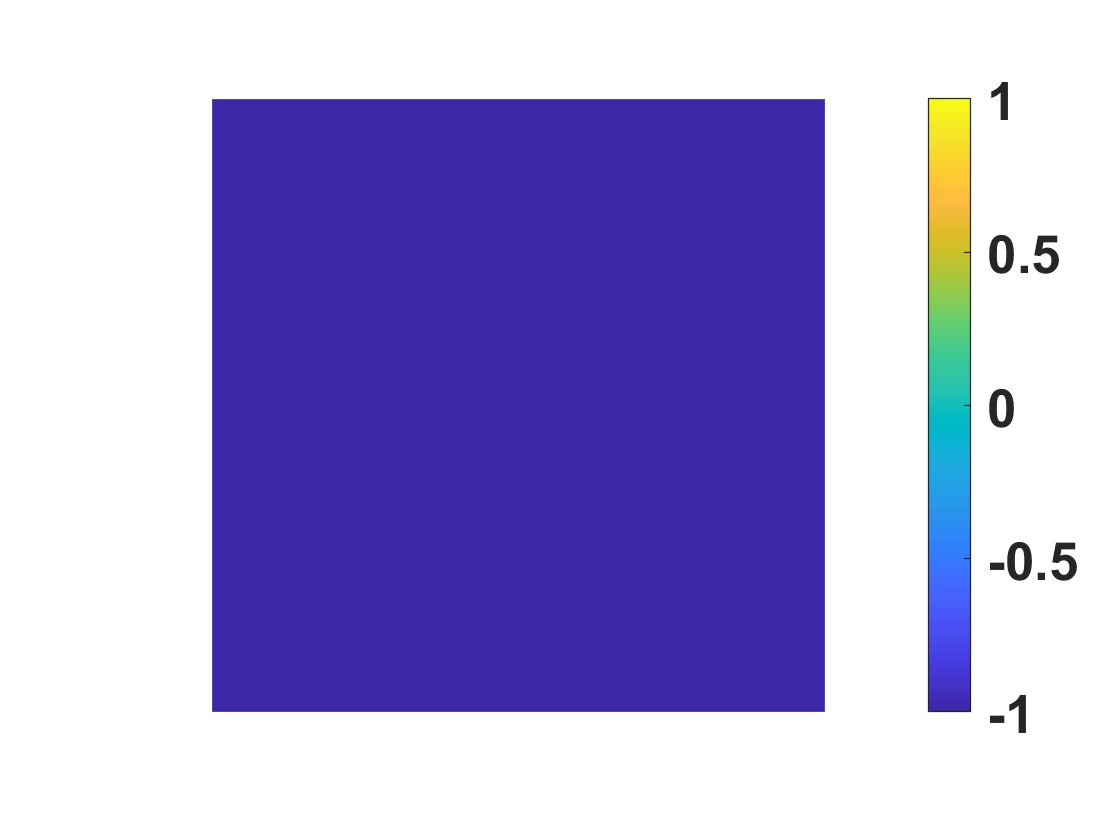} \\
\midrule
{\footnotesize RLM-Q} &
\snapshot{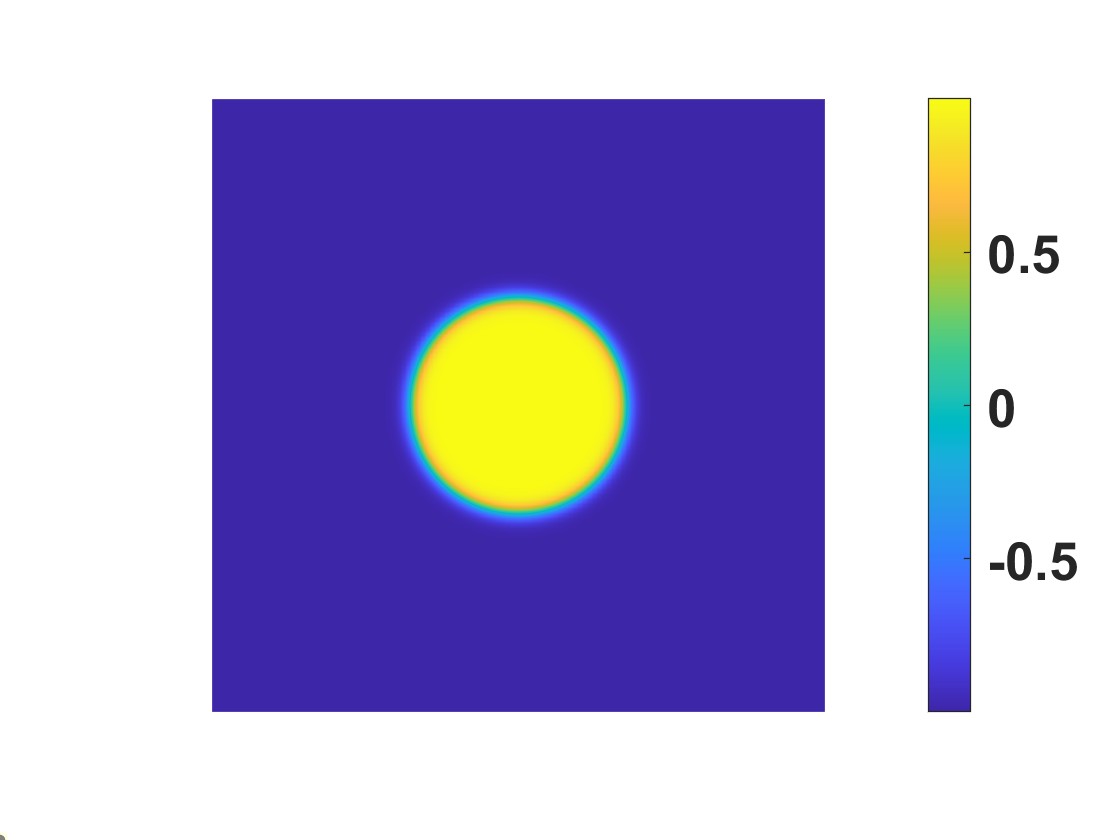} &
\snapshot{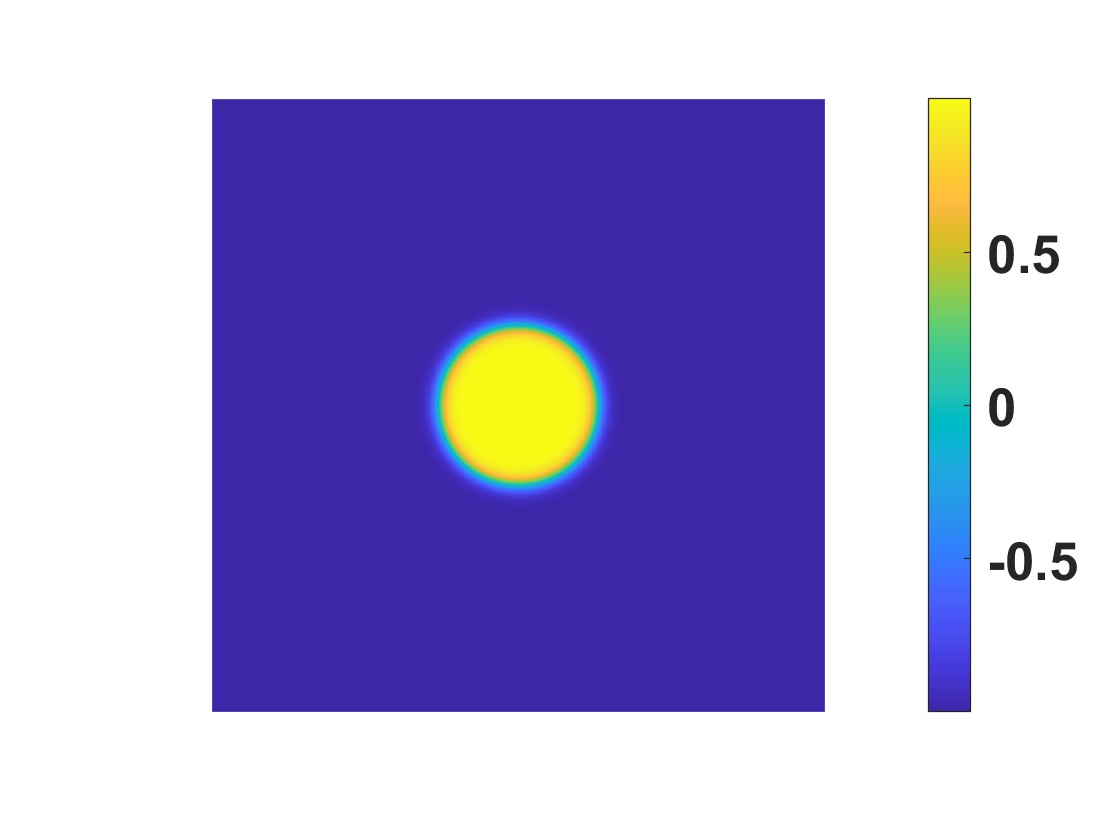} &
\snapshot{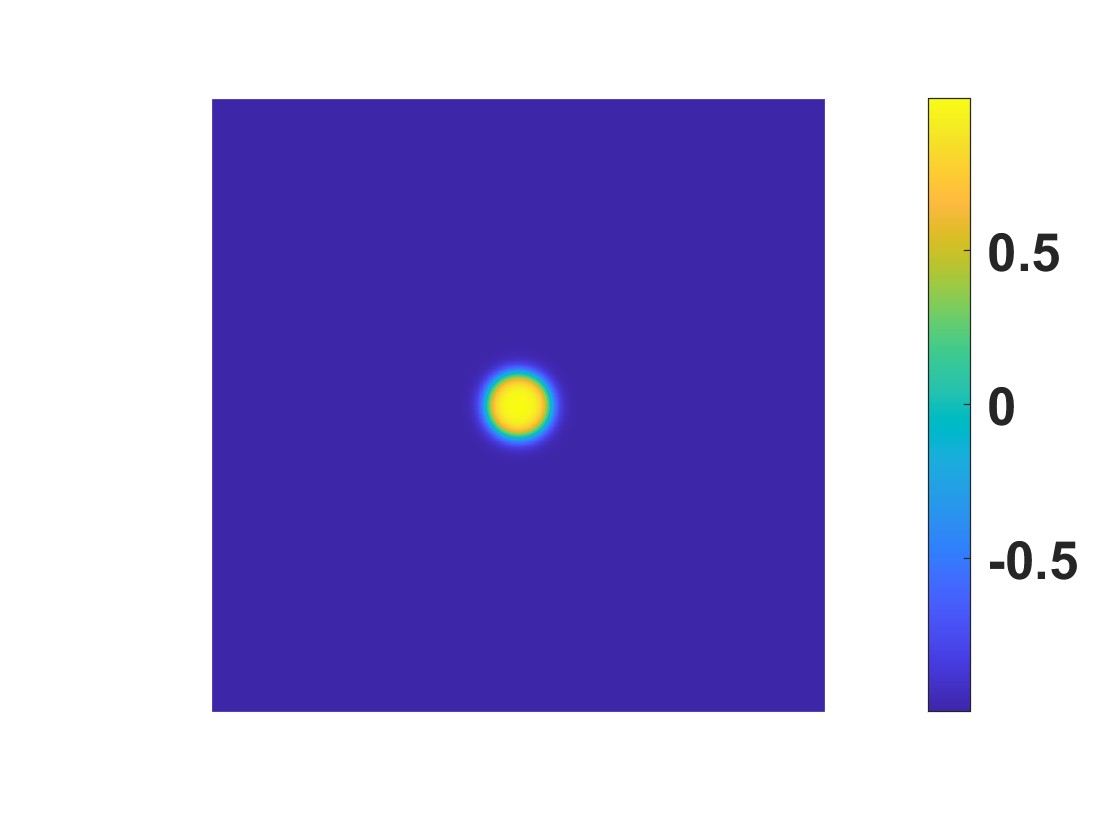} &
\snapshot{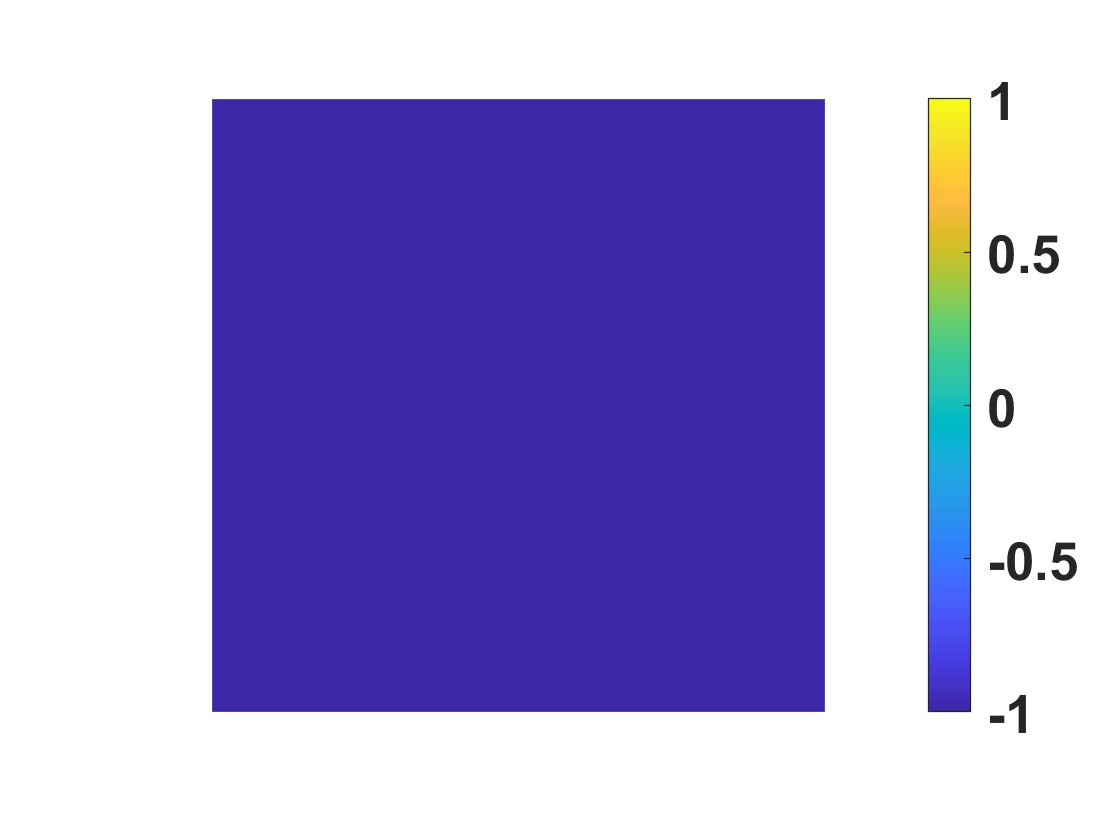} \\
\midrule
{\footnotesize RLM-PC} &
\snapshot{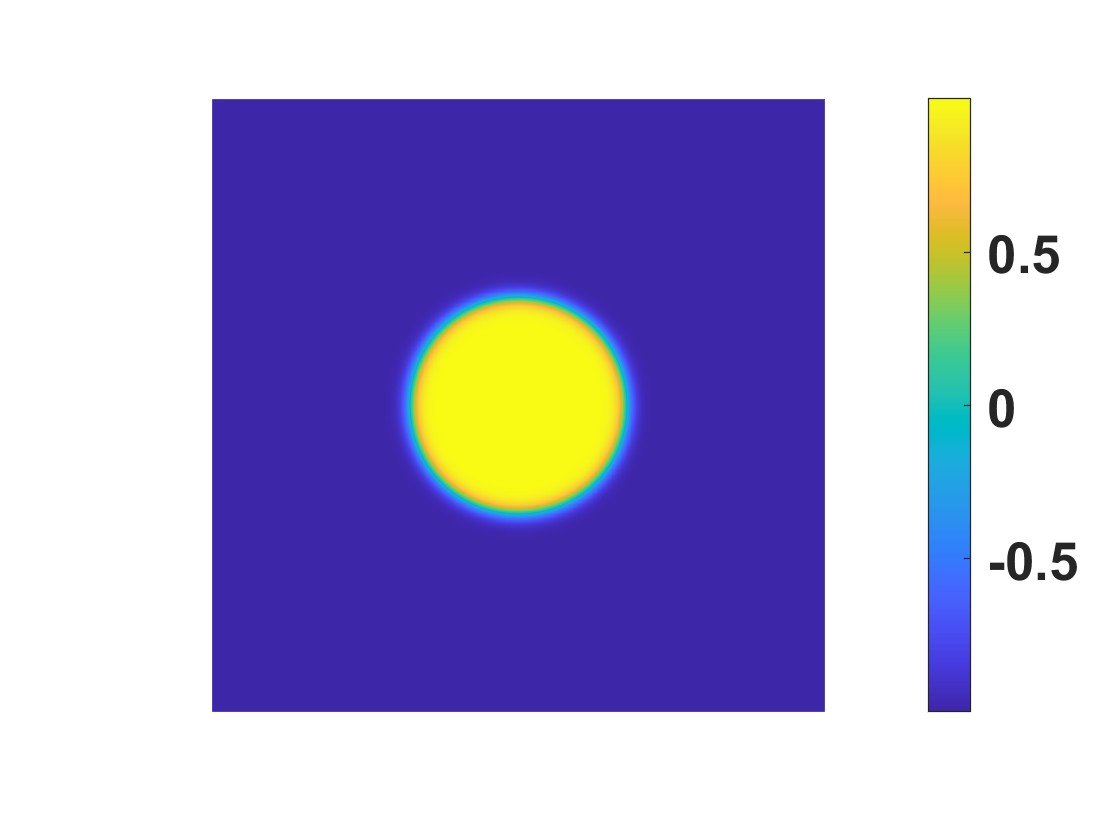} &
\snapshot{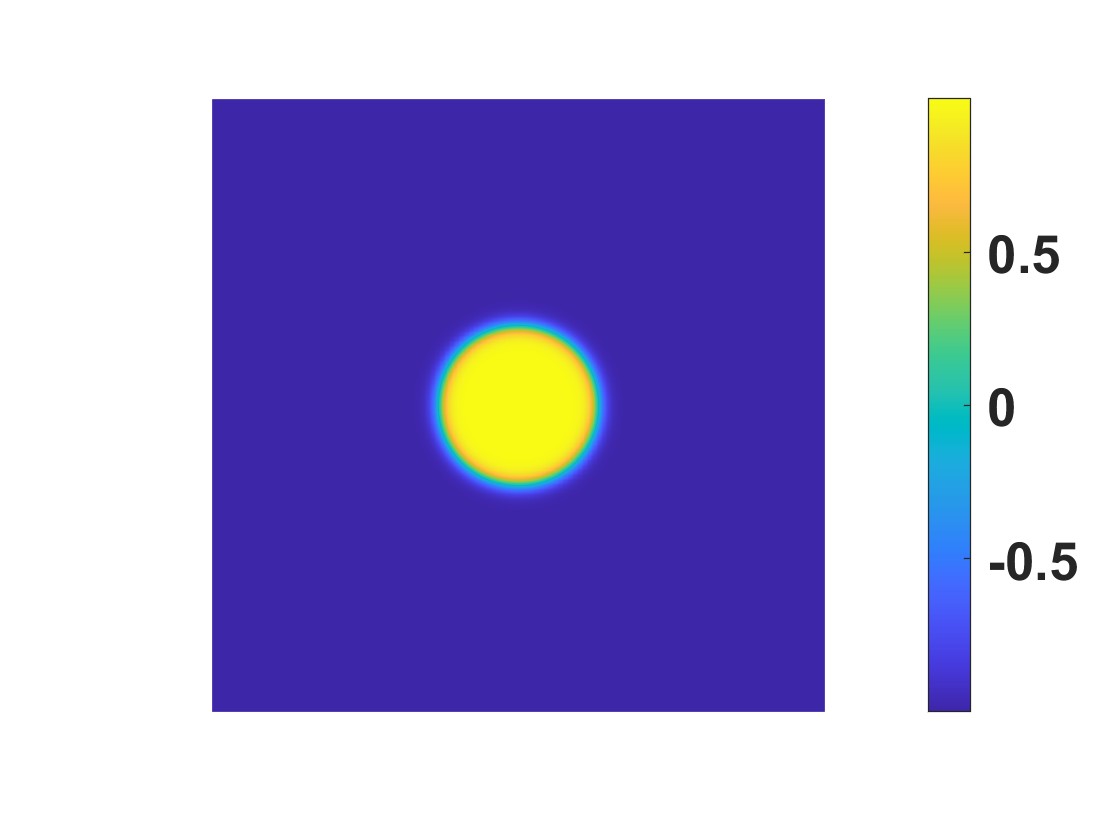} &
\snapshot{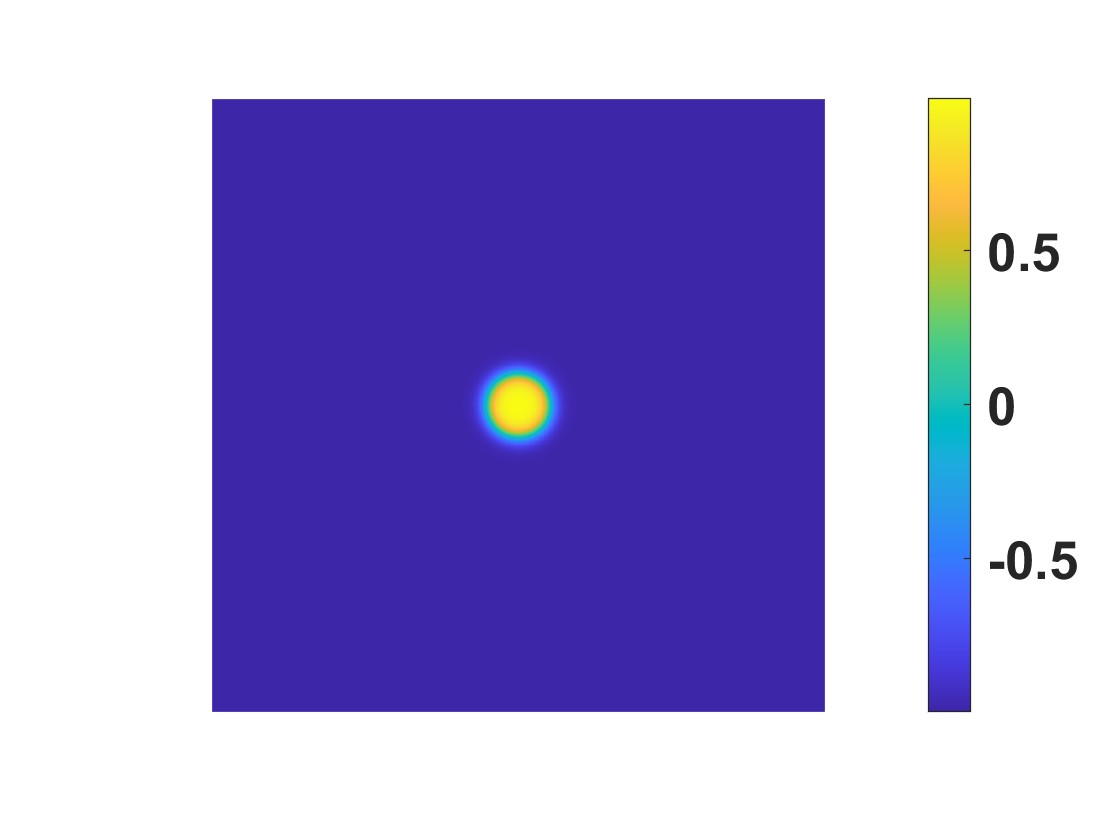} &
\snapshot{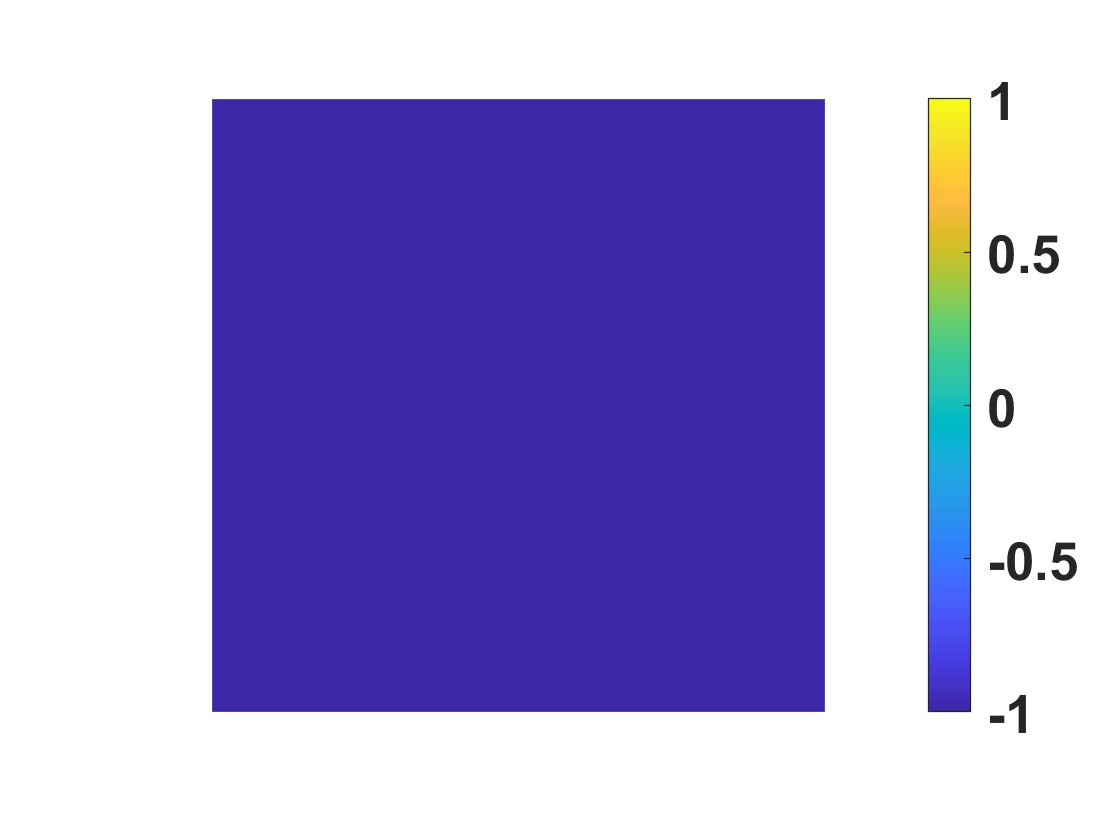} \\
\bottomrule
\end{tabular}
\caption{Snapshots of the phase variable $\phi$ for the Allen--Cahn equation at $T=24, 120, 192, 240$. Each row corresponds to a different scheme: SAV-CN (top), RLM-Q-CN with $\alpha=10^{-2}$ (middle), and RLM-PC-CN with $\alpha=10^{-2}$ (bottom). Each column shows results at the same time point.}
\label{fig:AC-snapshots-CN}
\end{figure}

We conduct quantitative comparisons between these schemes, with the results summarized in Figure~\ref{fig:AC_Energy_alpha_RLM}.
In Figure~\ref{fig:AC_Energy_alpha_RLM}(\subref{subfig:AC_energy_a}), we observe that the relaxed original energies of RLM-Q-CN and RLM-PC-CN (denoted $E_{\text{RLM}}$ in the numerical results) with $\alpha=10^{-2}$ are almost identical to those of SAV-CN ( denoted $E_{\text{SAV}}$, which is the benchmark).  The time evolutions of $q$ for four cases are shown in Figure~\ref{fig:AC_Energy_alpha_RLM}(\subref{subfig:AC_energy_b}). We find that decreasing $\alpha$ increases the deviation of $q$ from 1. The results indicate that $q$ is very close to 1 with $\alpha=10^{-2}$, while a smaller $\alpha$ may induce a larger deviation of $q$ from 1 for both RLM schemes. When $\alpha=10^{-4}$, the quadratic equation for RLM-Q-CN does not have a real solution, while RLM-PC-CN remains solvable, indicating that RLM-PC-CN is more robust.

To quantify the differences among RLM, SAV, and original energies, we define two classes of energy differences. The first class, $\tilde{E}  -E_{\text{SAV}}$, compares the relaxed original energies ($\tilde{E}$) with the benchmark energy ($E_{\text{SAV}}$). The second class, $\tilde{E} - E=\alpha (q^2-1)$, compares the relaxed original energy $\tilde{E} $ (as defined in Section~\ref{sec:rlm-methods}) with the original energy $E$. Figure~\ref{fig:AC_Energy_alpha_RLM}(\subref{subfig:AC_energy_c}) indicates that the relaxed original energies $\tilde{E}$ are closer to the benchmark energy with a larger $\alpha$. Meanwhile, a larger $\alpha$ also promotes the accuracy of the original energy $\alpha (q^2-1)$ in Figure~\ref{fig:AC_Energy_alpha_RLM}(\subref{subfig:AC_energy_d}). This is because a larger $\alpha$ improves the stability of $q$. For example, $q$ is very close to 1 when $\alpha=10^{-2}$ in Figure~\ref{fig:AC_Energy_alpha_RLM}(\subref{subfig:AC_energy_b}).

\begin{figure}[H]
\centering
\subfig{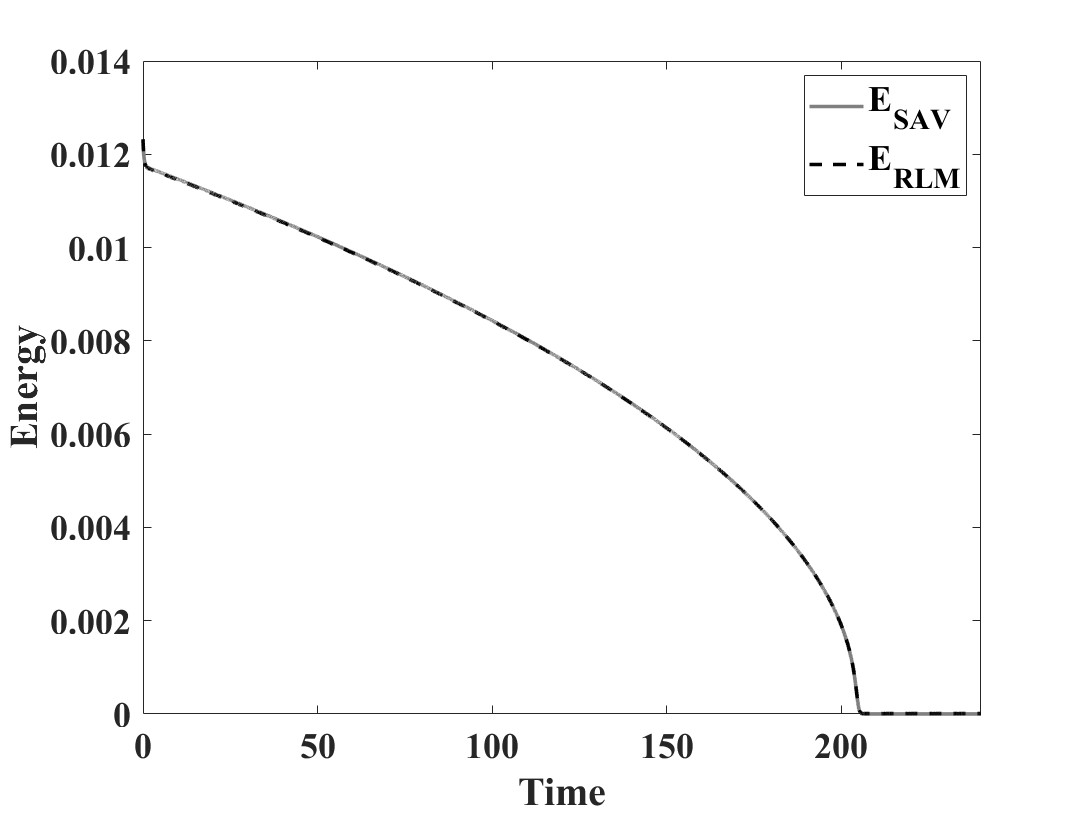}{Comparison of the energies}{subfig:AC_energy_a}
\subfig{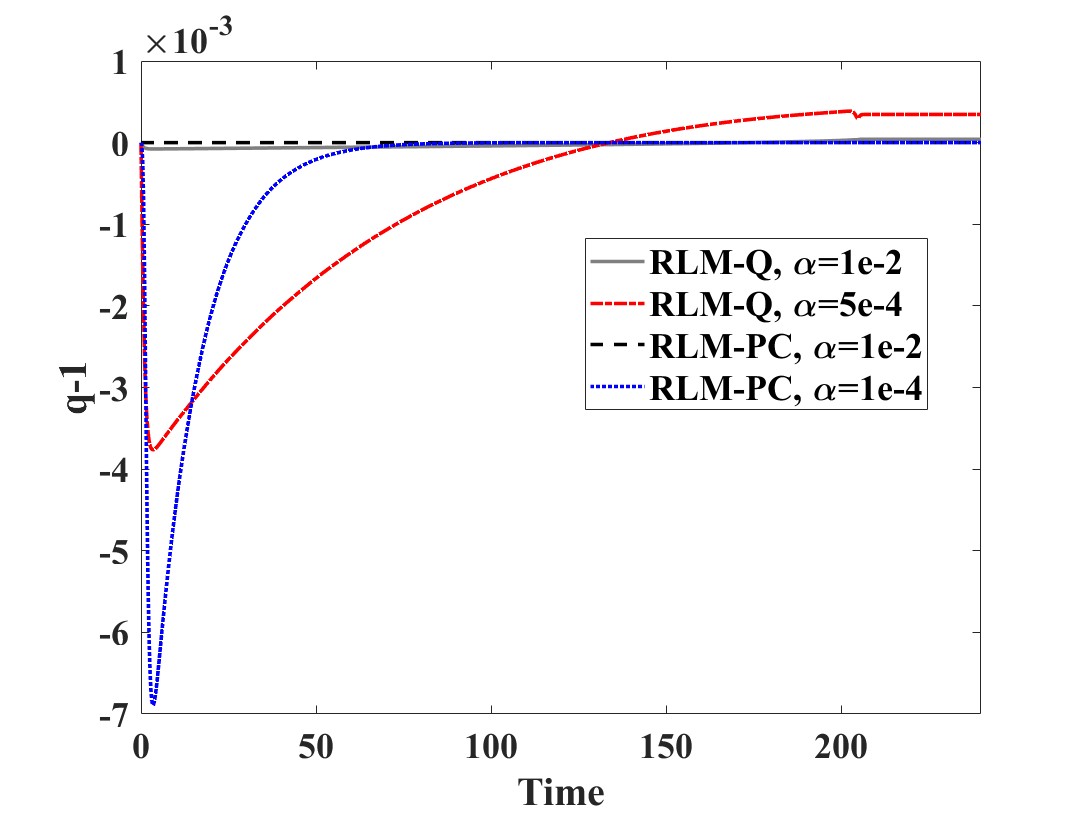}{Scaling factor $q$}{subfig:AC_energy_b}
\subfig{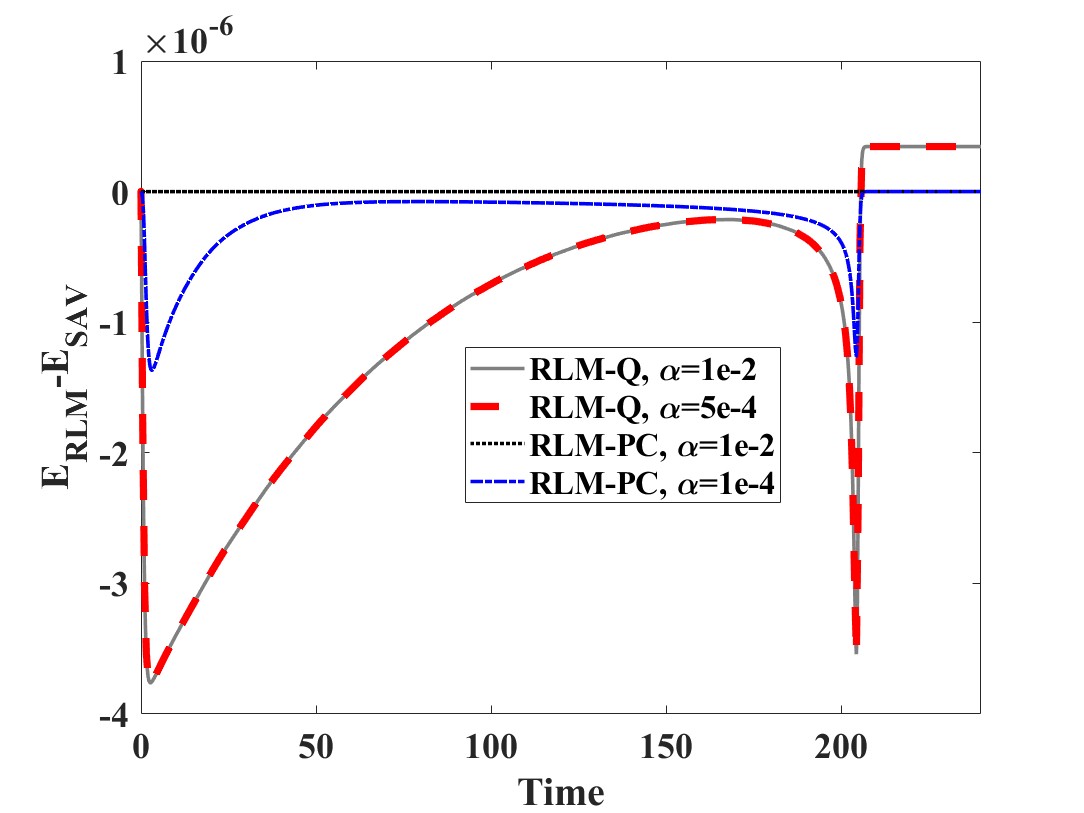}{First class of energy differences}{subfig:AC_energy_c}
\subfig{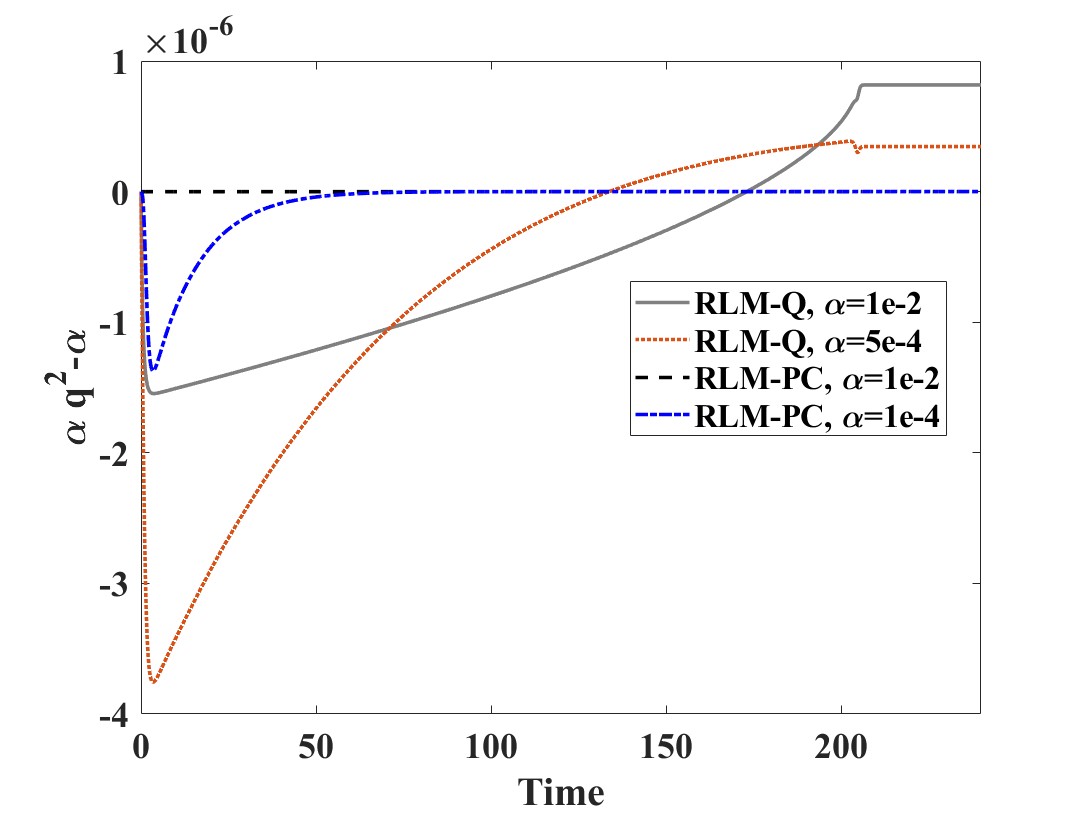}{Second class of energy differences}{subfig:AC_energy_d}
\caption{Time evolution of the energy, scaling factor $q$, and energy differences for the Allen--Cahn equation. (\subref{subfig:AC_energy_a})~Energies of SAV-CN, RLM-Q-CN (all with $\alpha=10^{-2}$). (\subref{subfig:AC_energy_b})~Scaling factor $q$ for RLM-Q ($\alpha=10^{-2}$, $5\times10^{-4}$) and RLM-PC ($\alpha=10^{-2}$, $1\times10^{-4}$). (\subref{subfig:AC_energy_c})~First class of energy differences $\tilde{E}-E_{\text{SAV}}$. (\subref{subfig:AC_energy_d})~Second class of energy differences $\alpha (q^2-1)$.}
\label{fig:AC_Energy_alpha_RLM}
\end{figure}

To further investigate the robustness of the RLM methods, we next examine their accuracy with sufficiently large values of $\alpha$. The relaxed original energy ($\tilde{E}$) evolutions of the RLM-Q-CN and RLM-PC-CN schemes with $\alpha=10^{-2}$ and $\alpha=10^3$ are shown in panels (\subref{subfig:AC_alpha_a}) and (\subref{subfig:AC_alpha_b}) of Figure~\ref{fig:AC_alpha_RLM}. For both RLM-Q-CN and RLM-PC-CN, the energy dissipation rates with $\alpha=10^{-2}$ and $\alpha=10^3$ are identical. Figure~\ref{fig:AC_alpha_RLM}(\subref{subfig:AC_alpha_c}) and (\subref{subfig:AC_alpha_d}) show that RLM-PC-CN is more robust than RLM-Q-CN with $\alpha=10^{-2}$ and $\alpha=10^{3}$.

\begin{figure}[H]
\centering
\subfig{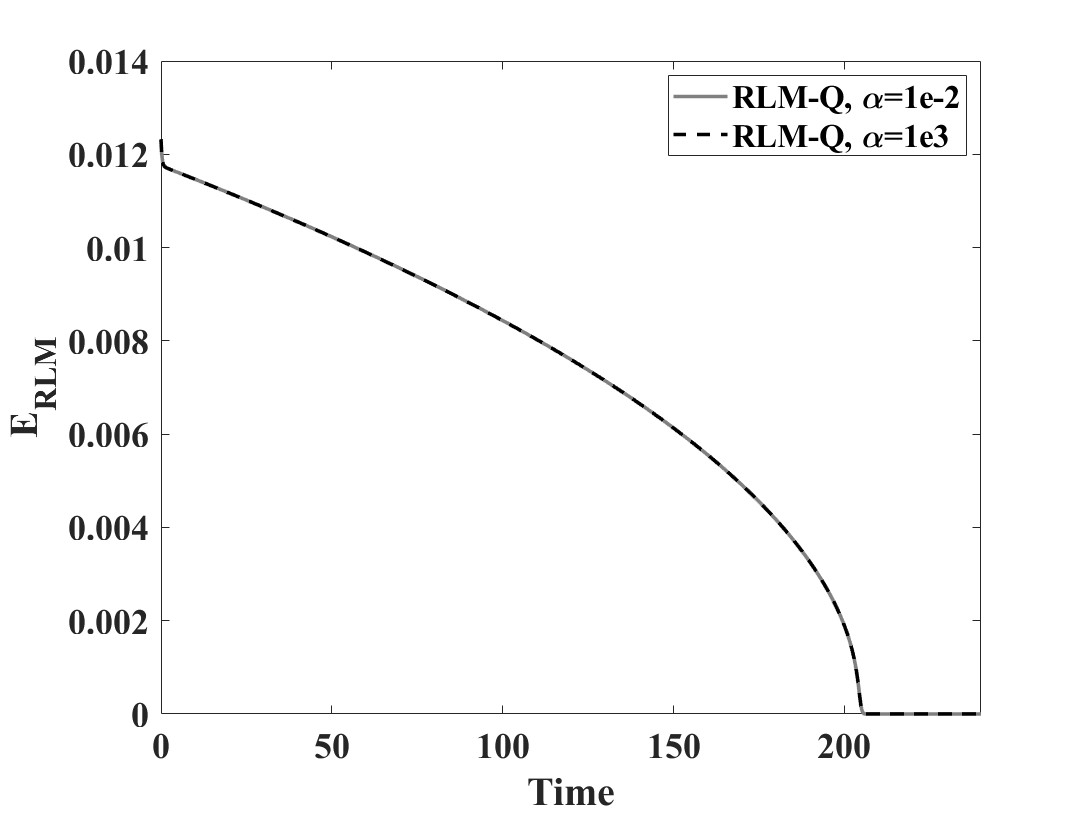}{RLM-Q}{subfig:AC_alpha_a}
\subfig{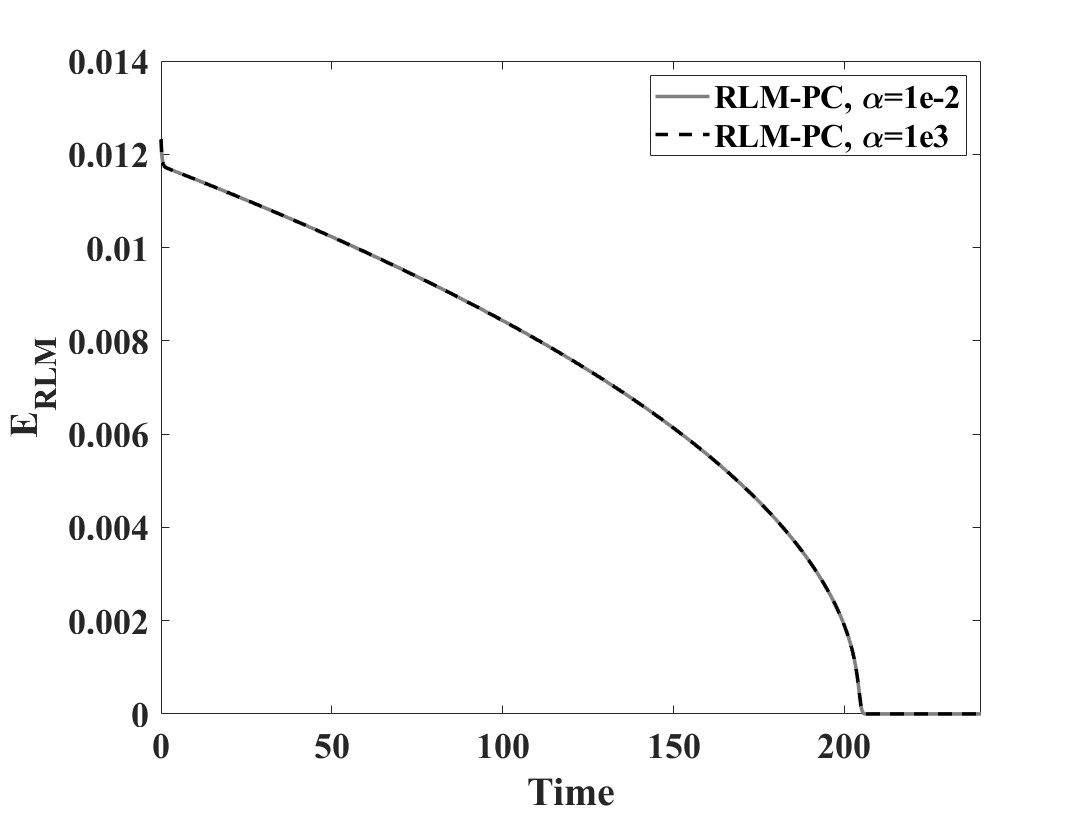}{RLM-PC}{subfig:AC_alpha_b}

\subfig{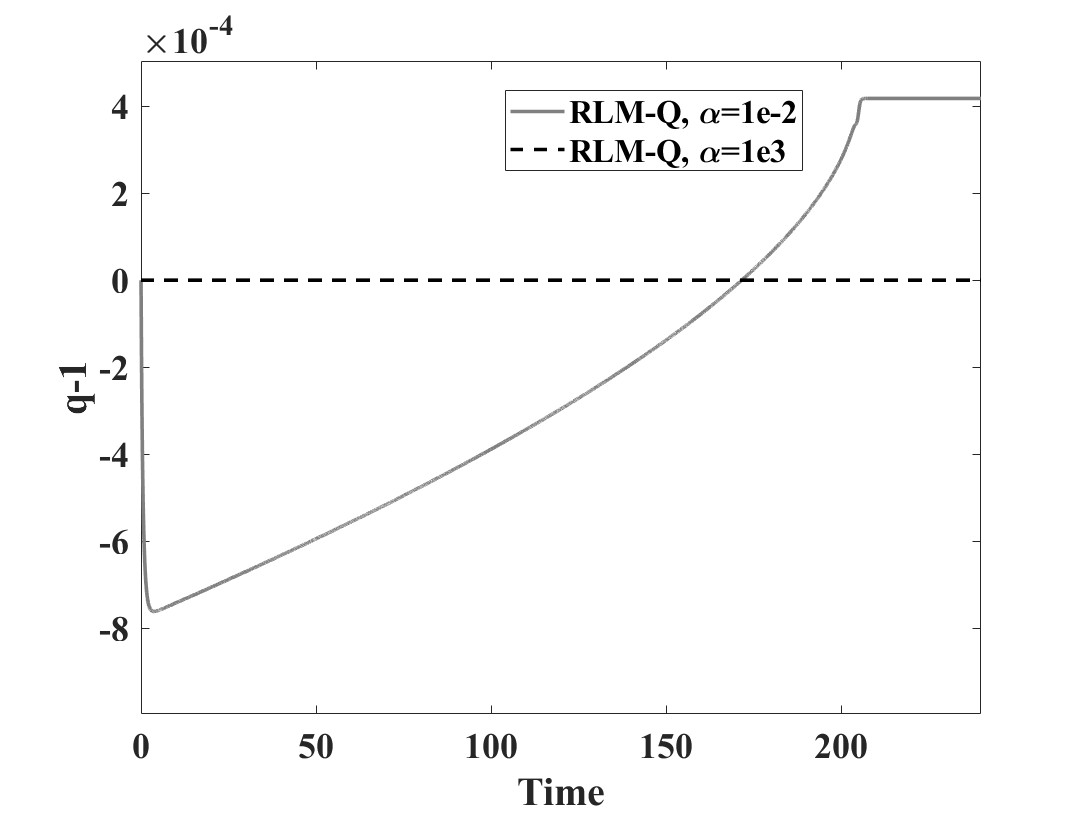}{Scaling factor $q$ ($\alpha=10^{-2}$)}{subfig:AC_alpha_c}
\subfig{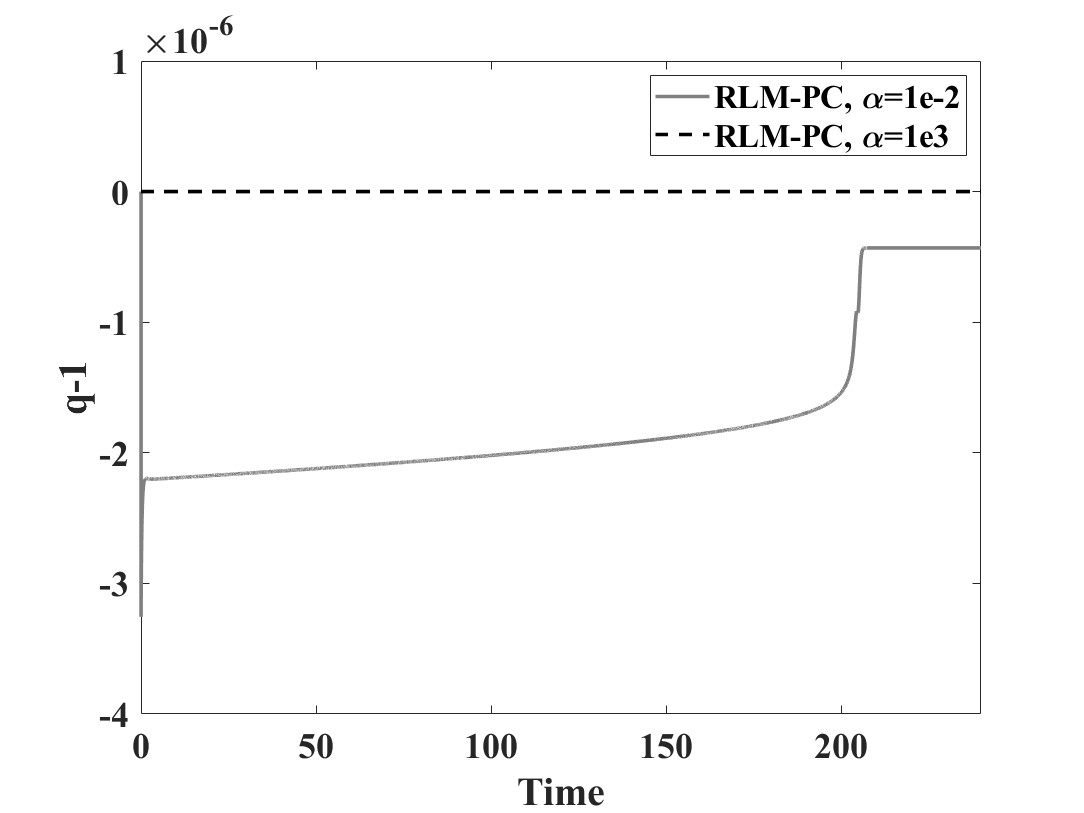}{Scaling factor $q$ ($\alpha=10^{3}$)}{subfig:AC_alpha_d}
\caption{Time evolution of the relaxed original energy $\tilde{E}$ and scaling factor $q$ for the Allen--Cahn equation with $\alpha=10^{-2}$ and $10^{3}$. (\subref{subfig:AC_alpha_a}) RLM-Q: the energy curves for $\alpha=10^{-2}$ coincide with those of  $\alpha=10^3$. (\subref{subfig:AC_alpha_b}) RLM-PC: the energy curves for $\alpha=10^{-2}$ are also the same as those of  $\alpha=10^3$. (\subref{subfig:AC_alpha_c}) and (\subref{subfig:AC_alpha_d}) show the time evolutions of $q$ for the two RLM methods with $\alpha=10^{-2}$ and $\alpha=10^{3}$.}
\label{fig:AC_alpha_RLM}
\end{figure}

\subsection{CH-type Equation with Double-Well Potential}\label{sec:CH-double-well}
Having validated the RLM schemes for the Allen--Cahn equation, we now apply the same framework to the Cahn--Hilliard equation with the same double-well potential to demonstrate the performance of RLM for mass-conserving gradient flows.

The free energy is the same double-well potential as in~\eqref{eq:free-energy-AC}.
Based on the RLM reformulation in Section~\ref{sec:rlm-methods}, the Cahn--Hilliard-type equation with this potential is given by
\begin{align}
& \partial_t \phi = M \Delta \mu, \label{eq:CH-double-well-phi} \\
& \mu = -\varepsilon^2 \Delta \phi + q(t) \phi(\phi^2 - 1), \label{eq:CH-double-well-mu} \\
& \frac{d}{dt} \int_\Omega F(\phi) \,\diff\bx + \alpha \,\frac{d \big(q(t)\big)^2 }{dt}= \int_\Omega q(t) \phi(\phi^2 - 1) \partial_t \phi \,\diff\bx, \label{eq:CH-double-well-q}
\end{align}
where $M > 0$ is the mobility constant and $q(0)=1$.

Using the same RLM-Q approximation as in Section~\ref{sec:AC-double-well}, we set $F^{n+1} = \frac{1}{4}(\phi^{n+1} \phibar^{n+1} - 1)^2$. We assess temporal convergence, energy accuracy, and computational efficiency.

\[
\phi_0(x,y)=\cos(2\pi x)\cos(2\pi y).
\]
We take the interface thickness $\varepsilon=10^{-2}$, $\alpha=1\times 10^{-2}$, and the mobility coefficient $M=5\times 10^{-4}$.

For the temporal convergence test, we compute the Cauchy difference $\phi_k-\phi_{k-1}$ for $k=1,2,\ldots,7$ in the $L^2$ norm at final time $T_{\mathrm{final}}=0.1$ with $\Delta t=0.1/2^{k-1}$ and spatial step size $h=1/256$. We compute the convergence order by
$
\text{order}=\log_2 \frac{L^2(k-1)}{L^2(k)}.
$
We perform convergence tests for the RLM-Q-BDF1, RLM-Q-CN, and RLM-PC-CN schemes. The results are summarized in Table~\ref{tab:CRofRLM-CH}.
From Table~\ref{tab:CRofRLM-CH}, we observe that the RLM-Q-BDF1 scheme achieves first-order temporal accuracy, while both the RLM-Q-CN and RLM-PC-CN schemes achieve approximately second-order temporal accuracy.

\begin{table}[H]
\centering
\caption{Temporal convergence test for the Cahn--Hilliard equation with the double-well potential: Cauchy-difference $L^2$ errors and observed convergence orders ($\varepsilon=10^{-2}$, $\alpha=10^{-2}$, $M=5\times10^{-4}$, $h=1/256$, $T_{\mathrm{final}}=0.1$).}
\begin{subtable}{0.48\textwidth}
\centering
\subcaption{RLM-Q-BDF1 scheme}
\begin{tabular}{c c r c}
\toprule
Coarse $\Delta t$ & Fine $\Delta t$ & $L^2$ error & Order \\
\midrule
0.1 & $0.1/2$ & $1.42\times 10^{-3}$ & -- \\
$0.1/2$ & $0.1/2^2$ & $7.15\times 10^{-4}$ & 0.9981 \\
$0.1/2^2$ & $0.1/2^3$ & $3.58\times 10^{-4}$ & 0.9990 \\
$0.1/2^3$ & $0.1/2^4$ & $1.79\times 10^{-4}$ & 0.9995\\
$0.1/2^4$ & $0.1/2^5$ & $8.95\times 10^{-5}$ & 0.9998 \\
$0.1/2^5$ & $0.1/2^6$ & $4.47\times 10^{-5}$ & 0.9999 \\
\bottomrule
\end{tabular}
\label{tab:CRofRLM-Q-BDF1_CH}
\end{subtable}

\vspace{1em}
\begin{minipage}{0.48\textwidth}
\begin{subtable}{\textwidth}
\centering
\subcaption{RLM-Q-CN scheme}
\begin{tabular}{c c r c}
\toprule
Coarse $\Delta t$ & Fine $\Delta t$ & $L^2$ error & Order \\
\midrule
0.1 & $0.1/2$ & $9.70\times 10^{-5}$ & -- \\
$0.1/2$ & $0.1/2^2$ & $2.52\times 10^{-5}$ & 1.95 \\
$0.1/2^2$ & $0.1/2^3$ & $6.39\times 10^{-6}$ & 1.98 \\
$0.1/2^3$ & $0.1/2^4$ & $1.60\times 10^{-6}$ & 2.00 \\
$0.1/2^4$ & $0.1/2^5$ & $3.98\times 10^{-7}$ & 2.01 \\
$0.1/2^5$ & $0.1/2^6$ & $9.80\times 10^{-8}$ & 2.02 \\
\bottomrule
\end{tabular}
\label{tab:CRofRLM-Q-CN_CH}
\end{subtable}
\end{minipage}
\begin{minipage}{0.48\textwidth}
\begin{subtable}{\textwidth}
\centering
\subcaption{RLM-PC-CN scheme}
\begin{tabular}{c c r c}
\toprule
Coarse $\Delta t$ & Fine $\Delta t$ & $L^2$ error & Order \\
\midrule
0.1 & $0.1/2$ & $5.95\times 10^{-5}$ & -- \\
$0.1/2$ & $0.1/2^2$ & $1.56\times 10^{-5}$ & 1.93 \\
$0.1/2^2$ & $0.1/2^3$ & $3.99\times 10^{-6}$ & 1.97 \\
$0.1/2^3$ & $0.1/2^4$ & $1.01\times 10^{-6}$ & 1.98 \\
$0.1/2^4$ & $0.1/2^5$ & $2.53\times 10^{-7}$ & 1.99 \\
$0.1/2^5$ & $0.1/2^6$ & $6.35\times 10^{-8}$ & 2.00 \\
\bottomrule
\end{tabular}
\label{tab:CRofRLM-PC-CN_CH}
\end{subtable}
\end{minipage}
\label{tab:CRofRLM-CH}
\end{table}

Next, we present dynamic simulations for the Cahn--Hilliard equation. We set the computational domain as $\Omega=[0,1]^2$, the interface thickness $\varepsilon=10^{-2}$, the time step size $\Delta t=10^{-2}$, the mobility coefficient $M=10^{-6}$, and the spatial step size $h=1/128$. The initial condition is given by
\begin{equation}
\phi_0(x,y) = 1- \frac{1}{2}\Big[1-\tanh\left(\frac{0.2-r_1}{\delta_0}\right)\Big]\Big[1-\tanh\left(\frac{0.2-r_2}{\delta_0}\right)\Big]
\end{equation}
where $r_1=\sqrt{(x-0.3)^2+(y-0.5)^2}$, $r_2=\sqrt{(x-0.7)^2+(y-0.5)^2}$, and $\delta_0=0.01$.

Using the same initial and boundary conditions, time step, and spatial step sizes, we observe no discernible difference in the dynamics of the phase variable among the SAV-CN, RLM-Q-CN, and RLM-PC-CN schemes (see Figure~\ref{fig:PFCH}).

\begin{figure}[H]
\centering
\renewcommand{\arraystretch}{0.9}
\setlength{\tabcolsep}{2pt}
\begin{tabular}{c c c c c}
\toprule
{{\footnotesize Method}} & $T=500$ & $T=5000$ & $T=25000$ & $T=50000$ \\
\midrule
{\footnotesize SAV} &
\snapshot{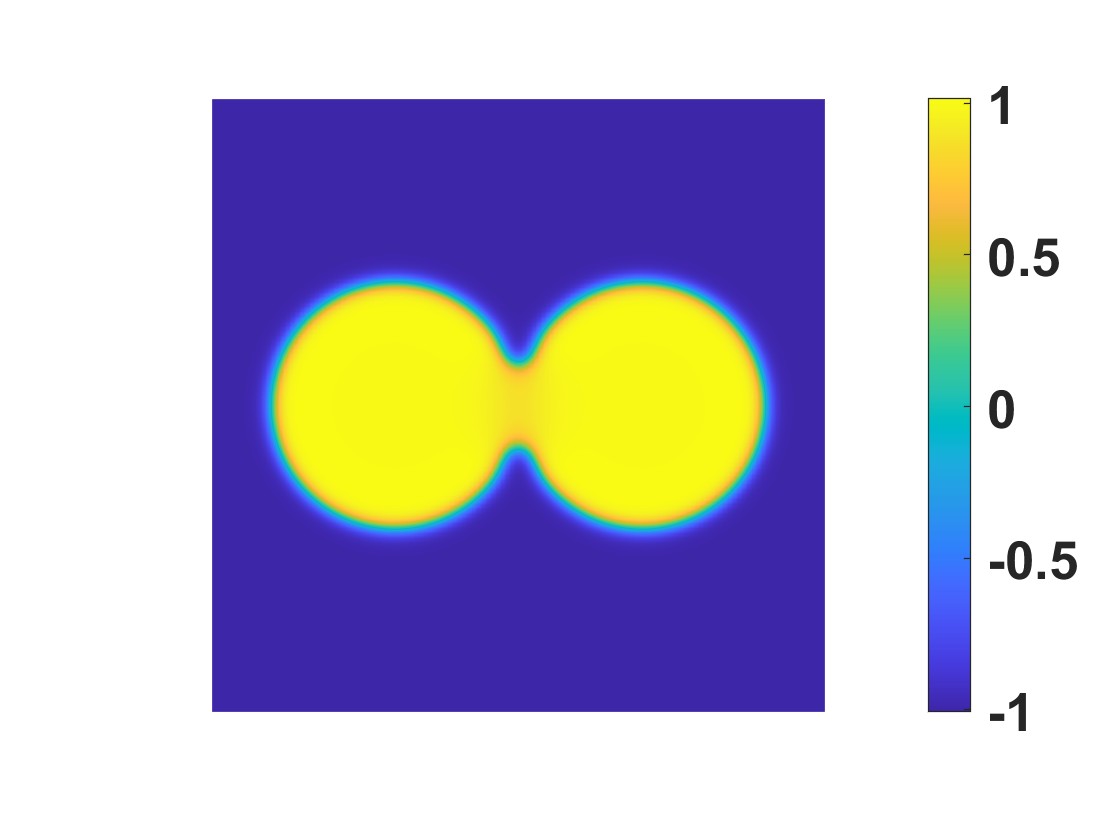} &
\snapshot{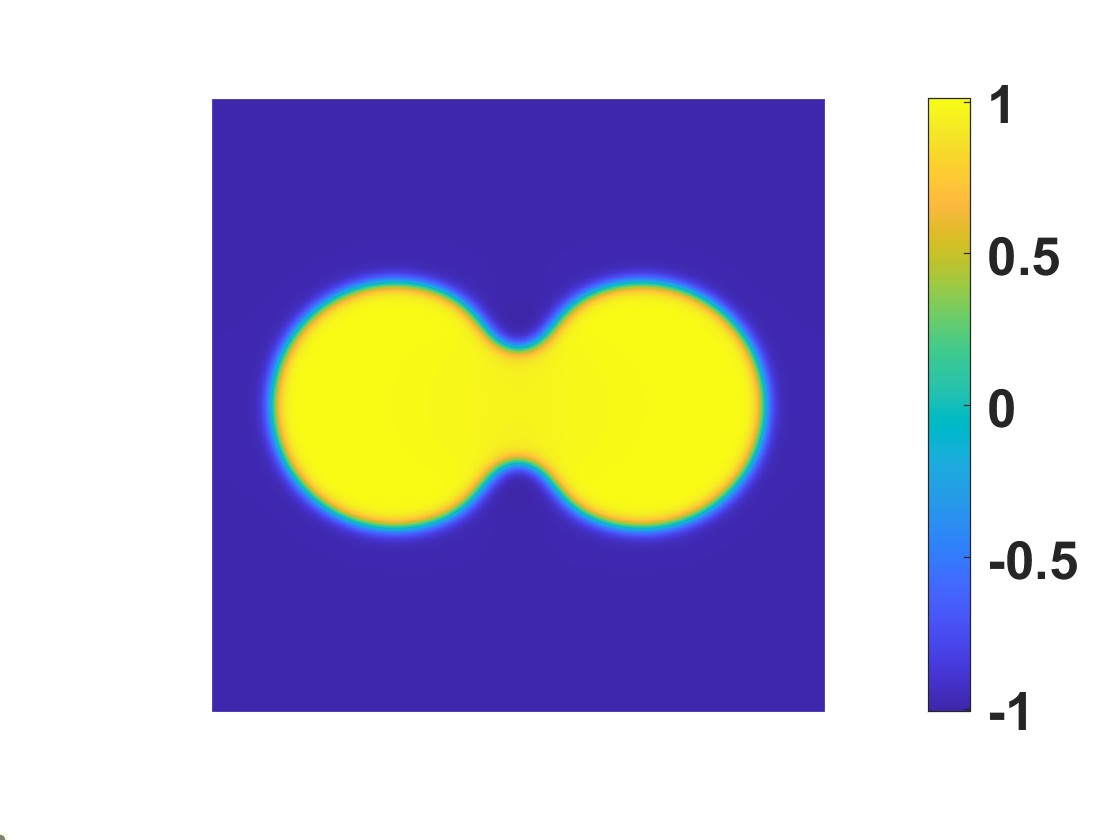} &
\snapshot{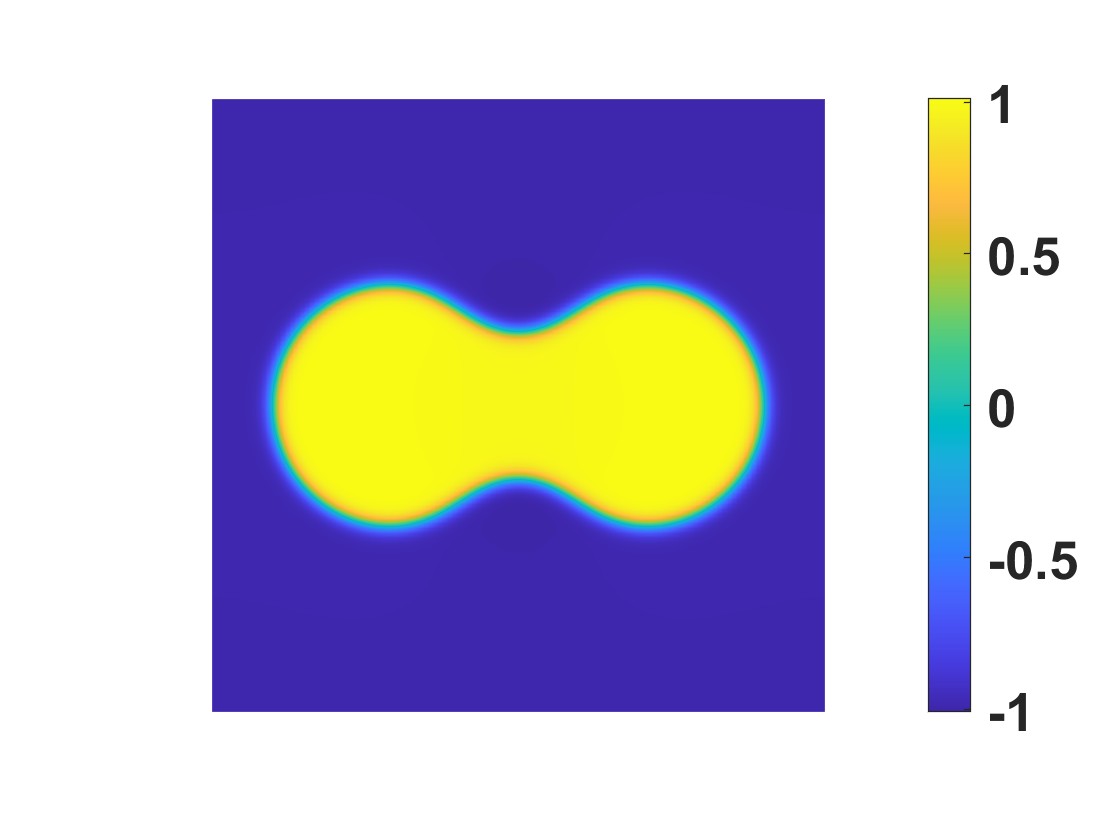} &
\snapshot{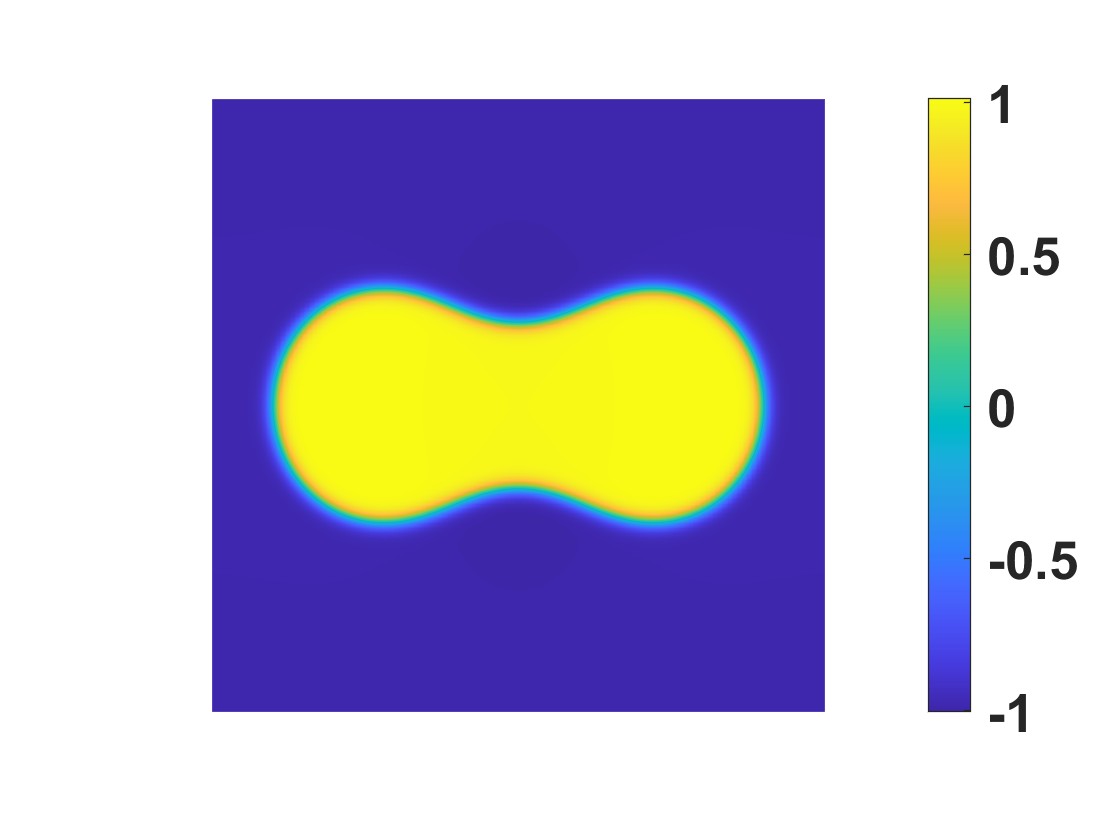} \\
\midrule
{\footnotesize RLM-Q} &
\snapshot{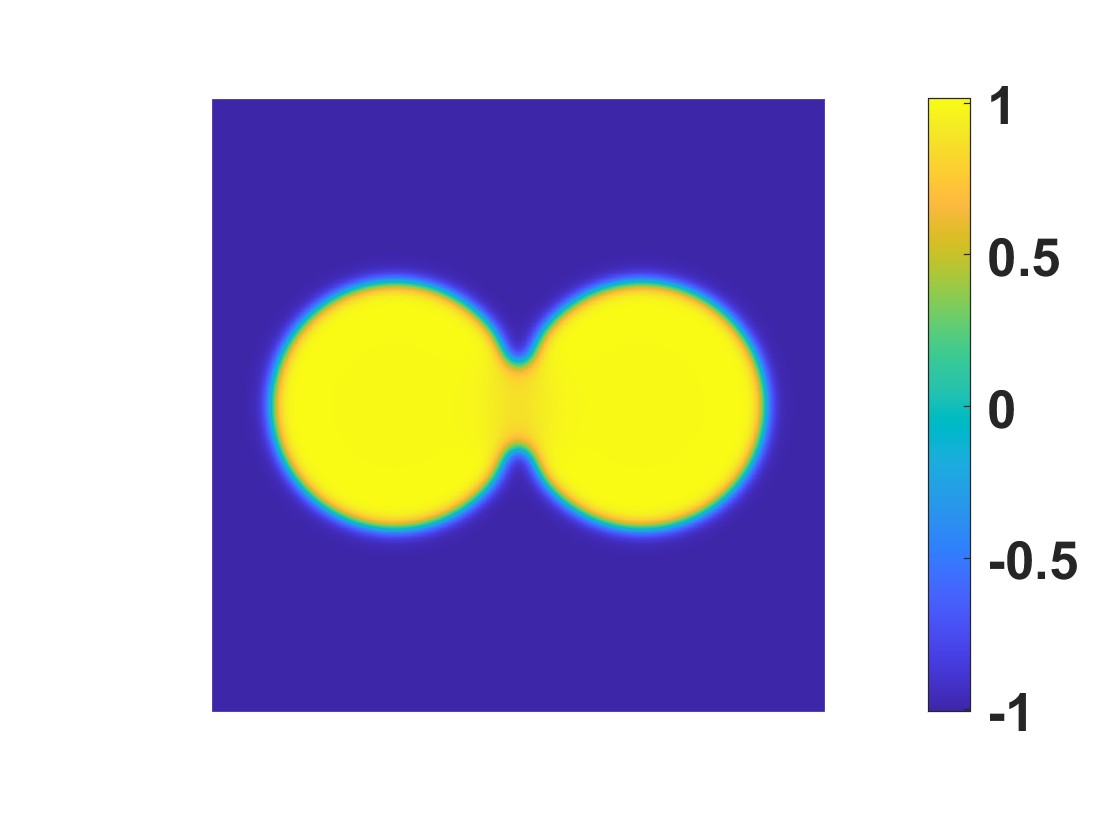} &
\snapshot{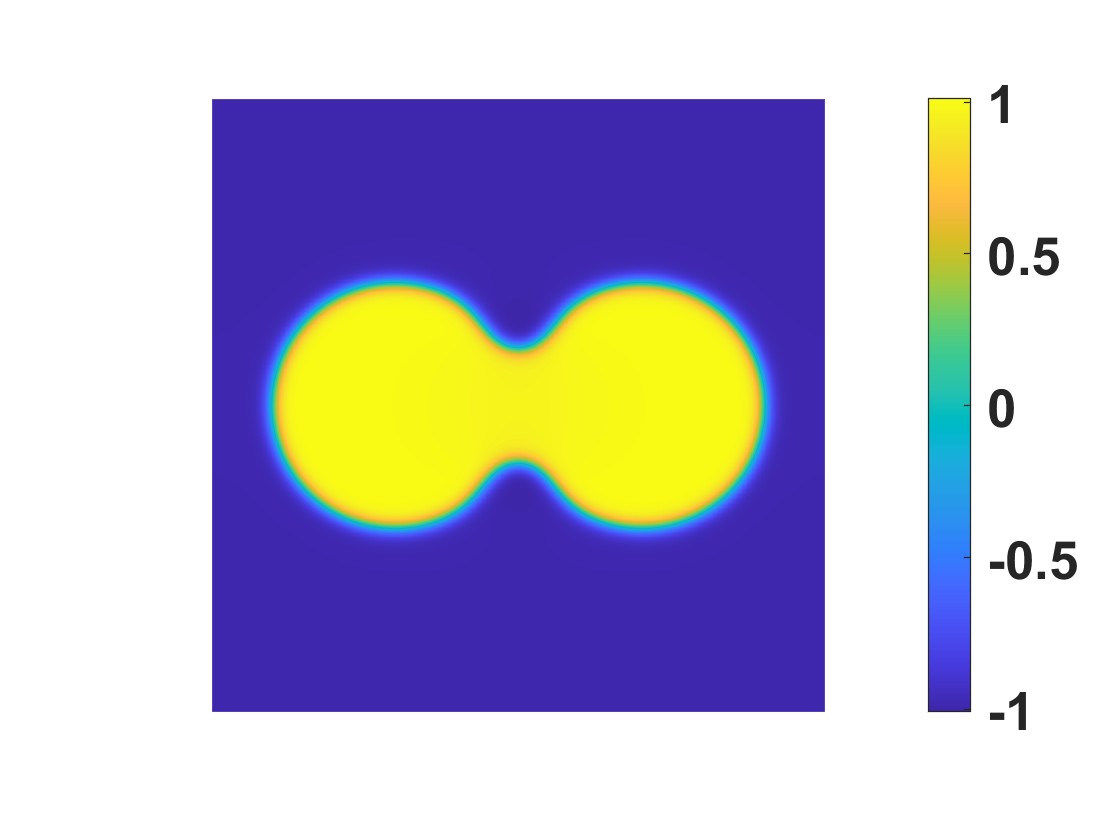} &
\snapshot{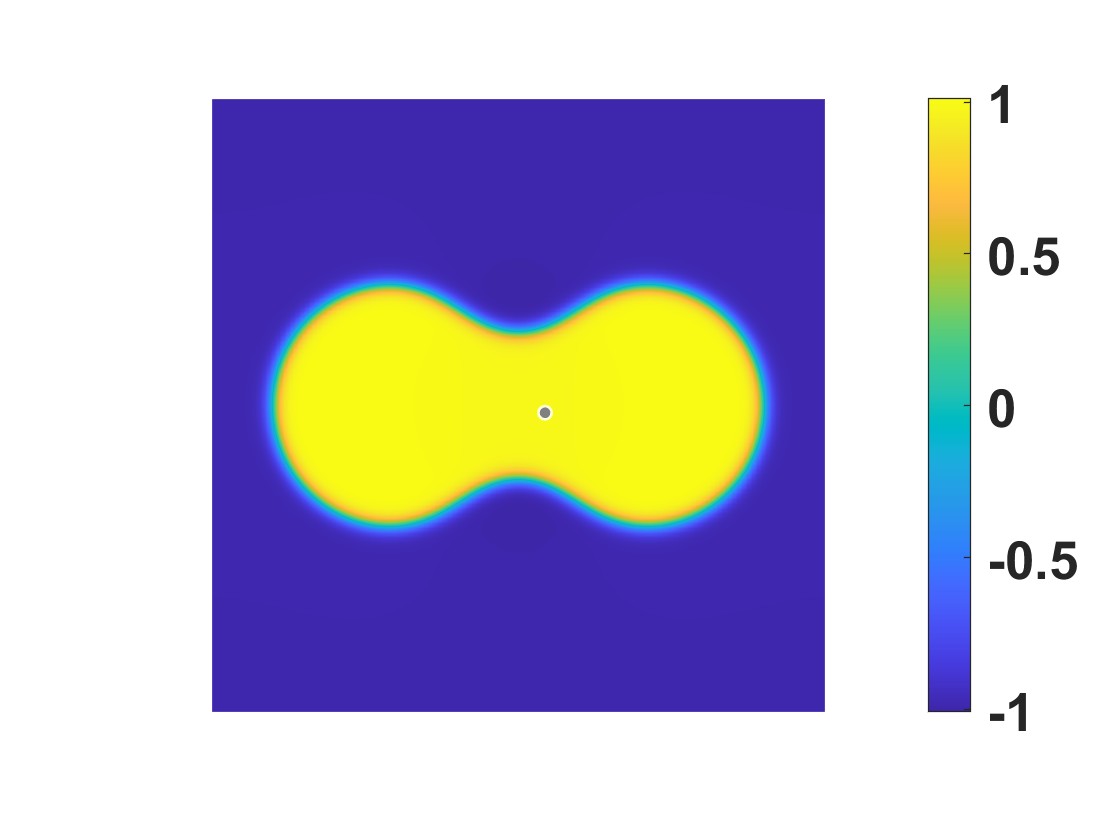} &
\snapshot{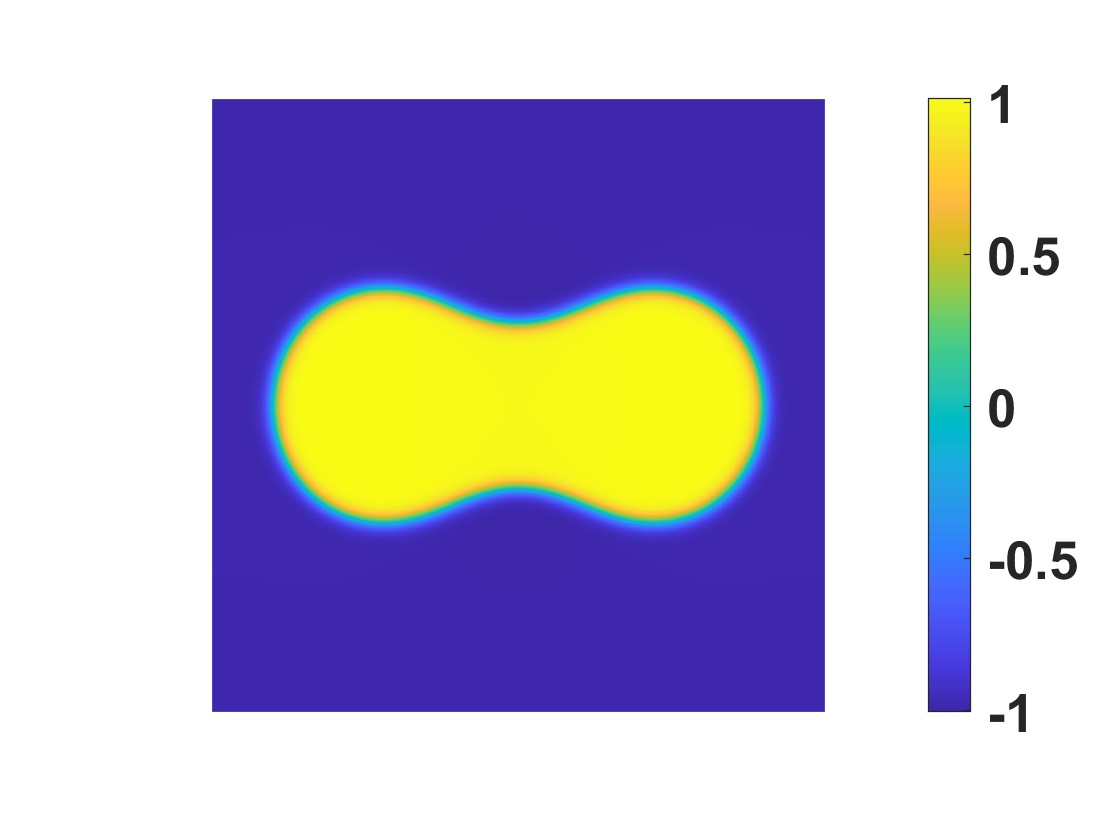} \\
\midrule
{\footnotesize RLM-PC} &
\snapshot{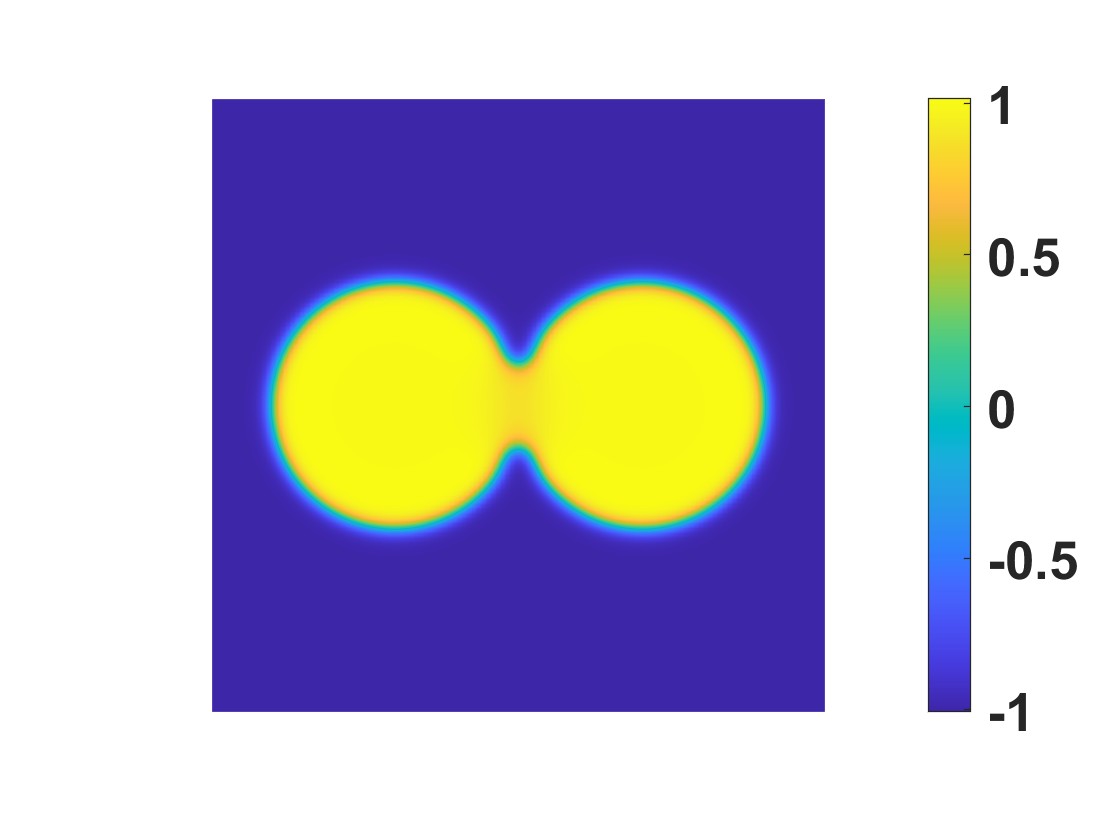} &
\snapshot{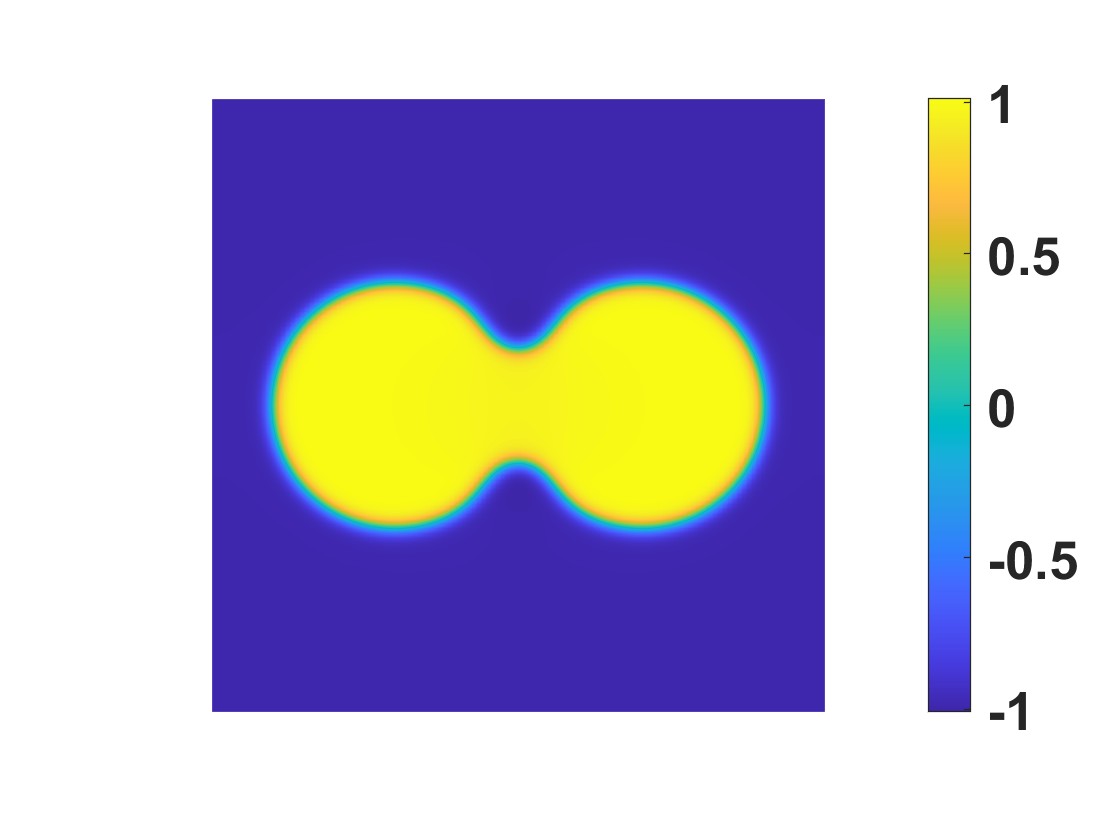} &
\snapshot{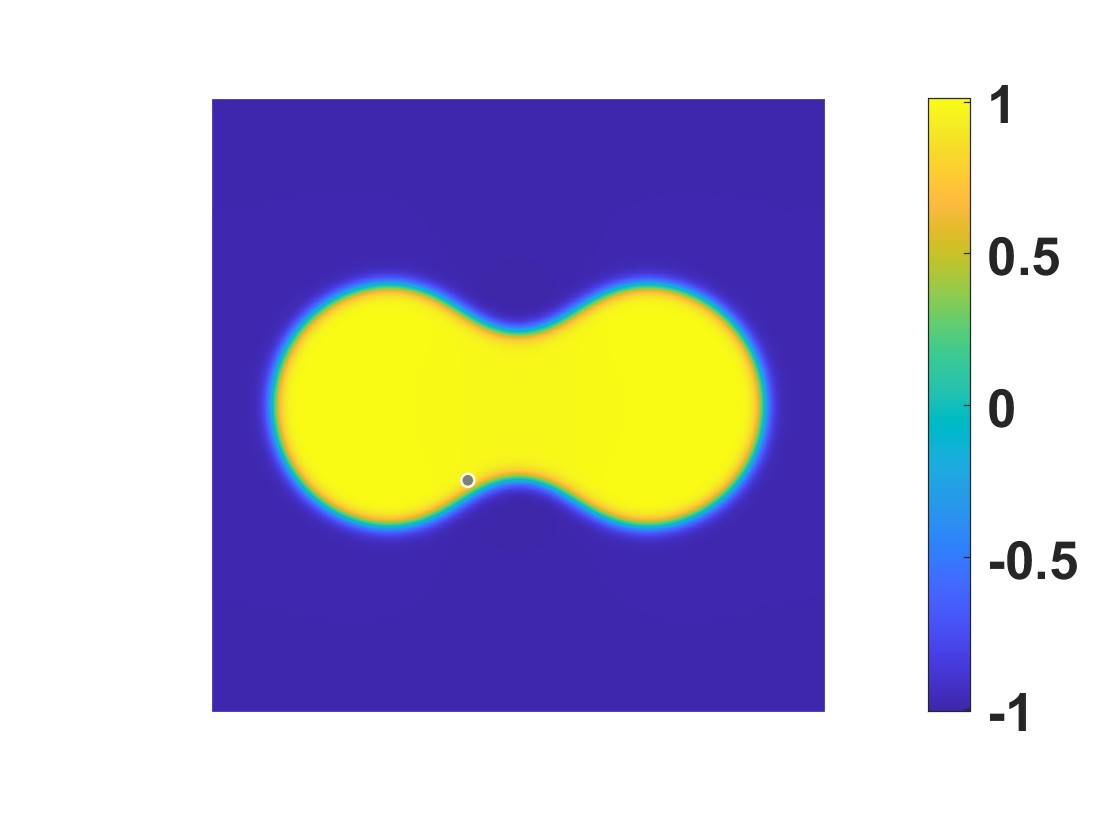} &
\snapshot{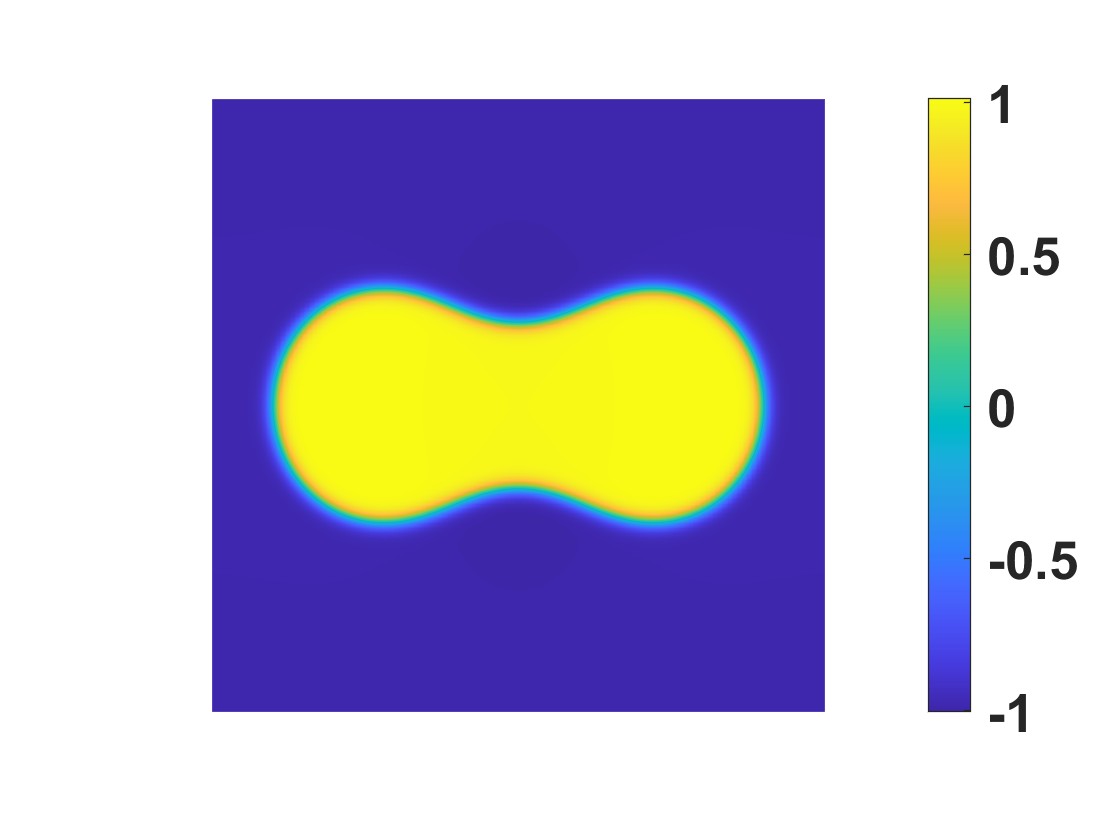} \\
\bottomrule
\end{tabular}
\caption{Snapshots of the phase variable $\phi$ for the Cahn--Hilliard equation at $T=500, 5000, 25000, 50000$. Each row corresponds to a different scheme: SAV-CN (top), RLM-Q-CN with $\alpha=10^{-2}$ (middle), and RLM-PC-CN with $\alpha=10^{-2}$ (bottom). Each column shows results at the same time point. }
\label{fig:PFCH}
\end{figure}

To further analyze the behavior of the schemes for the Cahn--Hilliard equation, we focus on the evolution of energy and related quantities. The results are summarized in Figure~\ref{fig:CH_energy}. In Figure~\ref{fig:CH_energy}(\subref{subfig:CH_energy_a}), we observe that the relaxed original energies $E_{\text{RLM}}$ of RLM-Q-CN and RLM-PC-CN with $\alpha=10^{-2}$ are identical to that of SAV-CN. We compare the $q$ values between the two RLM schemes in Figure~\ref{fig:CH_energy}(\subref{subfig:CH_energy_b}). When $\alpha=10^{-2}$, $q$ is very close to 1, while decreasing $\alpha$ causes $q$ to oscillate. In Figure~\ref{fig:CH_energy}(\subref{subfig:CH_energy_c}) and (\subref{subfig:CH_energy_d}), we compare two classes of energy differences. The first checks the effects of $\alpha$ on the accuracy of $\tilde{E}$, and the second checks the effects on the accuracy of the original energy $E$. When $\alpha=10^{-2}$, the performance of RLM-PC-CN is much better than that of RLM-Q-CN in (\subref{subfig:CH_energy_c}) and (\subref{subfig:CH_energy_d}). 

\begin{figure}[H]
\centering
\subfig{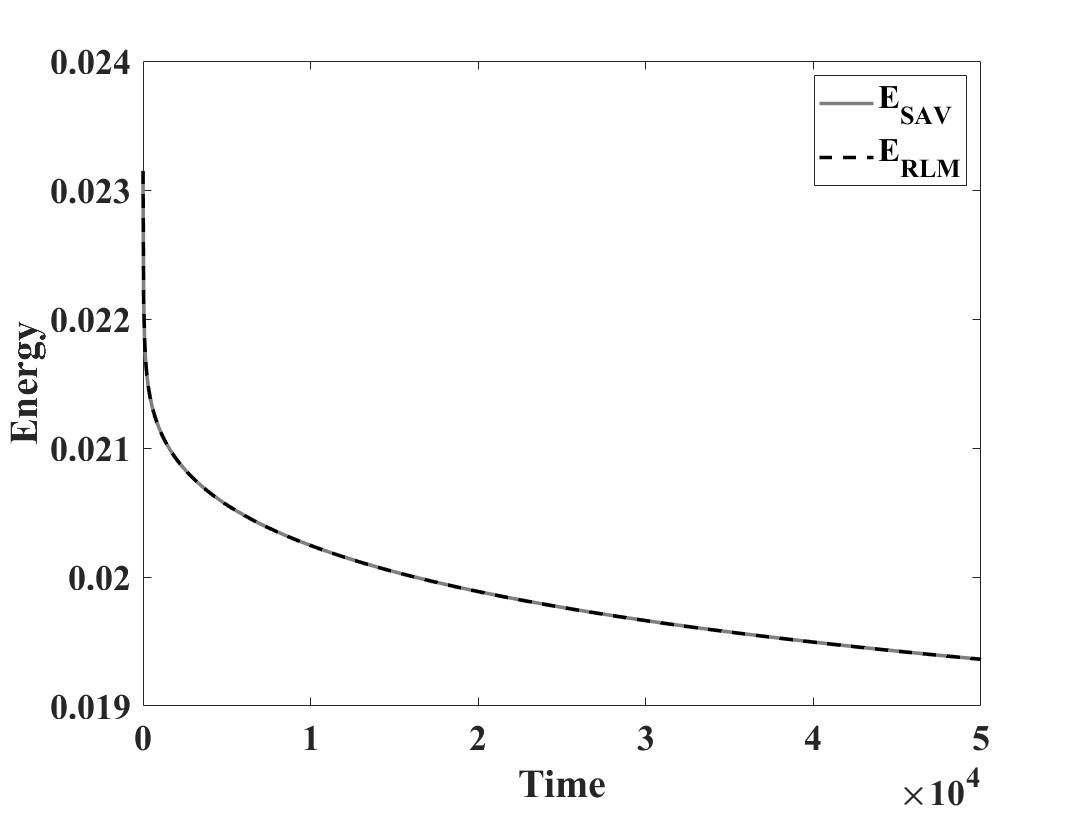}{Energy}{subfig:CH_energy_a}
\subfig{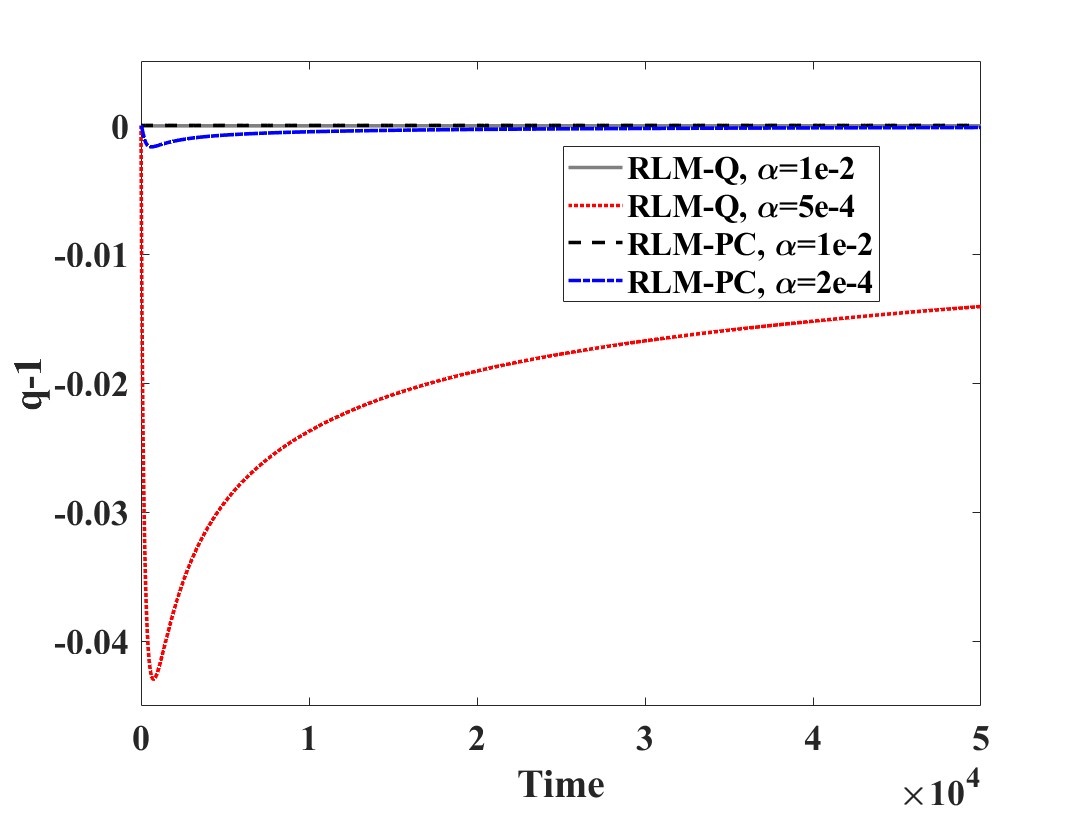}{Scaling factor $q$}{subfig:CH_energy_b}

\subfig{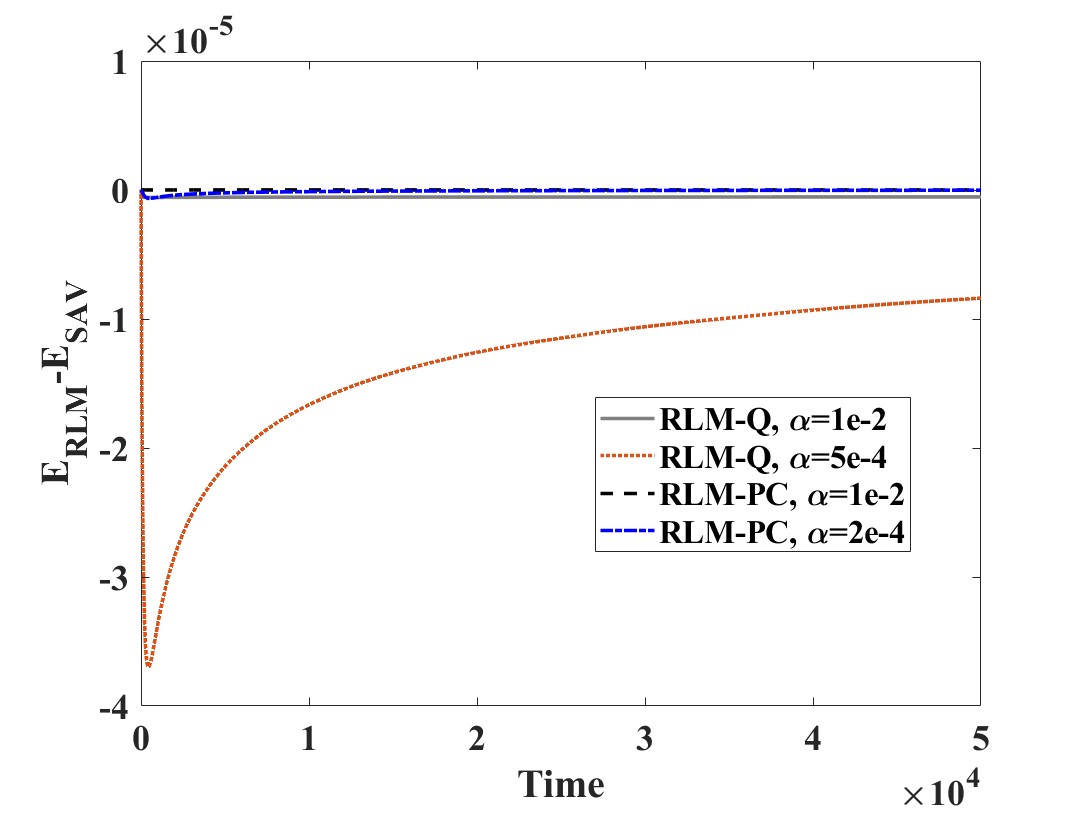}{First class of energy differences}{subfig:CH_energy_c}
\subfig{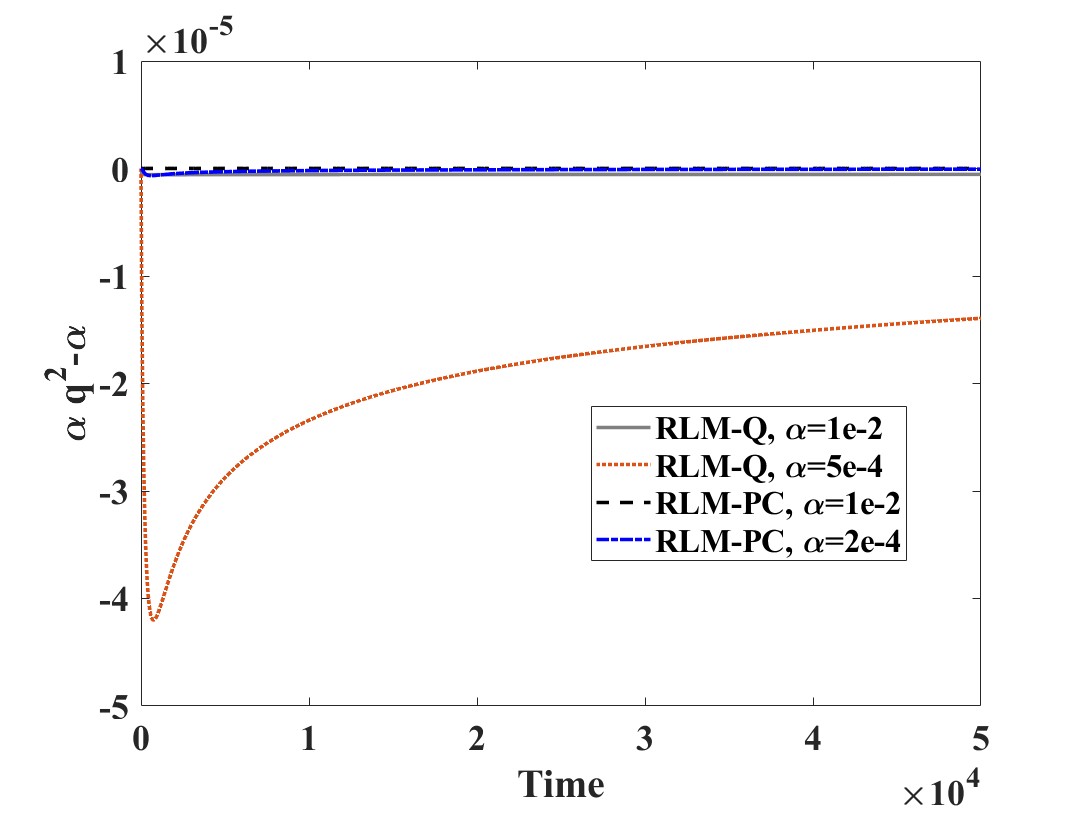}{Second class of energy differences}{subfig:CH_energy_d}
\caption{Time evolution of the energy, scaling factor $q$, and energy differences for the Cahn--Hilliard equation. (\subref{subfig:CH_energy_a})~Energies of SAV-CN, RLM-Q-CN ($\alpha=10^{-2}$, $5\times10^{-4}$), and RLM-PC-CN ($\alpha=10^{-2}$, $2\times10^{-4}$). (\subref{subfig:CH_energy_b})~Scaling factor $q$. (\subref{subfig:CH_energy_c})~First class of energy differences $\tilde{E}-E_{\text{SAV}}$. (\subref{subfig:CH_energy_d})~Second class of energy differences $\alpha (q^2-1)$. All simulations used $\Delta t=1\times 10^{-2}$.}
\label{fig:CH_energy}
\end{figure}

Next, we examine the accuracy of RLM-Q-CN and RLM-PC-CN with sufficiently large $\alpha$ values, as shown in Figure~\ref{fig:CH_energy_RLMs}(\subref{subfig:CH_RLMs_a}) and (\subref{subfig:CH_RLMs_b}), with the final time $T_{\mathrm{final}}=5\times 10^4$. For $\alpha= 10^{-2}$ and $\alpha=10^3$, the energy dissipation rates are close. This occurs because when $\alpha=10^{-2}$ or $\alpha=10^3$, $q$ is very close to $1$ in Figure~\ref{fig:CH_energy_RLMs}(\subref{subfig:CH_RLMs_c}) and (\subref{subfig:CH_RLMs_d}), thus the relaxation term $\alpha (q^2-1)$ remains approximately zero, thereby preserving the energy dissipation rate of the original energy. With the same $\alpha$, $q$ of RLM-PC is much closer to $1$ in the two cases.
These results match the Allen--Cahn case. For both the Allen--Cahn and Cahn--Hilliard equations, RLM-PC keeps $q$ closer to 1 than RLM-Q at the same $\alpha$. We also find that a larger $\alpha$ introduces a smaller error in $q$.
\begin{figure}[H]
\centering
\subfig{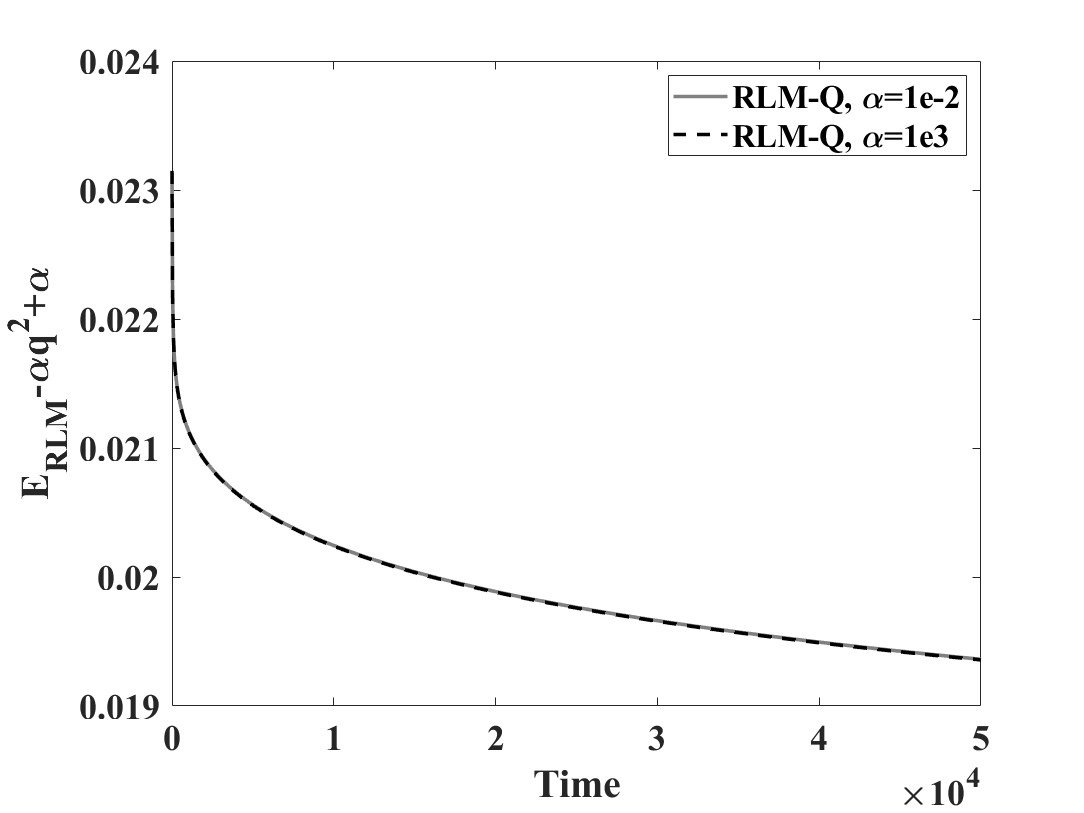}{RLM-Q}{subfig:CH_RLMs_a}
\subfig{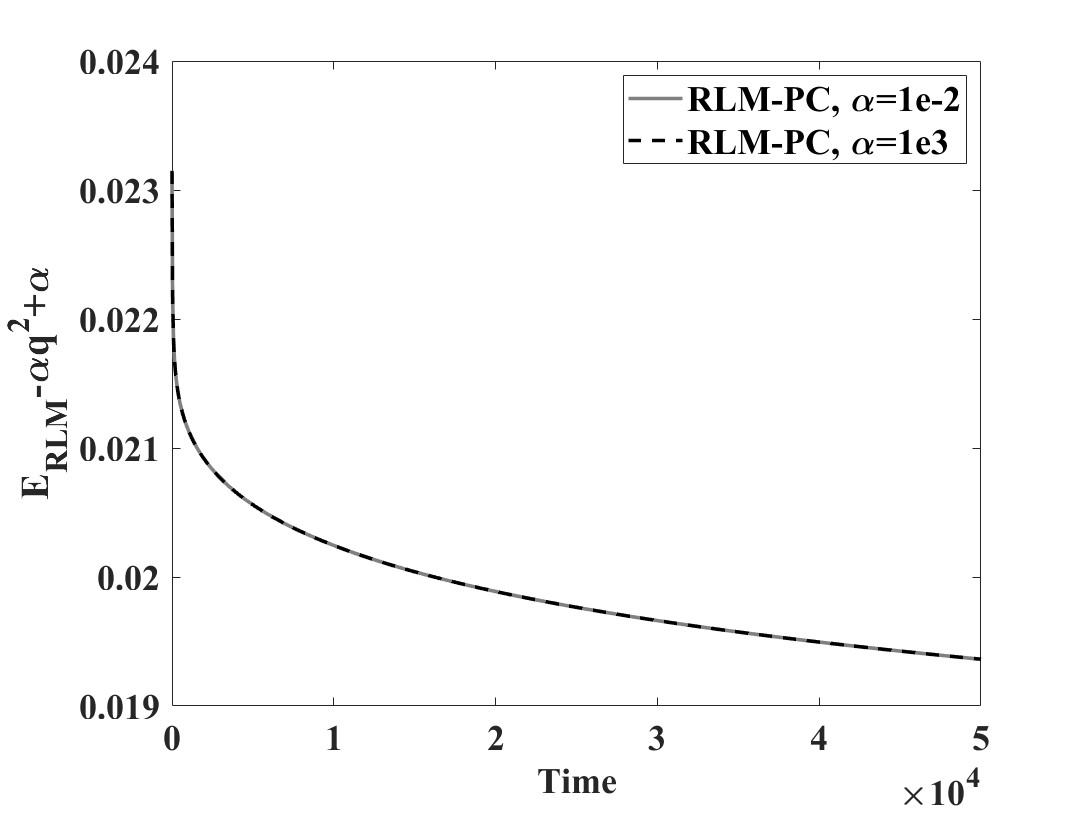}{RLM-PC}{subfig:CH_RLMs_b}

\subfig{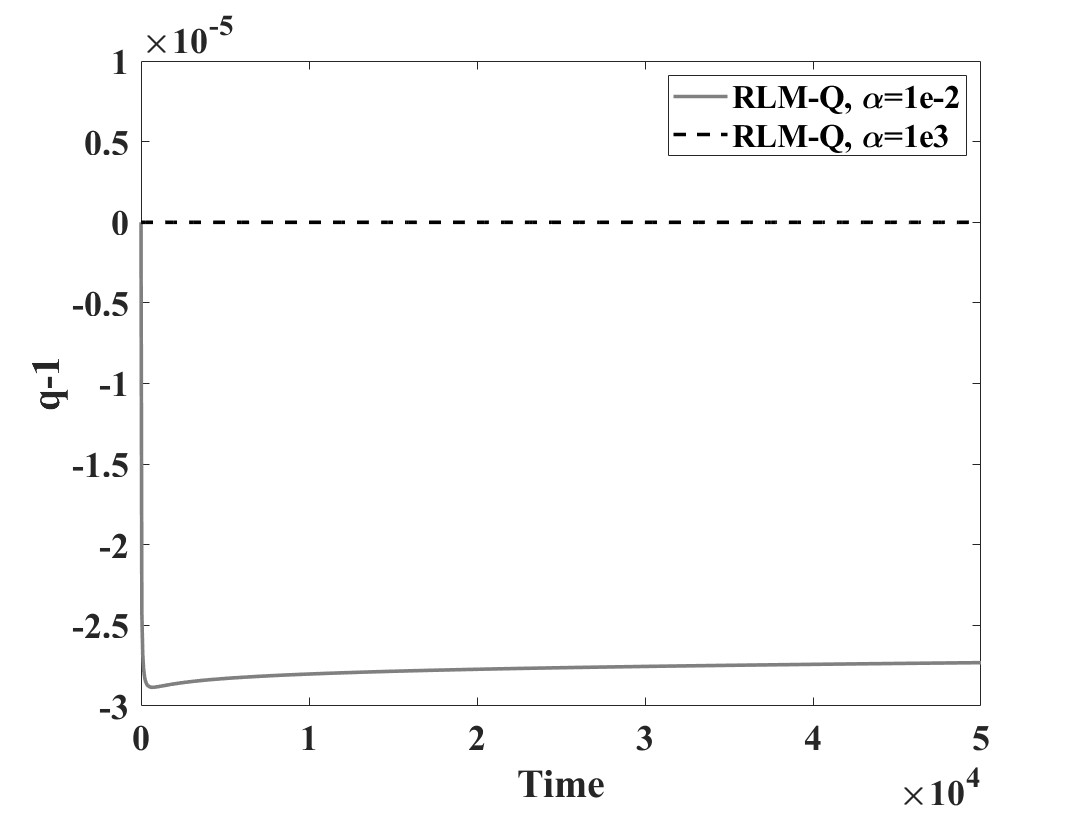}{Scaling factor $q$ ($\alpha=10^{-2}$)}{subfig:CH_RLMs_c}
\subfig{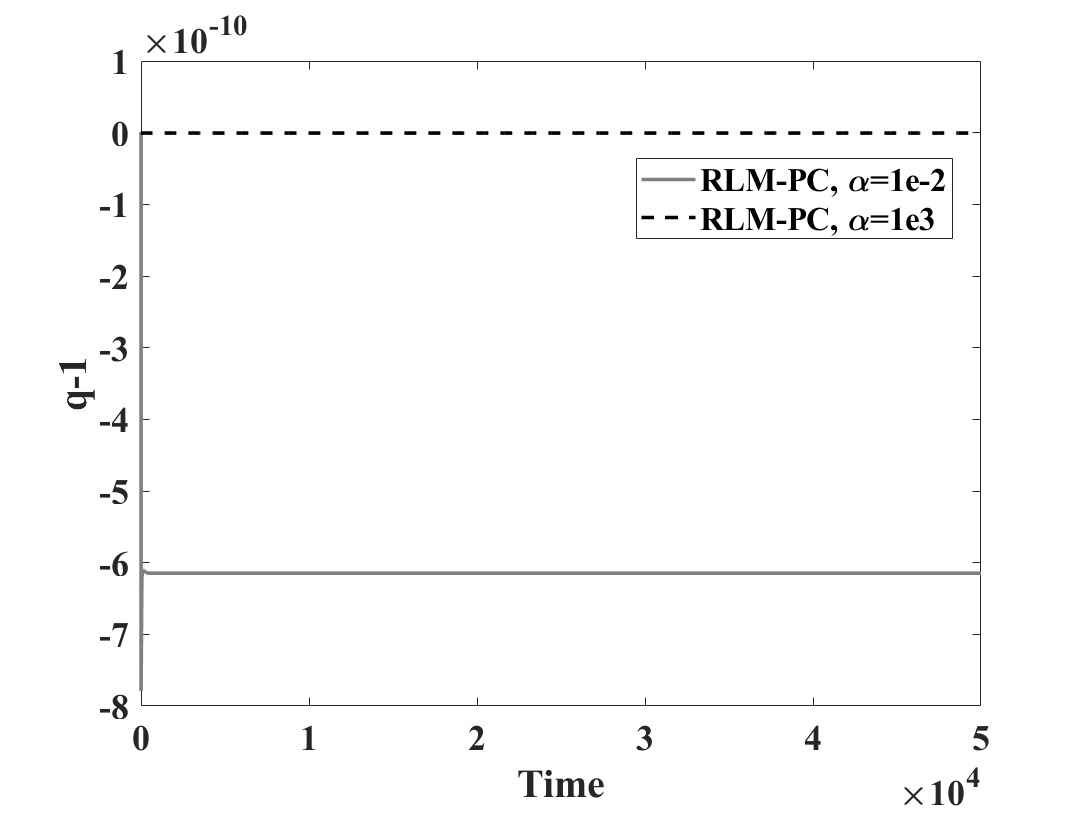}{Scaling factor $q$ ($\alpha=10^{3}$)}{subfig:CH_RLMs_d}

\caption{Time evolution of the original energy $\tilde{E}-\alpha (q^2-1)$ and scaling factor $q$ for the Cahn--Hilliard equation with $\alpha=10^{-2}$ and $10^{3}$. (\subref{subfig:CH_RLMs_a}) RLM-Q: the energy curves for $\alpha=10^{-2}$ coincide with those of  $\alpha=10^3$. (\subref{subfig:CH_RLMs_b}) RLM-PC: the energy curves for $\alpha=10^{-2}$ are also the same as those of  $\alpha=10^3$. (\subref{subfig:CH_RLMs_c}) and (\subref{subfig:CH_RLMs_d}) show the time evolutions of $q$ for the two RLM methods with $\alpha=10^{-2}$ and $\alpha=10^{3}$.}
\label{fig:CH_energy_RLMs}
\end{figure}

Given the high performance of the RLM schemes at large $\alpha$, we next use both RLM schemes with $\alpha = 1\times 10^3$ to simulate the long-time behavior of two merging droplets. The initial conditions and parameters are the same as above.
We adopt $\Delta t = 0.1$ to investigate the long-time behavior until the system reaches a steady state with $T_{final} = 3\times 10^6$. 
Figure~\ref{fig:Energy-long-CH}(\subref{subfig:Energy-CH-long}) shows that the energy curves of the RLM-Q and RLM-PC schemes coincide, indicating that both share the same accuracy.
We then define the total volume $V(t) = \int_\Omega \phi(\bx, t) \, d\bx$. 
In Figure~\ref{fig:Energy-long-CH}(\subref{subfig:q-long}) and (\subref{subfig:volume-long}), we evaluate the performance of the two RLM schemes with respect to the scaling factor $q$ and the conservation of total volume. 
The results demonstrate that RLM-PC-CN outperforms RLM-Q-CN.  
\begin{figure}[H]
\centering
\begin{subfigure}[b]{0.32\textwidth}
\centering
\includegraphics[width=\textwidth]{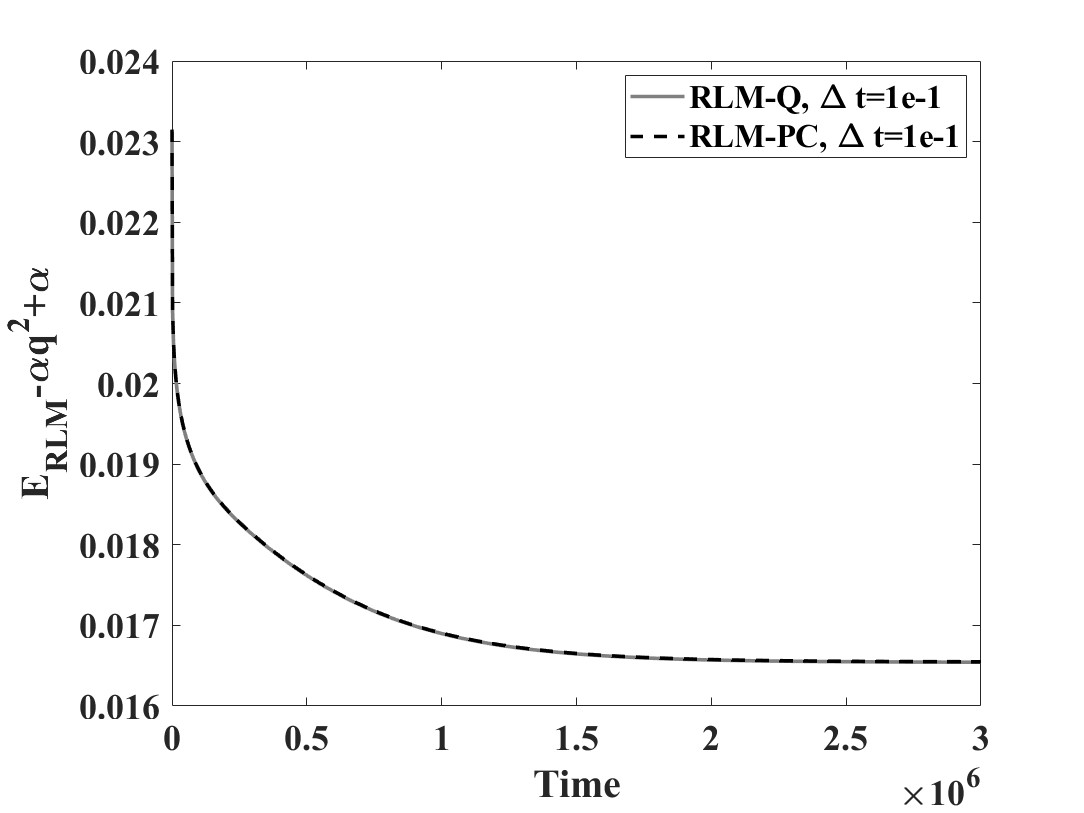}
\caption{Original energy} 
\label{subfig:Energy-CH-long}
\end{subfigure}
\begin{subfigure}[b]{0.32\textwidth}
\centering
\includegraphics[width=\textwidth]{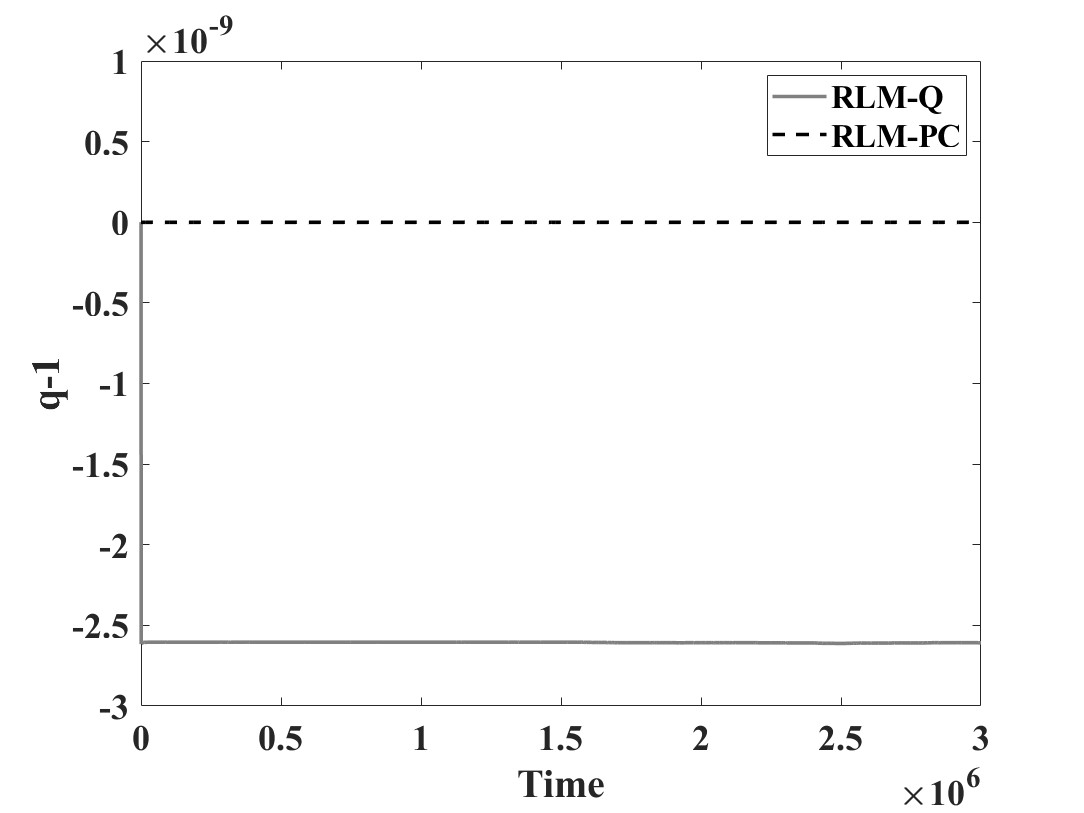}
\caption{Scaling factor $q$}
\label{subfig:q-long}
\end{subfigure}
\begin{subfigure}[b]{0.32\textwidth}
\centering
\includegraphics[width=\textwidth]{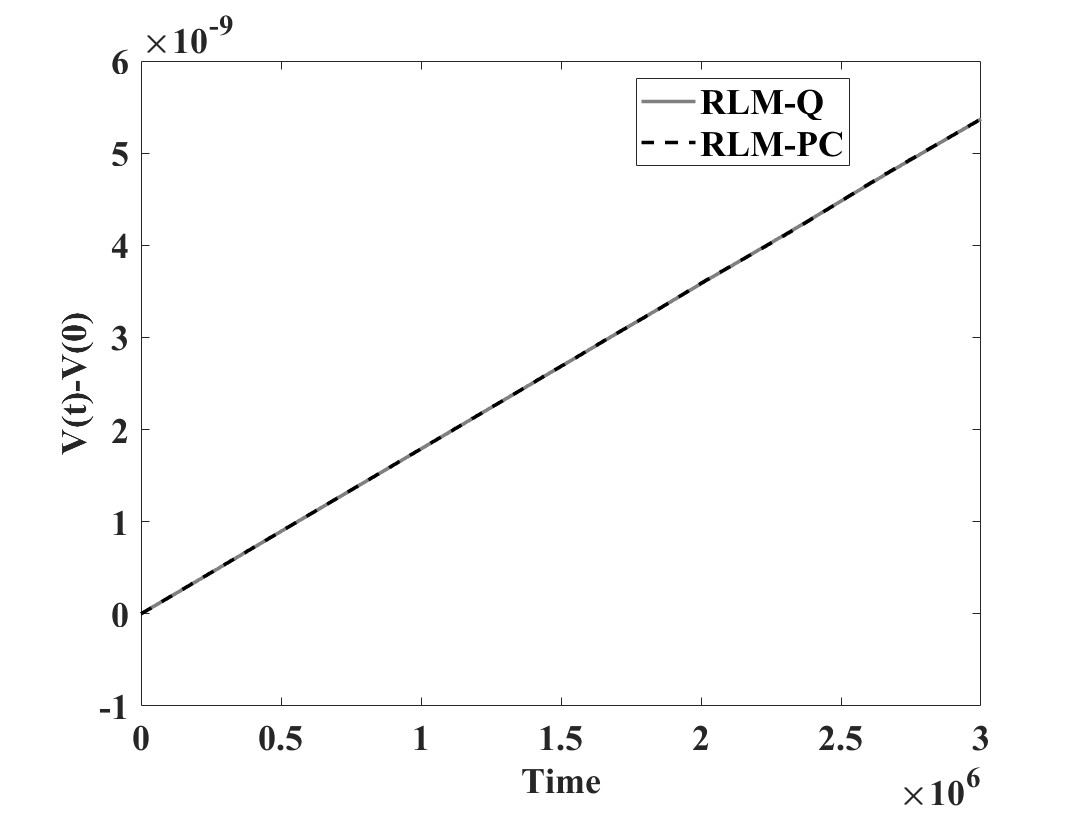}
\caption{Total volume}
\label{subfig:volume-long}
\end{subfigure}
\caption{Long-time simulations of the Cahn--Hilliard equation with the double-well potential using the RLM-Q-CN and RLM-PC-CN schemes ($\alpha=10^{3}$): time evolution of the original energy in (a), scaling factor $q$ in (b), and total volume $V(t)=\int_\Omega \phi(\bx,t)\,\diff\bx$ in (c).}
\label{fig:Energy-long-CH}
\end{figure}
The snapshots of the phase variable with $\Delta t=0.1$ at $ T= 1.5\times 10^5,1.5\times 10^6,3\times 10^6$ simulated by RLM-Q-CN and RLM-PC-CN are shown in Figure \ref{fig:long-CH}. The two schemes yield visually indistinguishable patterns. 
\begin{figure}[H]
\centering
\renewcommand{\arraystretch}{0.9}
\setlength{\tabcolsep}{2pt}
\begin{tabular}{c c c c c}
\toprule
{{\footnotesize Method}} & $T=1.5\times 10^5$ & $T=1.5\times 10^6$ & $T=3\times 10^6$  \\
\midrule
{\footnotesize RLM-Q} &
\snapshot{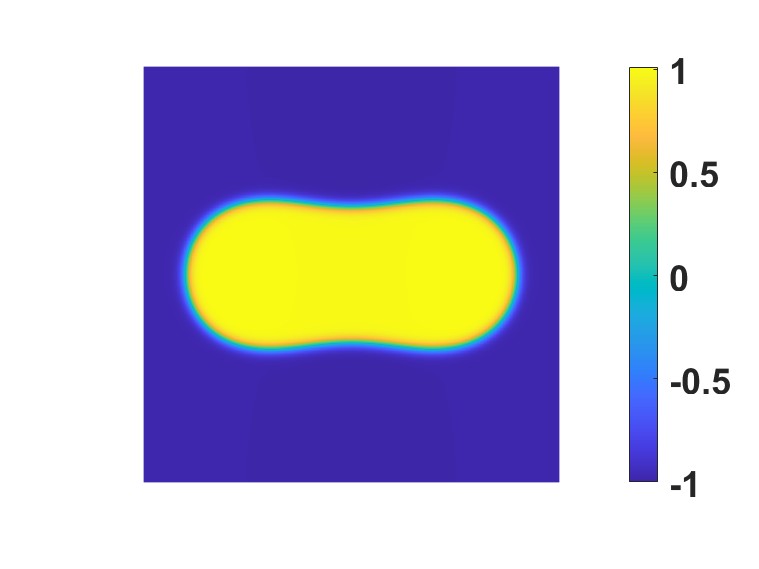} &
\snapshot{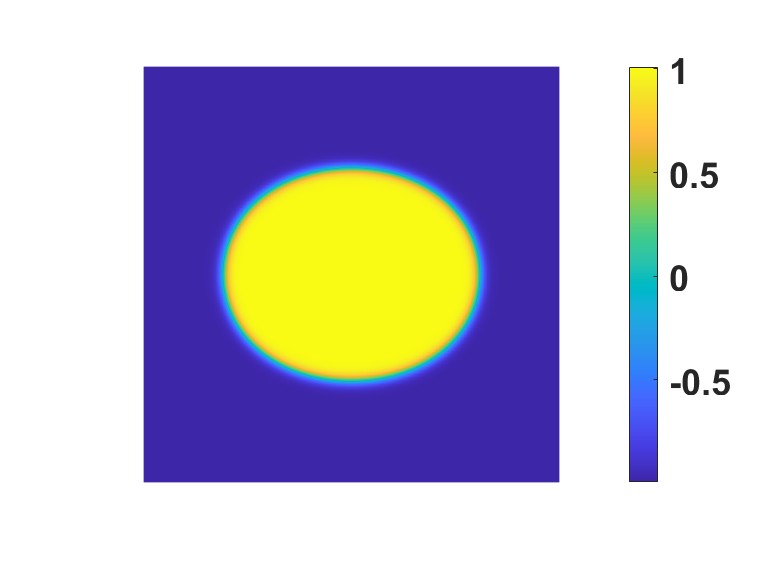} &
\snapshot{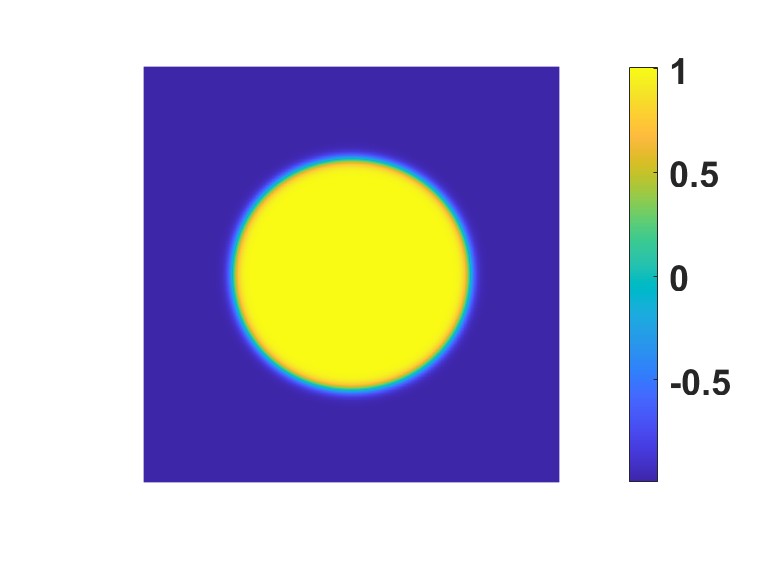} \\
\midrule
{\footnotesize RLM-PC} &
\snapshot{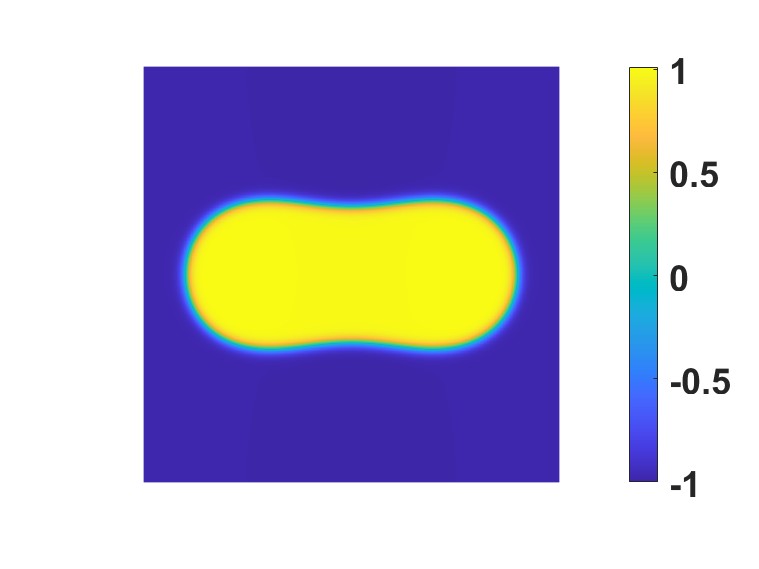} &
\snapshot{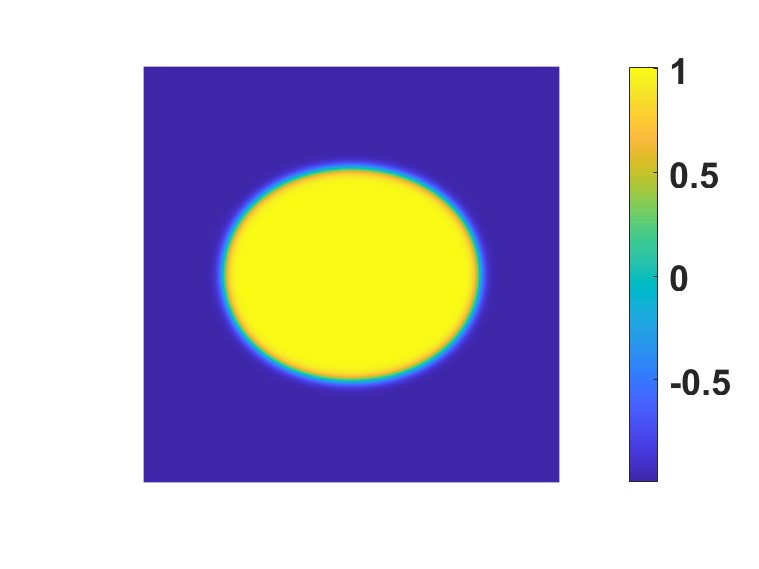} &
\snapshot{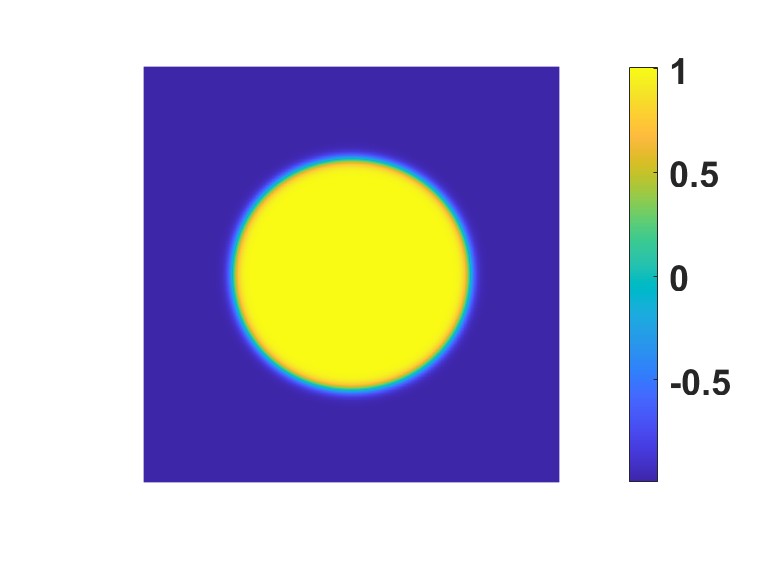} \\
\bottomrule
\end{tabular}
\caption{Snapshots of the phase variable $\phi$ for the Cahn--Hilliard equation at $T=1.5\times 10^5, 1.5\times 10^6, 3\times 10^6$. Each row corresponds to a different scheme: RLM-Q-CN with $\alpha=10^{3}, \Delta t=0.1$ (top), and RLM-PC-CN with $\alpha=10^{3}, \Delta t=0.1$ (bottom). Each column shows results at the same time point. }
\label{fig:long-CH}
\end{figure}

The preceding examples with the double-well potential serve as benchmark cases. We now demonstrate the applicability of RLM to more challenging potentials, starting with the Flory--Huggins potential that contains logarithmic terms.

\subsection{CH-type Equation with Flory--Huggins Potential}\label{sec:CH-flory-huggins}

In this example, we consider the polymer solution system with the bulk free energy given by a Flory--Huggins potential
\begin{equation} \label{eq:free-energy-polymer}
E(\phi) = \int_\Omega \left( \frac{\varepsilon^2}{2} |\nabla \phi|^2 + F(\phi) \right) \,\diff\bx, \quad F(\phi) = \phi \ln \phi+(1-\phi) \ln(1-\phi)+\chi_c \phi(1-\phi),
\end{equation}
where $\phi$ is the volume fraction of the polymer segments, $\varepsilon$ is the interfacial thickness parameter, and $\chi_c$ is a parameter describing the affinity between the polymer and the solvent. In this case, 
\[
f(\phi) = F'(\phi) = \ln \phi -\ln (1-\phi)+\chi_c (1-2\phi).
\]

Based on the RLM reformulation in Section~\ref{sec:rlm-methods}, the Cahn--Hilliard-type equation with the Flory--Huggins potential is given by
\begin{align}
& \partial_t \phi = M \Delta \mu, \label{eq:CH-flory-huggins-phi} \\
& \mu = -\varepsilon^2 \Delta \phi + q(t) \left[\ln \phi - \ln(1-\phi) + \chi_c(1-2\phi)\right], \label{eq:CH-flory-huggins-mu} \\
& \frac{d}{dt} \int_\Omega F(\phi) \,\diff\bx + \alpha \,\frac{d \big(q(t)\big)^2 }{dt}= \int_\Omega q(t) \left[\ln \phi - \ln(1-\phi) + \chi_c(1-2\phi)\right] \partial_t \phi \,\diff\bx, \label{eq:CH-flory-huggins-q}
\end{align}
where $M > 0$ is the mobility constant and $q(0)=1$. Homogeneous Neumann boundary conditions are assumed.

Although physically $\phi \in [0,1]$, the Cahn--Hilliard equation with constant mobility operator $M\Delta$ does not enforce such a bound. To overcome this shortcoming,  $F$ is modified as follows~\cite{copetti1992numerical}
\[
\tilde F(\phi)=\begin{cases} \phi \ln \phi +\frac{(1-\phi)^2}{2\sigma}+(1-\phi)\ln \sigma -\frac{\sigma}{2}+\chi_c \phi(1-\phi), & \text{if} \quad \phi \geq 1-\sigma,\\
\phi \ln \sigma +\frac{\phi^2}{2\sigma}+(1-\phi)\ln (1-\phi) -\frac{\sigma}{2}+\chi_c \phi(1-\phi),&  \text{if}\quad \phi \leq \sigma,\\
F,&  \text{others} .
\end{cases}
\]
When using the RLM-Q approach, we introduce the stabilization technique as \eqref{eq:F-quadratize-stabilization}, namely
\[
\tilde F^{n+1} =(\phi^{n+1})^2+\tilde F (\phibar^{n+1})-(\phibar^{n+1})^2,
\]
i.e., we take $S=1$ in \eqref{eq:F-quadratize-stabilization}.
When applying the SAV approach, we introduce the scalar auxiliary variable as $\sqrt{\int_\Omega \tilde F \,\diff\bx + \Csav}$, where $\Csav$ is a constant to ensure $\int_\Omega \tilde F \,\diff\bx + \Csav \geq 0$.

In this case, we set the computational domain as $\Omega=[0,1]^2$, the interface thickness $\varepsilon=10^{-2}$, the regularization parameter $\sigma=0.01$, $\chi_c=3$, the mobility coefficient $M=10^{-5}$, and the spatial step size $h=1/128$. The initial condition is set as
\[
\phi_0(x,y) = 0.3+0.01\zeta,
\]
where $\zeta$ is a random variable uniformly distributed on $[-1,1]$. Convergence tests are omitted for brevity, as the temporal accuracy is consistent with the results in Sections~\ref{sec:AC-double-well} and~\ref{sec:CH-double-well}.

Next, we examine the accuracy of RLM for free energy systems with logarithmic terms using $\Delta t=10^{-2}$ and final time $T_{final}=2\times 10^4$. Figure~\ref{fig:Energy-polymer}(\subref{subfig:Energy-polymer-a}) and (\subref{subfig:Energy-polymer-b}) show the time evolution of the original energy $E_{\text{RLM}}-\alpha (q^2-1)$ for various $\alpha$ using the RLM-Q-CN and RLM-PC-CN schemes. The energy curves for $\alpha=10^{-2}$ and $10^3$ coincide. Figure~\ref{fig:Energy-polymer}(\subref{subfig:Energy-polymer-c}) shows that $q$ approaches 1 with different $\alpha$, and RLM-PC-CN performs better than RLM-Q-CN in both $\alpha=10^{-2}$ and $\alpha=10^3$ cases. RLM-PC-CN is also computationally more efficient than RLM-Q-CN.

\begin{figure}[H]
\centering
\begin{subfigure}[b]{0.32\textwidth}
\centering
\includegraphics[width=\textwidth]{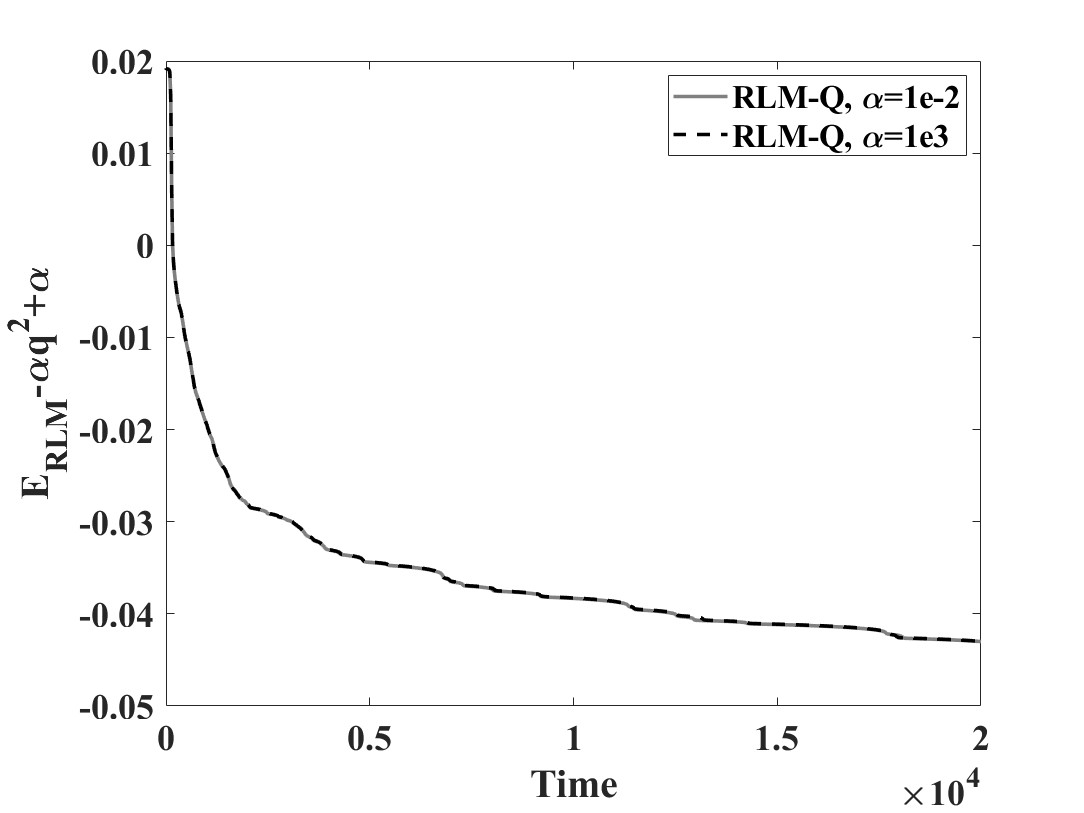}
\caption{RLM-Q}
\label{subfig:Energy-polymer-a}
\end{subfigure}
\hfill
\begin{subfigure}[b]{0.32\textwidth}
\centering
\includegraphics[width=\textwidth]{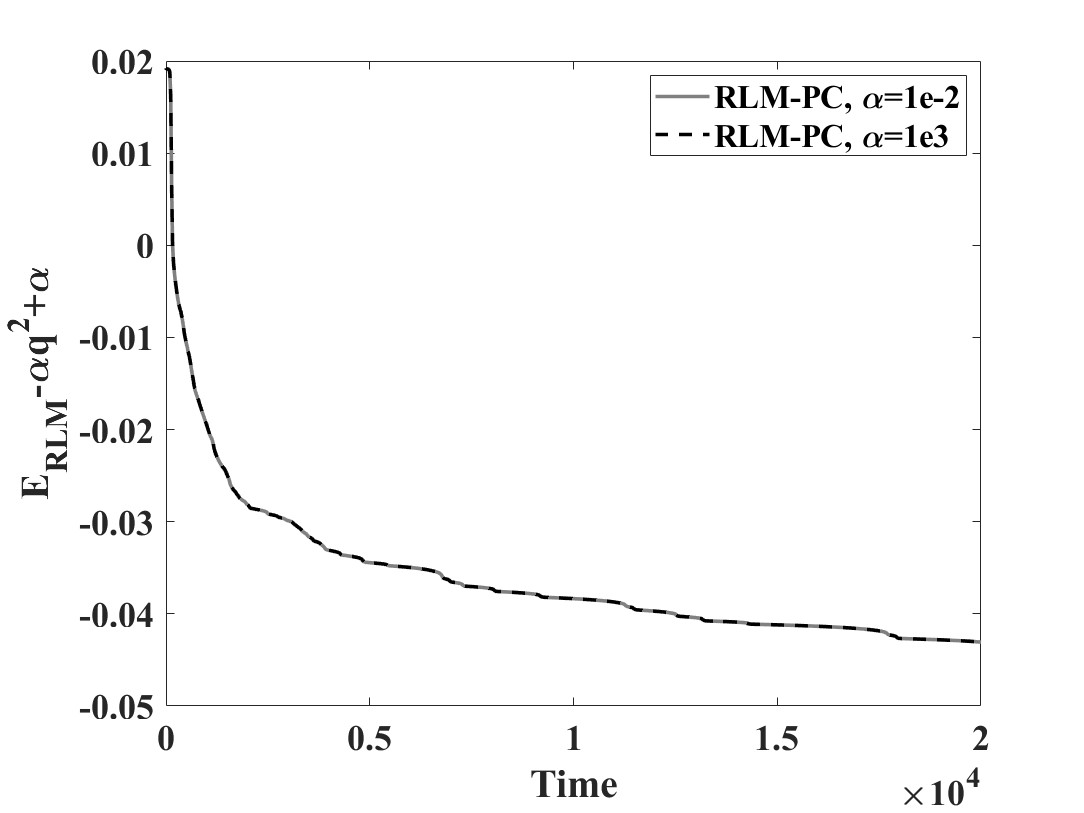}
\caption{RLM-PC}
\label{subfig:Energy-polymer-b}
\end{subfigure}
\hfill
\begin{subfigure}[b]{0.32\textwidth}
\centering
\includegraphics[width=\textwidth]{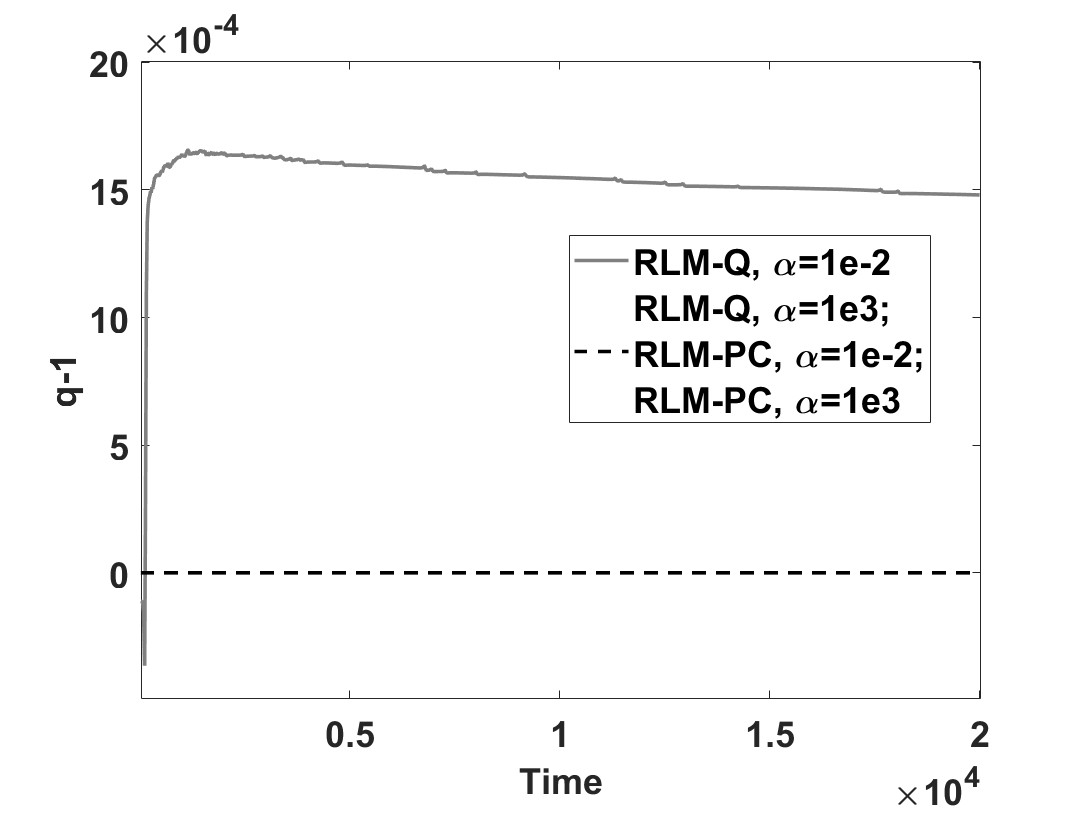}
\caption{Scaling factor}
\label{subfig:Energy-polymer-c}
\end{subfigure}
\caption{Time evolution of the original energies $E_{\text{RLM}}-\alpha (q^2-1)$ in (a) and (b),  and scaling factor $q$ for the Cahn--Hilliard equation with the Flory--Huggins potential, computed by the RLM-Q-CN and RLM-PC-CN schemes with $\alpha=10^{-2}$ and $10^{3}$.}
\label{fig:Energy-polymer}
\end{figure}

Next, we compare snapshots of the phase variable at $T=0$, $400$, $10000$, and $20000$ obtained by the SAV-CN and RLM-PC-CN schemes. For $\Delta t=10^{-2}$, the two methods yield visually indistinguishable patterns in Figure~\ref{fig:polymerCH}.

\begin{figure}[H]
\centering
\renewcommand{\arraystretch}{0.9}
\setlength{\tabcolsep}{2pt}
\begin{tabular}{c c c c c}
\toprule
{{\footnotesize Method}} & $T=0$ & $T=400$ & $T=10000$ & $T=20000$  \\
\midrule
{\footnotesize SAV} &
\snapshot{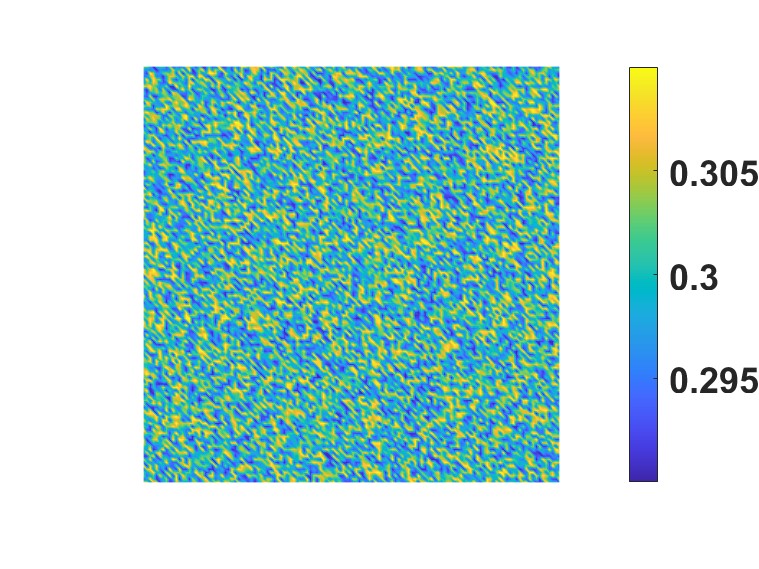} &
\snapshot{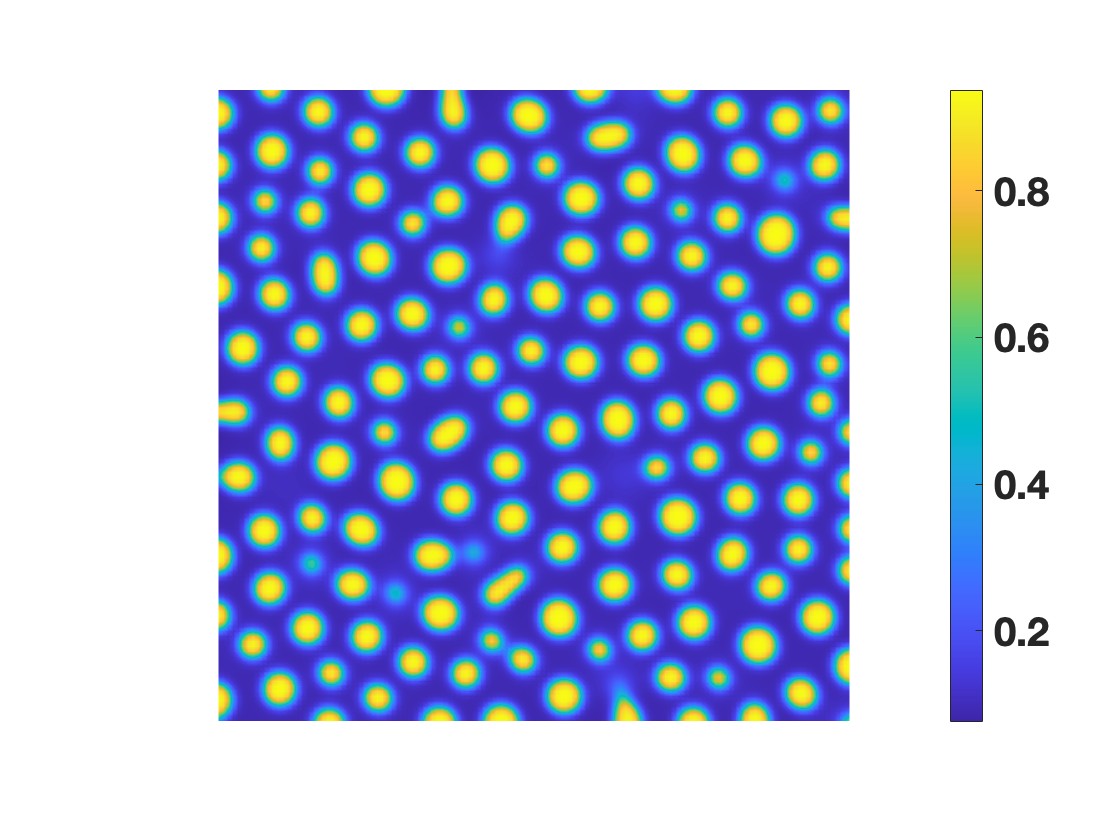} &
\snapshot{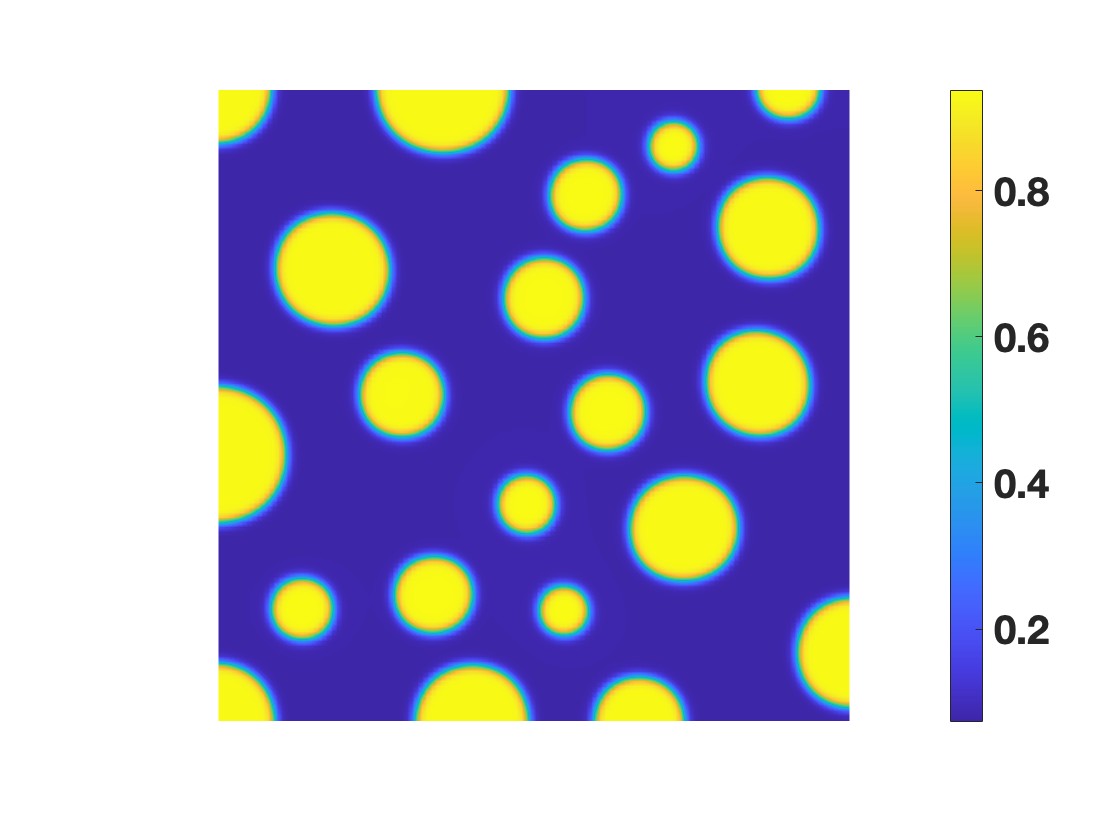} &
\snapshot{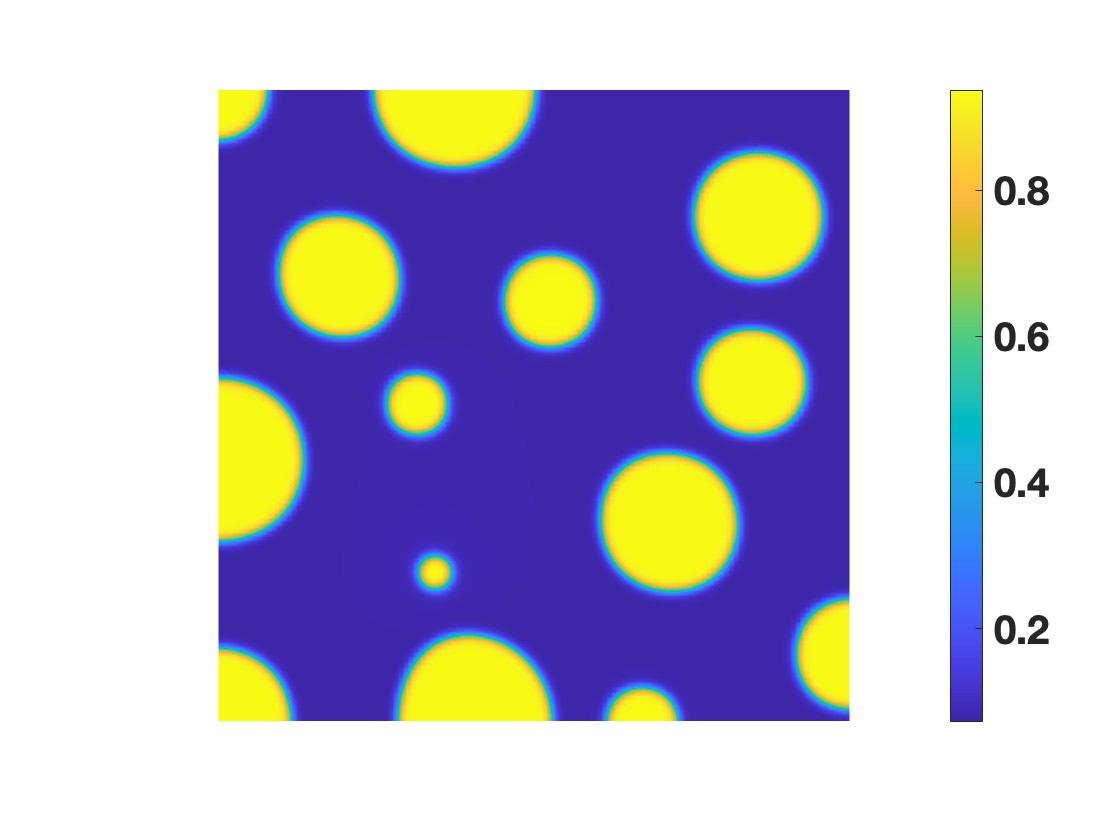} \\
\midrule
{\footnotesize RLM-PC} &
\snapshot{fig/snapshotspolymerSAVT=0.jpg} &
\snapshot{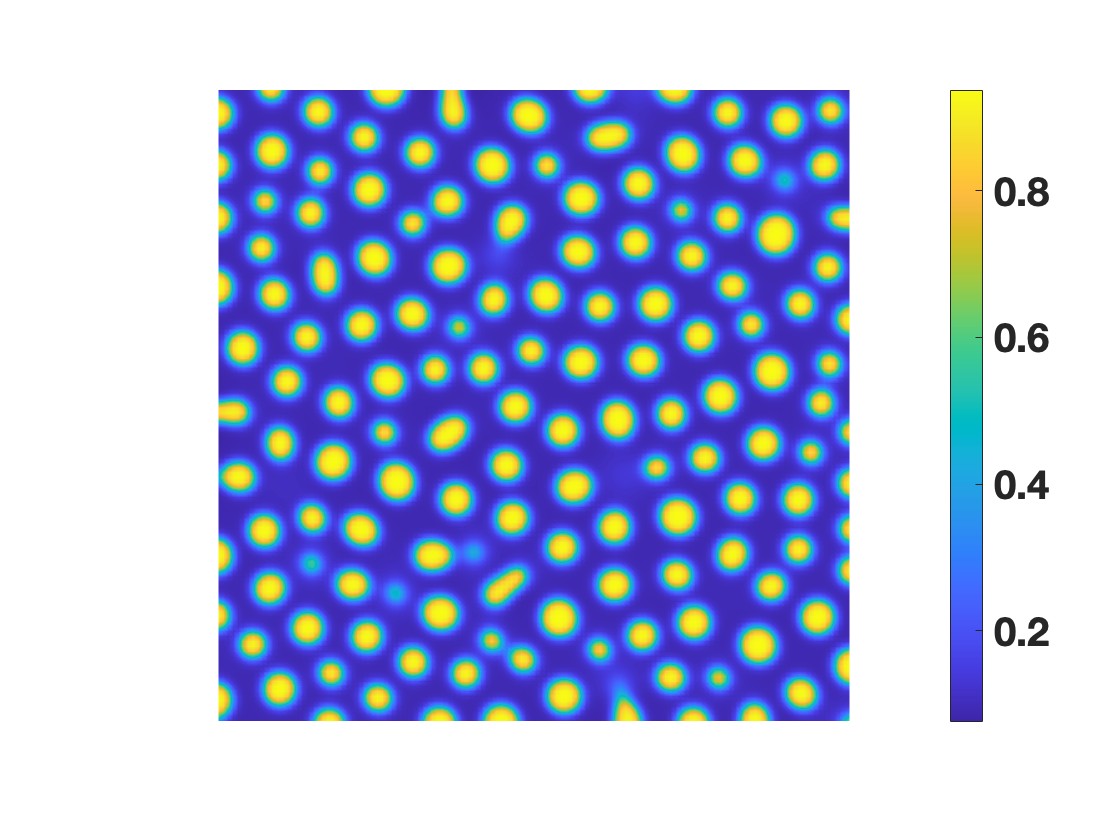} &
\snapshot{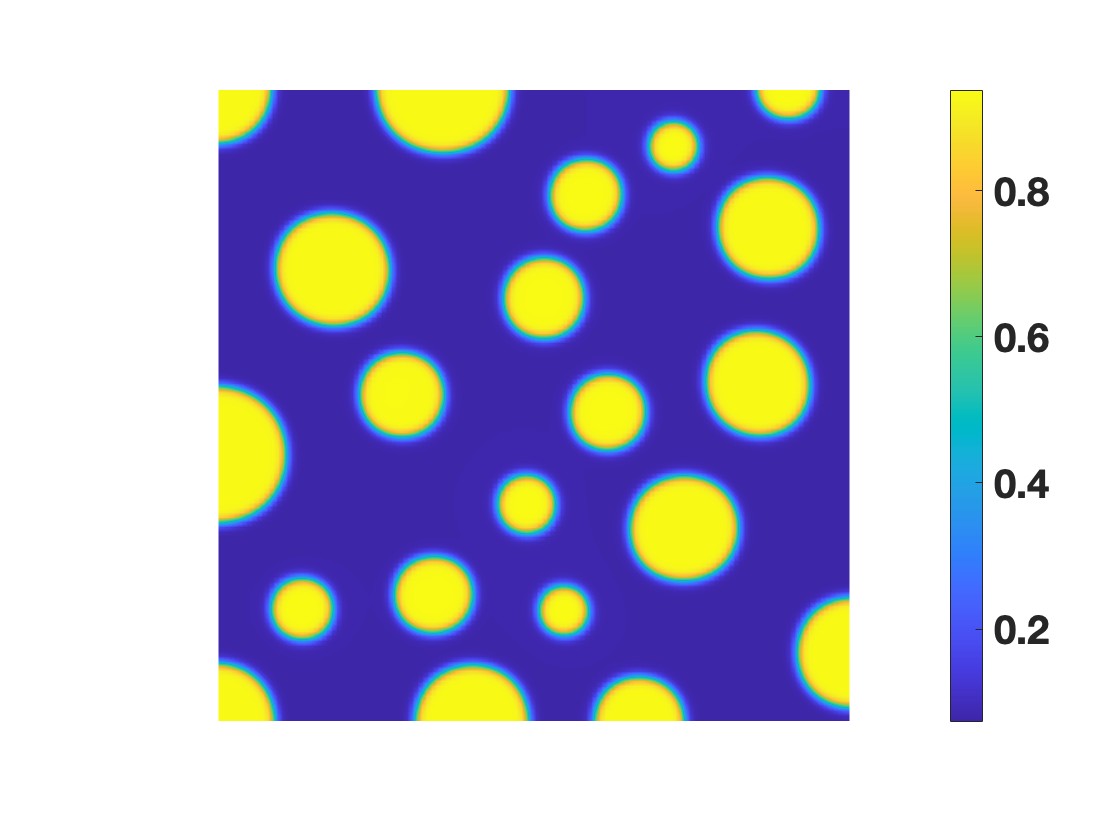} &
\snapshot{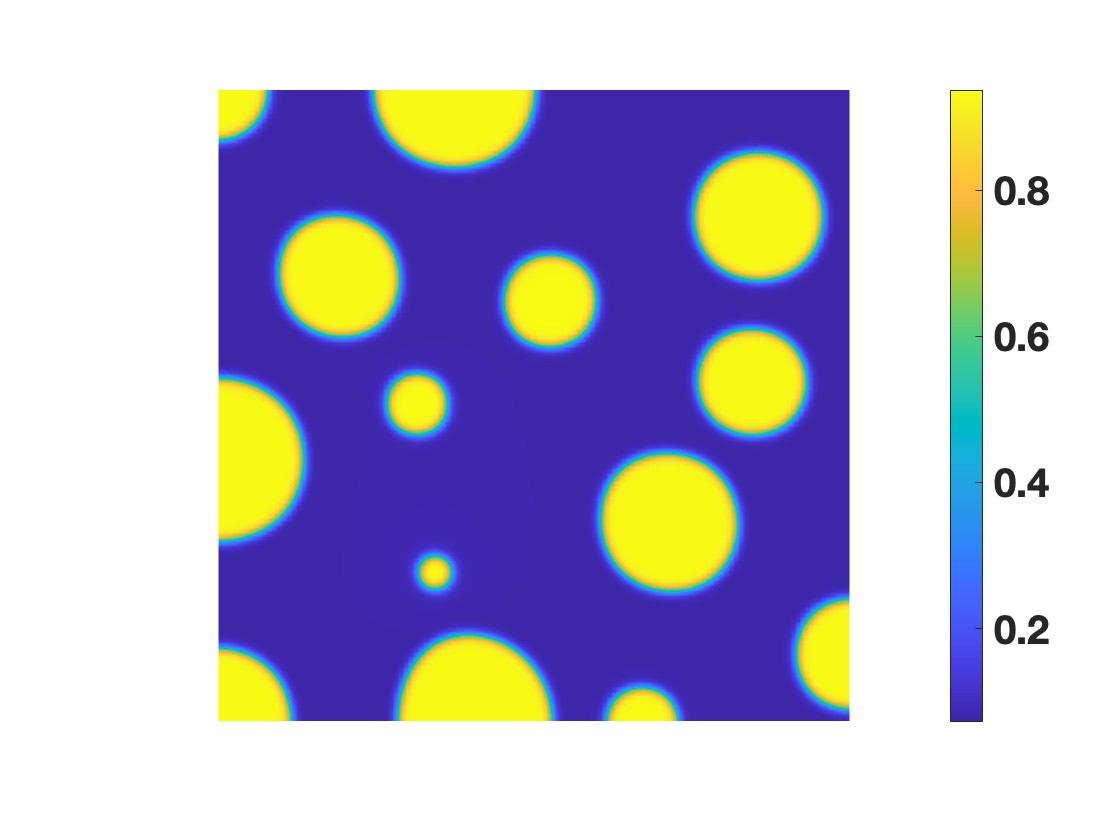} \\
\bottomrule
\end{tabular}
\caption{Snapshots of the phase variable $\phi$ for polymer solution at $T=0, 400, 10000, 20000$. Each row corresponds to a different scheme: SAV-CN (top) and RLM-PC-CN with $\alpha=10^{3}$ (bottom). Each column shows results at the same time point. }
\label{fig:polymerCH}
\end{figure}

Since the constant $\Csav$ in the SAV-CN scheme often affects numerical accuracy, we compare two values, $\Csav=10^2$ and $\Csav=10^8$, to assess their effects on the energy at different time step sizes. We first choose the results of SAV-CN with $\Csav=10^2$, $\Delta t=10^{-2}$ and RLM-CN with $\Delta t=10^{-2}$, $\alpha=10^3$ as the benchmarks, since the energy curves are identical in Figure~\ref{fig:Volumepolymer}(\subref{subfig:Volumepolymer-a}). When increasing $\Csav$ to $10^8$ and setting $\Delta t=5\times 10^{-2}$, the energy curve of SAV-CN deviates from the benchmarks in Figure~\ref{fig:Volumepolymer}(\subref{subfig:Volumepolymer-a}). Meanwhile, the RLM energy curve (simulated by RLM-Q-CN and RLM-PC-CN schemes) with $\Delta t=5\times 10^{-2}$ is identical to the benchmark. This demonstrates that the RLM approach is more robust than SAV. To check the volume-preserving property for both SAV-CN and RLM-CN, the time evolution of total volume with different parameter values is summarized in Figure~\ref{fig:Volumepolymer}(\subref{subfig:Volumepolymer-b}).

\begin{figure}[H]
\centering
\begin{subfigure}[b]{0.45\textwidth}
\centering
\includegraphics[width=\textwidth]{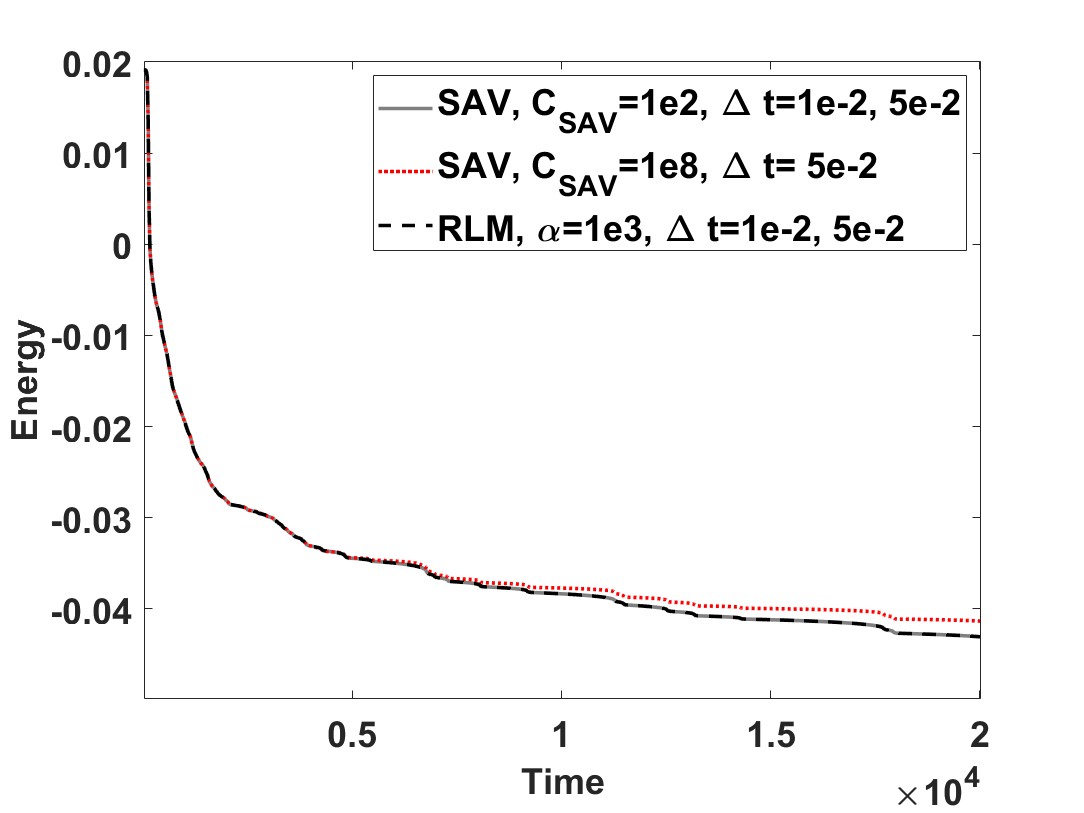}
\caption{Energy}
\label{subfig:Volumepolymer-a}
\end{subfigure}
\begin{subfigure}[b]{0.45\textwidth}
\centering
\includegraphics[width=\textwidth]{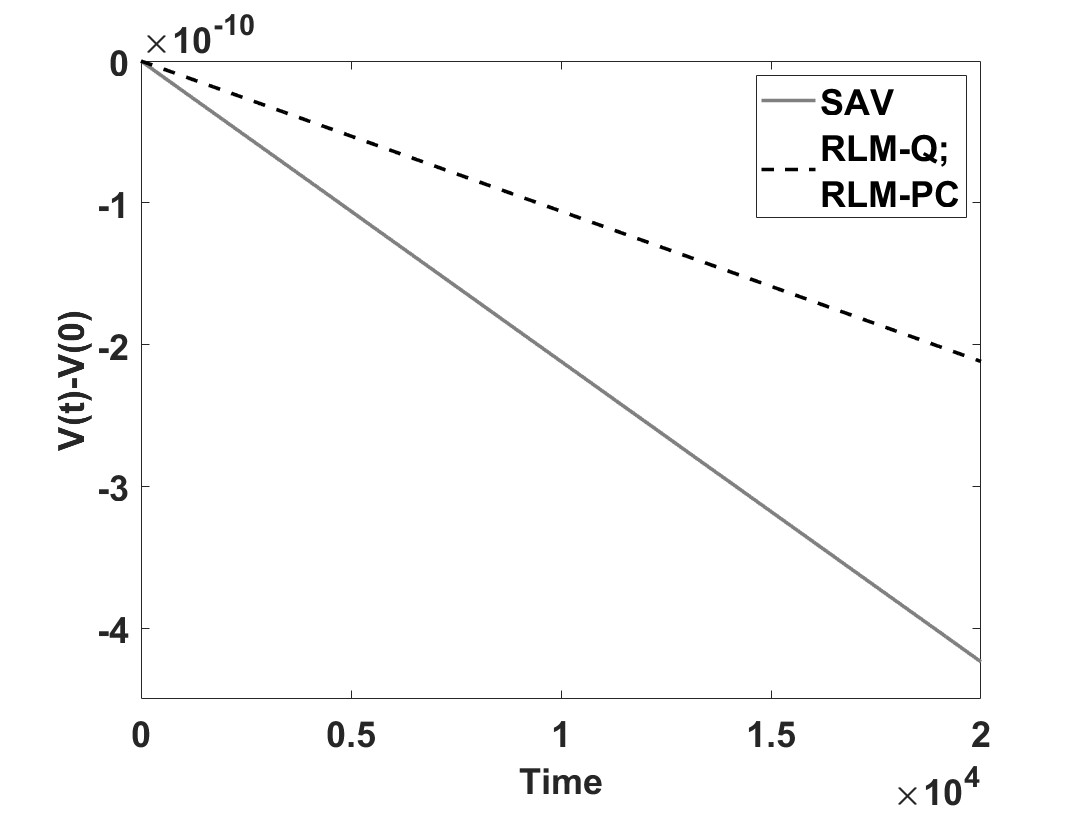}
\caption{Volume drift $V(t)-V(0)$}
\label{subfig:Volumepolymer-b}
\end{subfigure}
\caption{Time evolution of the energy ($E_{\text{SAV}}$ or $E_{\text{RLM}}-\alpha (q^2-1)$) and the volume drift $V(t)-V(0)$ for the Cahn--Hilliard equation with the Flory--Huggins potential, where $V(t)=\int_\Omega \phi(\bx,t)\,\diff\bx$. (\subref{subfig:Volumepolymer-a})~Energies of SAV-CN with $\Csav=10^{2}$ ($\Delta t=10^{-2}$ and $5\times10^{-2}$), SAV-CN with $\Csav=10^{8}$ ($\Delta t=5\times10^{-2}$), and RLM with $\alpha=10^{3}$ ($\Delta t=10^{-2}$ and $5\times10^{-2}$).  (\subref{subfig:Volumepolymer-b})~Volume drift for the SAV and RLM schemes.}
\label{fig:Volumepolymer}
\end{figure}

\subsection{CH-type Equation with Lennard--Jones-type Potential}\label{sec:CH-lennard-jones}
Finally, we consider a system whose bulk free energy is singular at $\phi=0$, which poses a unique challenge for the SAV method. In particular, we consider the liquid thin film coarsening system with the bulk free energy given by a Lennard--Jones-type potential
\begin{equation} \label{eq:free-energy-thin-film}
E(\phi) = \int_\Omega \left( \frac{\varepsilon^2}{2} |\nabla \phi|^2 + F(\phi) \right) \,\diff\bx, \quad F(\phi) = \frac{1}{3\phi^8}-\frac{4}{3\phi^2},
\end{equation}
where $\phi$ represents the film thickness and $\varepsilon$ is the interfacial parameter. The bulk free energy $F(\phi)$ is singular at $\phi=0$, with $F(\phi)\to+\infty$ as $\phi\to0^+$~\cite{zhang2021structure}. In this case, $f(\phi) = F'(\phi) = \frac{8}{3\phi^9}(\phi^6-1)$.

Based on the RLM reformulation in Section~\ref{sec:rlm-methods}, the Cahn--Hilliard-type equation with the mobility operator $M\Delta$ and the Lennard--Jones-type potential is reformulated as follows.
\begin{align}
& \partial_t \phi = M \Delta \mu, \label{eq:CH-lennard-jones-phi} \\
& \mu = -\varepsilon^2 \Delta \phi + q(t) \frac{8}{3\phi^9}(\phi^6-1), \label{eq:CH-lennard-jones-mu} \\
& \frac{d}{dt} \int_\Omega F(\phi) \,\diff\bx + \alpha \,\frac{d \big(q(t)\big)^2 }{dt}= \int_\Omega q(t) \frac{8}{3\phi^9}(\phi^6-1) \partial_t \phi \,\diff\bx, \label{eq:CH-lennard-jones-q}
\end{align}
where $M > 0$ is the mobility constant and $q(0)=1$. For the RLM-Q approach, we use the stabilization technique from~\eqref{eq:F-quadratize-stabilization}, setting $F^{n+1} =(\phi^{n+1})^2 + F(\phibar^{n+1}) - (\phibar^{n+1})^2$, i.e., taking $S=1$.

Since $F(\phi)$ and $f(\phi)=F'(\phi)$ blow up as $\phi\to0$, the SAV auxiliary variable and its associated square-root functional become ill-behaved near $\phi=0$, so the SAV method is inapplicable. Here, we assess the accuracy of the RLM schemes for this singular free energy system. In the first case, we set the computational domain as $\Omega=[0,1]^2$, the interface thickness parameter $\varepsilon=10^{-2}$, the time step size $\Delta t=10^{-2}$, the mobility coefficient $M=10^{-5}$, and the spatial step size $h=1/128$. The initial condition is given by
\begin{equation}
\phi_0(x,y) = 2+0.01\zeta,
\end{equation}
where $\zeta$ is a random variable uniformly distributed on $[-1,1]$. Convergence tests and computational efficiency comparisons are omitted for brevity. The energy and scaling factor evolutions are shown in Figure~\ref{fig:Energy-thinfilm}.

\begin{figure}[H]
\centering
\begin{subfigure}[b]{0.48\textwidth}
\centering
\includegraphics[width=\textwidth]{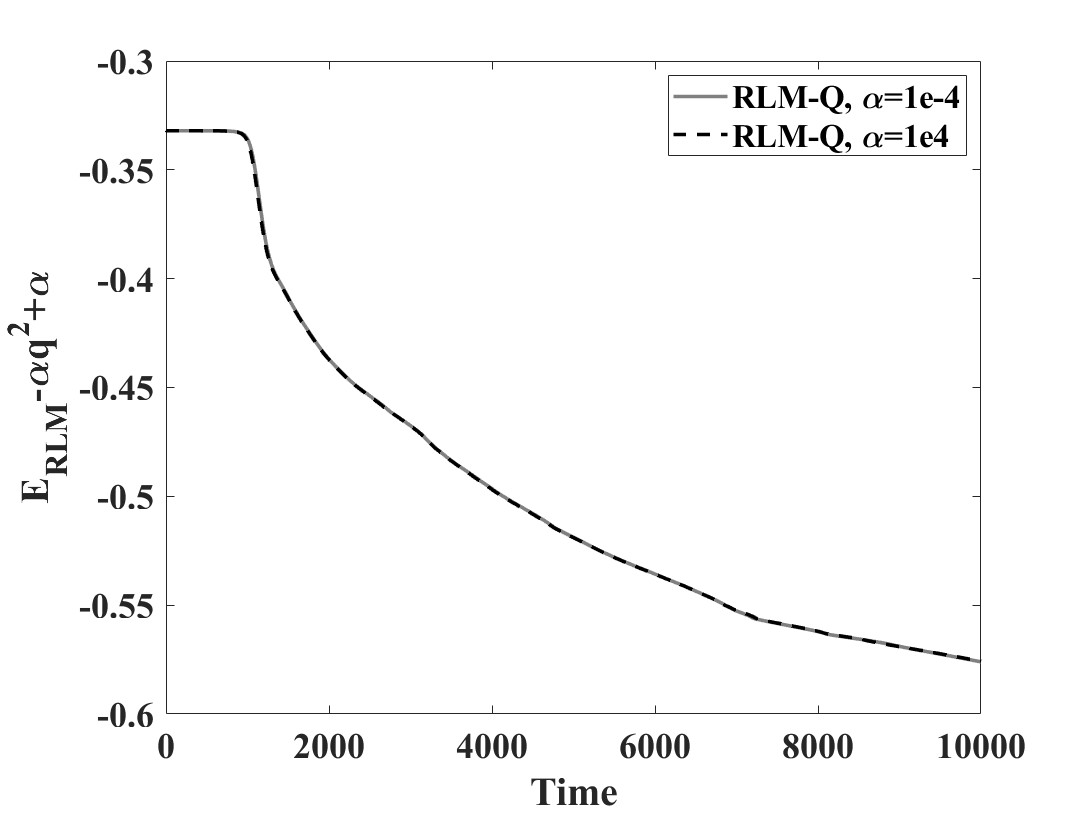}
\caption{RLM-Q}
\label{subfig:Energy-thinfilm-a}
\end{subfigure}
\begin{subfigure}[b]{0.48\textwidth}
\centering
\includegraphics[width=\textwidth]{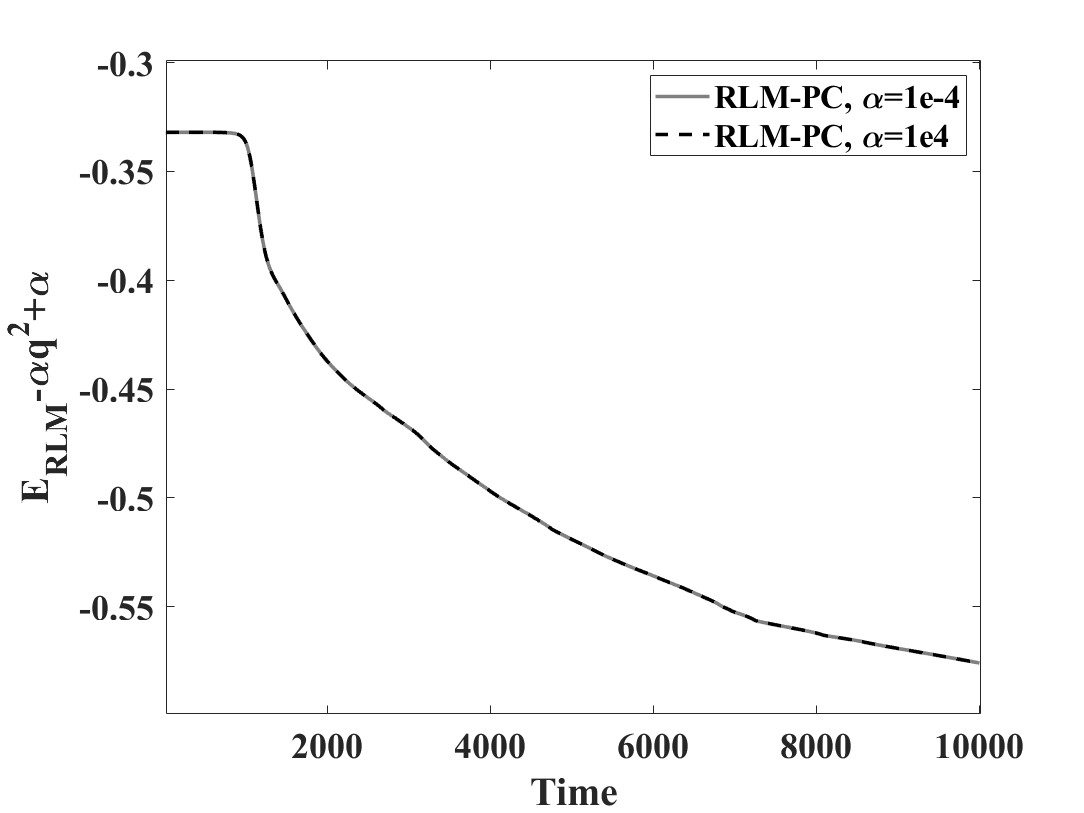}
\caption{RLM-PC}
\label{subfig:Energy-thinfilm-b}
\end{subfigure}
\begin{subfigure}[b]{0.48\textwidth}
\centering
\includegraphics[width=\textwidth]{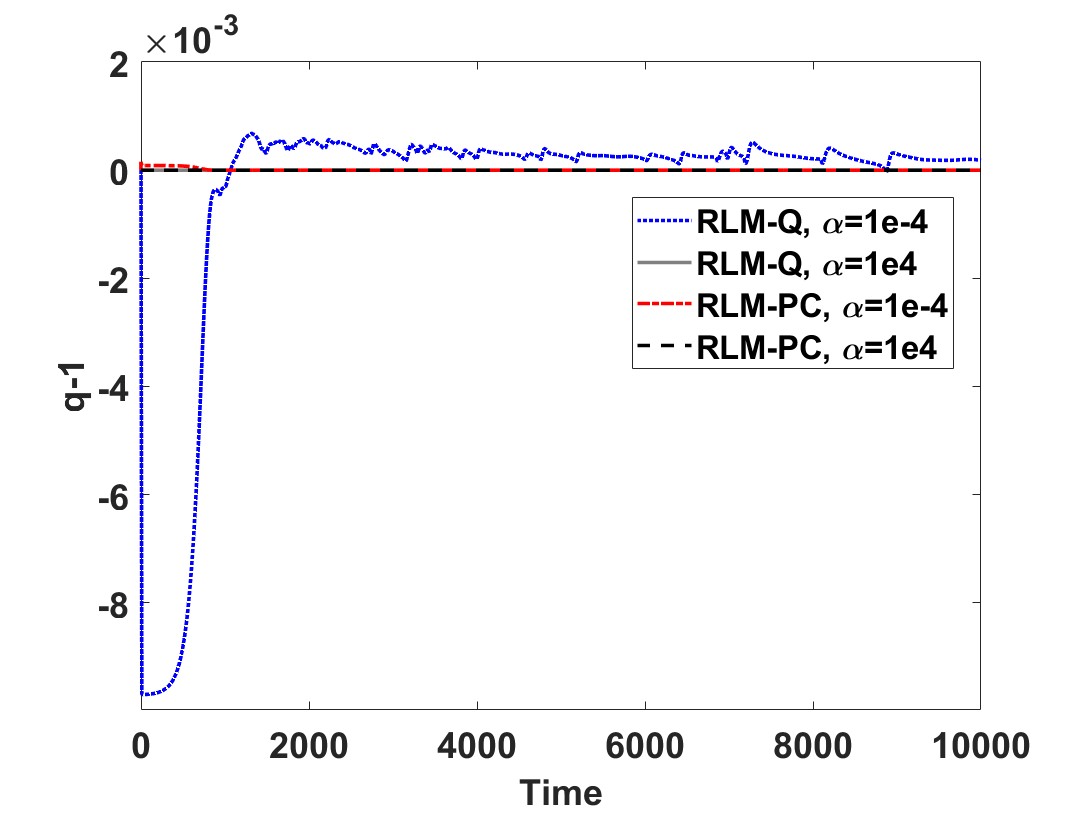}
\caption{Scaling factor}
\label{subfig:Energy-thinfilm-c}
\end{subfigure}
\begin{subfigure}[b]{0.48\textwidth}
\centering
\includegraphics[width=\textwidth]{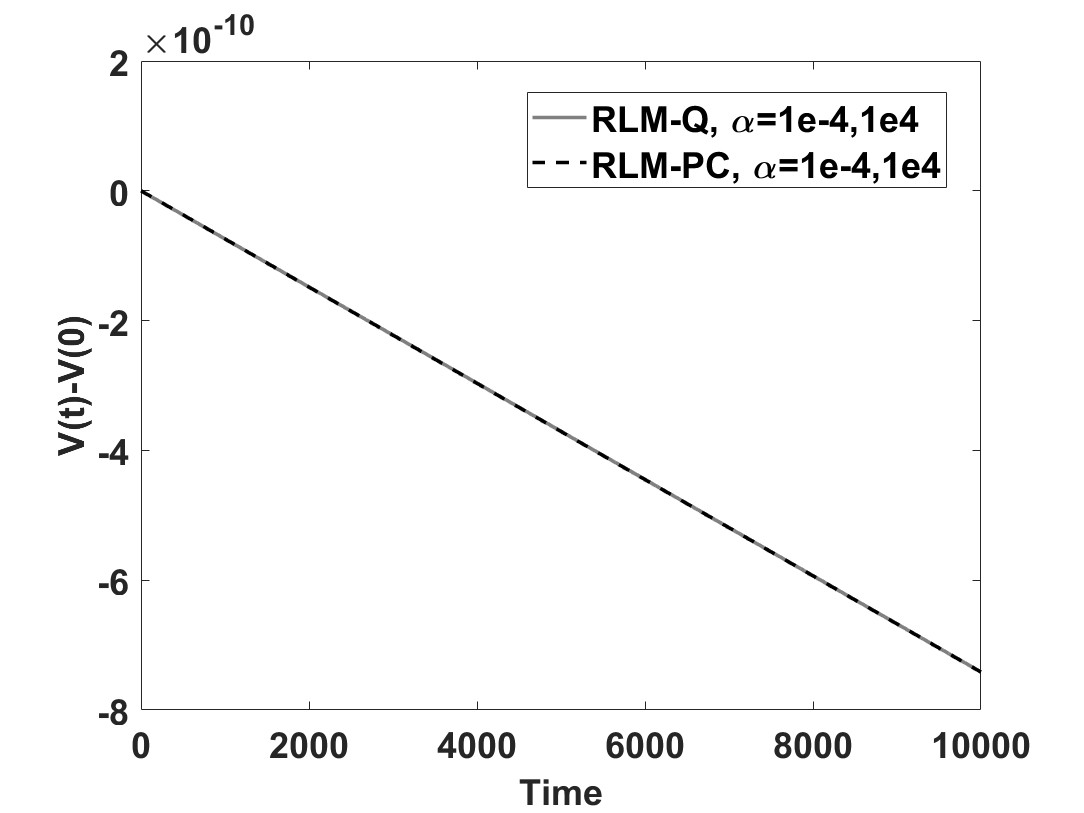}
\caption{Total volume}
\label{subfig:mass-thinfilm}
\end{subfigure}
\caption{Time evolution of the energy, scaling factor $q$, and total volume for the liquid thin-film model with the Lennard--Jones-type potential, computed by the RLM-Q-CN and RLM-PC-CN schemes with $\alpha=10^{-4}$ and $10^{4}$, where the total volume at $t$ is given by $V(t)=\int_\Omega \phi(\bx,t)\,\diff\bx$.}
\label{fig:Energy-thinfilm}
\end{figure}

In Figure~\ref{fig:Energy-thinfilm}(\subref{subfig:Energy-thinfilm-a}) and (\subref{subfig:Energy-thinfilm-b}), we show the time evolution of the original energy with various $\alpha$ using the RLM-Q-CN and RLM-PC-CN schemes. The two schemes yield coincident energy curves for $\alpha\in\{10^{-4},10^4\}$. In Figure~\ref{fig:Energy-thinfilm}(\subref{subfig:Energy-thinfilm-c}), we observe that $q$ approaches 1 with a sufficiently large $\alpha$, and RLM-PC-CN performs better than RLM-Q-CN. RLM-PC-CN is also computationally more efficient than RLM-Q-CN.  We also find that a larger $\alpha$ introduces a smaller error in $q$. In Figure~\ref{fig:Energy-thinfilm}(\subref{subfig:mass-thinfilm}), we find that the total thickness (which can be viewed as the total volume) is conserved in four cases.

Then, we compare snapshots of the phase variable from the simulations above. Figure~\ref{fig:snapshotofthinfilm} shows snapshots obtained by RLM-PC-CN with $\alpha=10^4$ at $T=0,200, 5000, 10000$ to illustrate the coarsening dynamics. Across the tested $\alpha$ values and RLM variants, the snapshots are visually indistinguishable.

\begin{figure}[H]
\centering
\renewcommand{\arraystretch}{0.9}
\setlength{\tabcolsep}{2pt}
\begin{tabular}{c c c c }
\toprule
$T=0$ & $T=200$ & $T=5000$ & $T=10000$  \\
\midrule
\snapshot{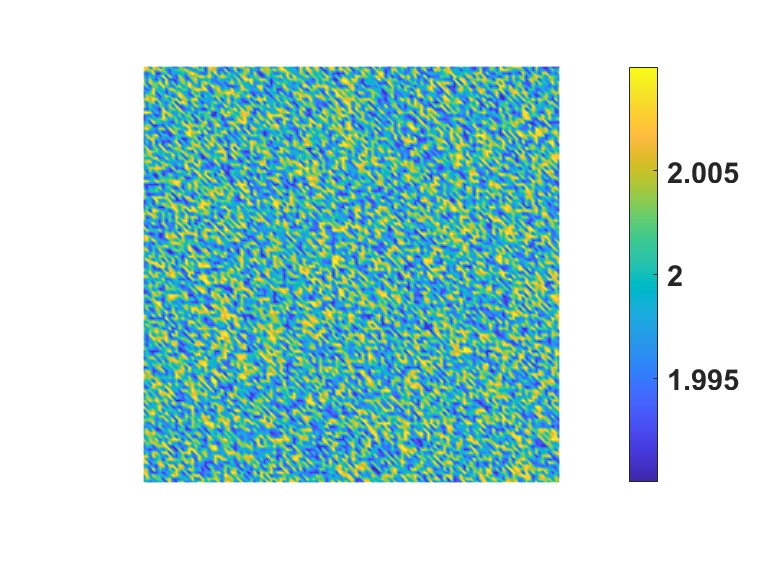} &
\snapshot{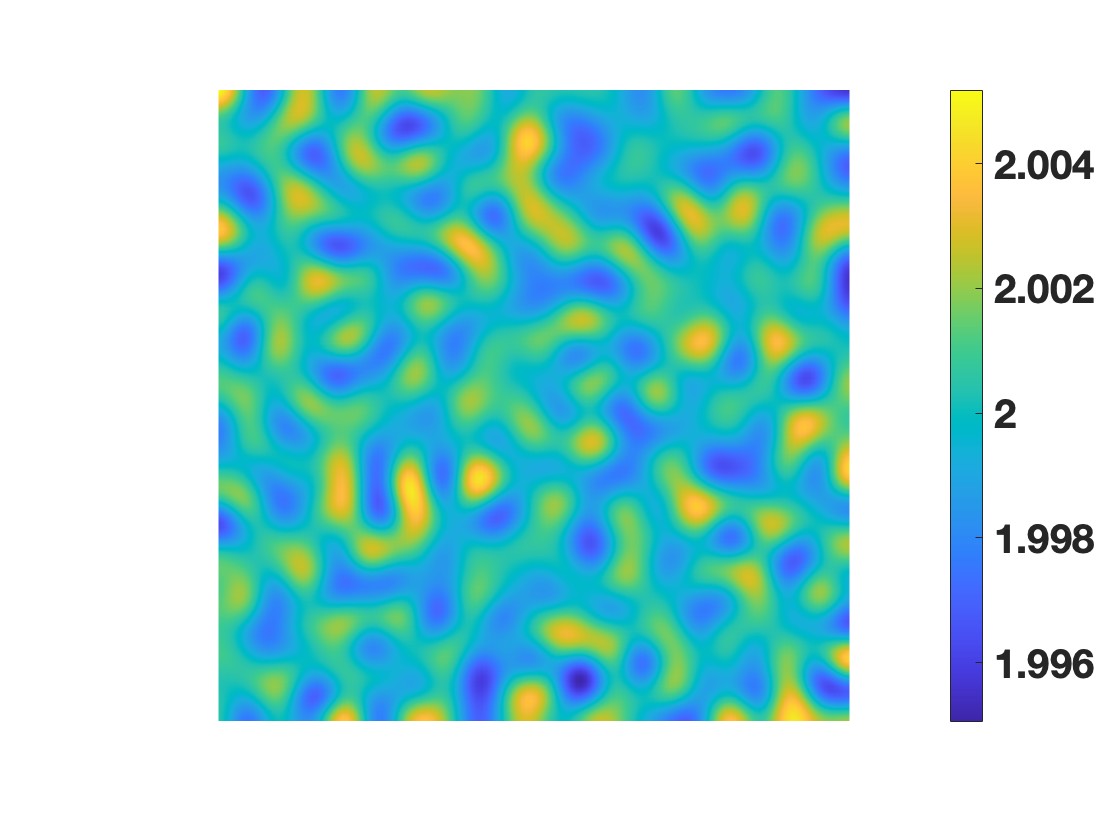} &
\snapshot{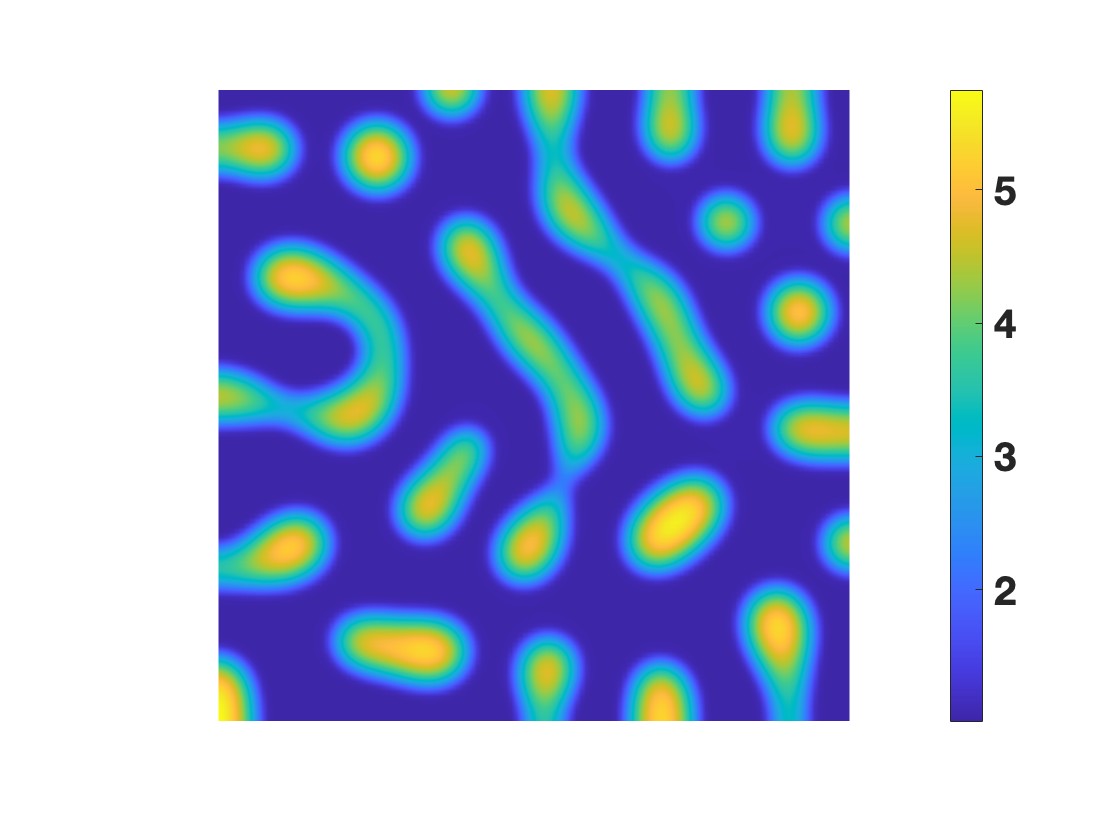} &
\snapshot{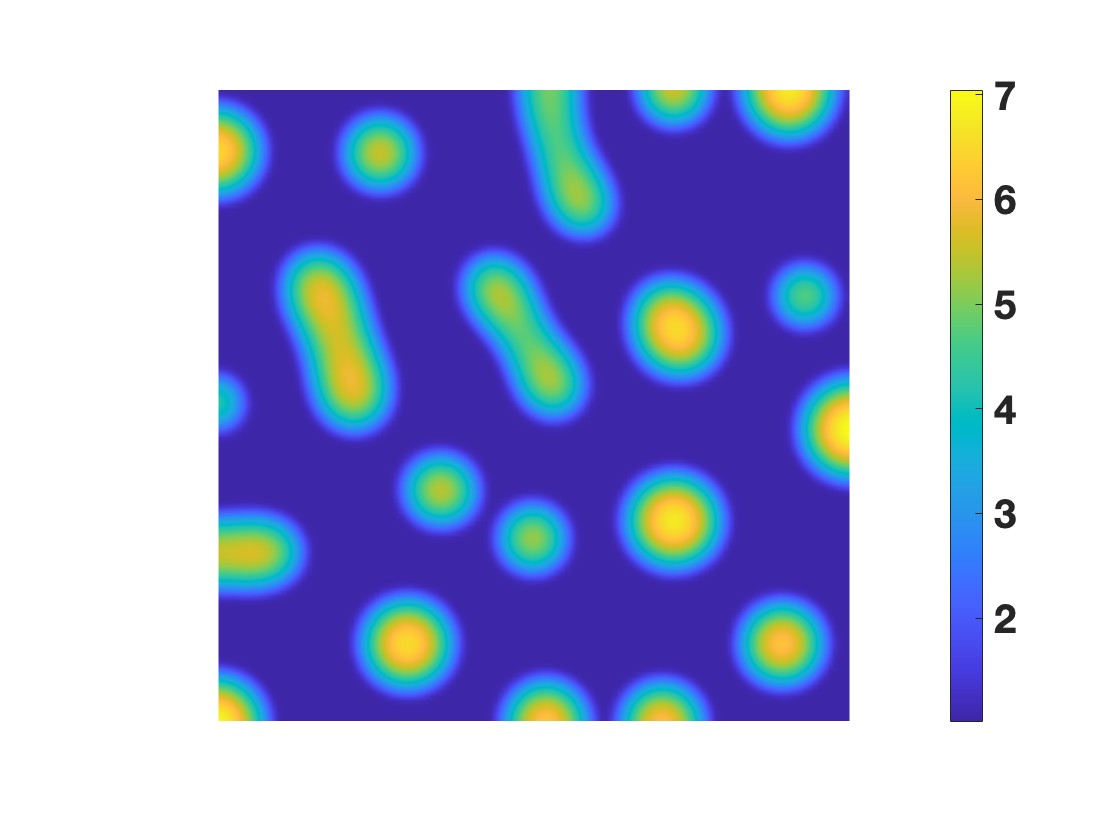}  \\
\bottomrule
\end{tabular}
\caption{Snapshots of the phase variable $\phi$ for the liquid thin film model at $T=0,200, 5000, 10000$, computed by the RLM-PC-CN scheme with $\alpha=10^{4}$.}
\label{fig:snapshotofthinfilm}
\end{figure}

\subsection{Conserved Allen--Cahn Equation with Double-well Potential in 3D}\label{sec:CAC}

Finally, we consider a three-dimensional binary system with a double-well potential. The governing equation is the conserved Allen--Cahn equation~\cite{jing2019second}
\begin{align}\label{eq:CAC-equation}
& \partial_t \phi = -M  (\mu+\tilde L),  \\
& \mu = -\varepsilon^2 \Delta \phi + \phi(\phi^2-1), 
\end{align}
with the energy
\begin{equation} \label{eq:CAC-energy}
E(\phi) = \int_\Omega \left( \frac{\varepsilon^2}{2} |\nabla \phi|^2 + F(\phi) \right) \,\diff\bx+\tilde L\left(\int_\Omega \phi  \,\diff\bx-V_0 \right), \quad F(\phi) = \frac{1}{4}(\phi^2-1)^2, 
\end{equation}
where $\phi$ is the volume fraction, $\varepsilon$ is the interfacial parameter, $\tilde L$ is a spatially constant Lagrange multiplier enforcing volume conservation, and $V_0=\int_\Omega \phi(\bx,0)\,\diff\bx$ is the initial total volume.  The Lagrange multiplier is derived as
$
\tilde L=-\frac{\int_\Omega \phi(\phi^2-1) \diff\bx}{\int_\Omega 1 \diff\bx}.
$
Here $\int_\Omega 1\,\diff\bx=|\Omega|$ denotes the volume of $\Omega$. 

Based on the RLM reformulation strategy in Section~\ref{sec:rlm-methods}, the conserved Allen--Cahn equation is reformulated as follows.
\begin{align}
& \partial_t \phi = -M  \tilde \mu, \\
& \tilde \mu = -\varepsilon^2 \Delta \phi + q(t) \Big( \phi(\phi^2-1)-\frac{\int_\Omega \phi(\phi^2-1) \diff\bx}{\int_\Omega 1 \diff\bx} \Big), \\
& \frac{d}{dt} \int_\Omega F(\phi) \,\diff\bx + \alpha \,\frac{d \big(q(t)\big)^2 }{dt}= \int_\Omega q(t) \Big( \phi(\phi^2-1)-\frac{\int_\Omega \phi(\phi^2-1) \diff\bx}{\int_\Omega 1 \diff\bx} \Big) \partial_t \phi \,\diff\bx.
\end{align}

In the first example, we set the computational domain as $\Omega=[0,1]^3$, the interface thickness parameter $\varepsilon=10^{-2}$, the mobility coefficient $M=1$, the spatial step size $h=1/128$, the relaxation parameter $\alpha=10^3$, and the final time $T_{final}=100$. The initial condition is given by
\begin{equation}
\phi_0(x,y,z) = 1- \frac{1}{2}\Big[1-\tanh\left(\frac{0.2-r_1}{\delta_0}\right)\Big]\Big[1-\tanh\left(\frac{0.1-r_2}{\delta_0}\right)\Big]
\end{equation}
where $r_1=\sqrt{(x-0.3)^2+(y-0.5)^2+(z-0.5)^2}$, $r_2=\sqrt{(x-0.7)^2+(y-0.5)^2+(z-0.5)^2}$, and $\delta_0=0.01$.
Convergence tests are omitted for brevity. This three-dimensional, nonlocal example serves as a numerical demonstration of the RLM framework: the discrete RLM-Q and RLM-PC updates are formed exactly as in Section~\ref{sec:rlm-methods}, now with the volume-projected chemical potential. The energy, volume, and scaling factor evolutions are shown in Figure~\ref{fig:3D}.

\begin{figure}[H]
\centering
\subfig{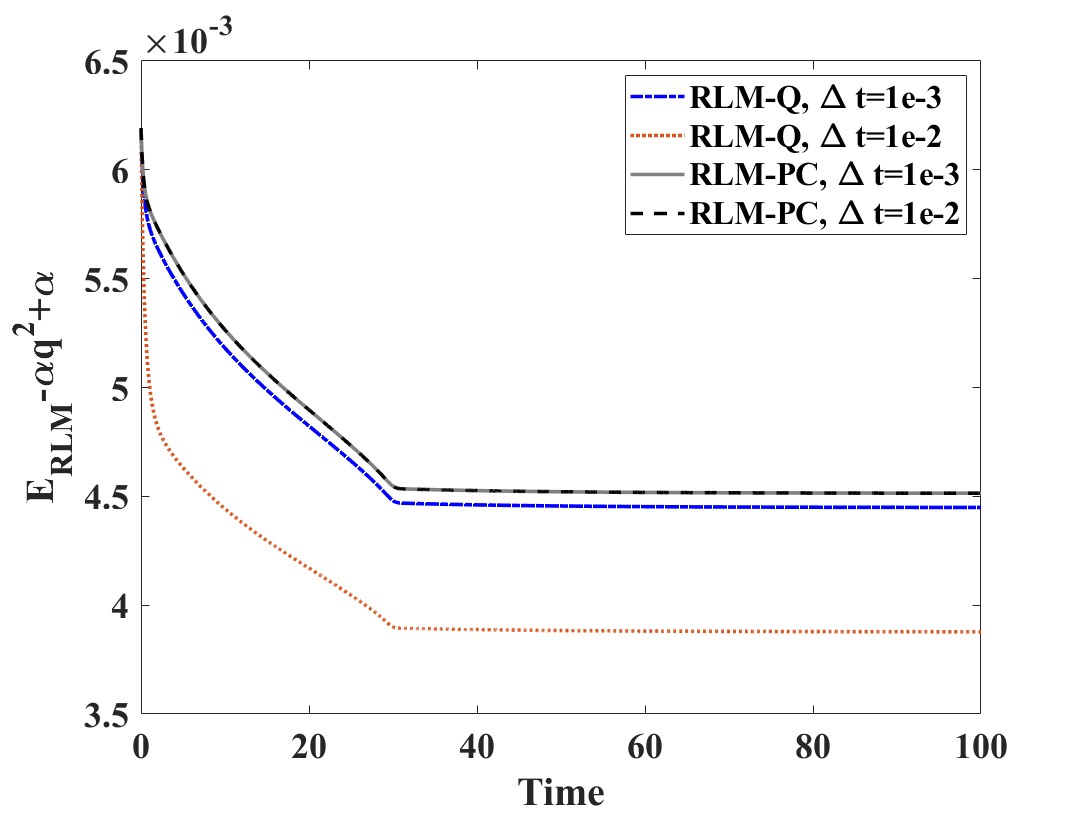}{Original energy}{subfig:3D-modified-energy}
\subfig{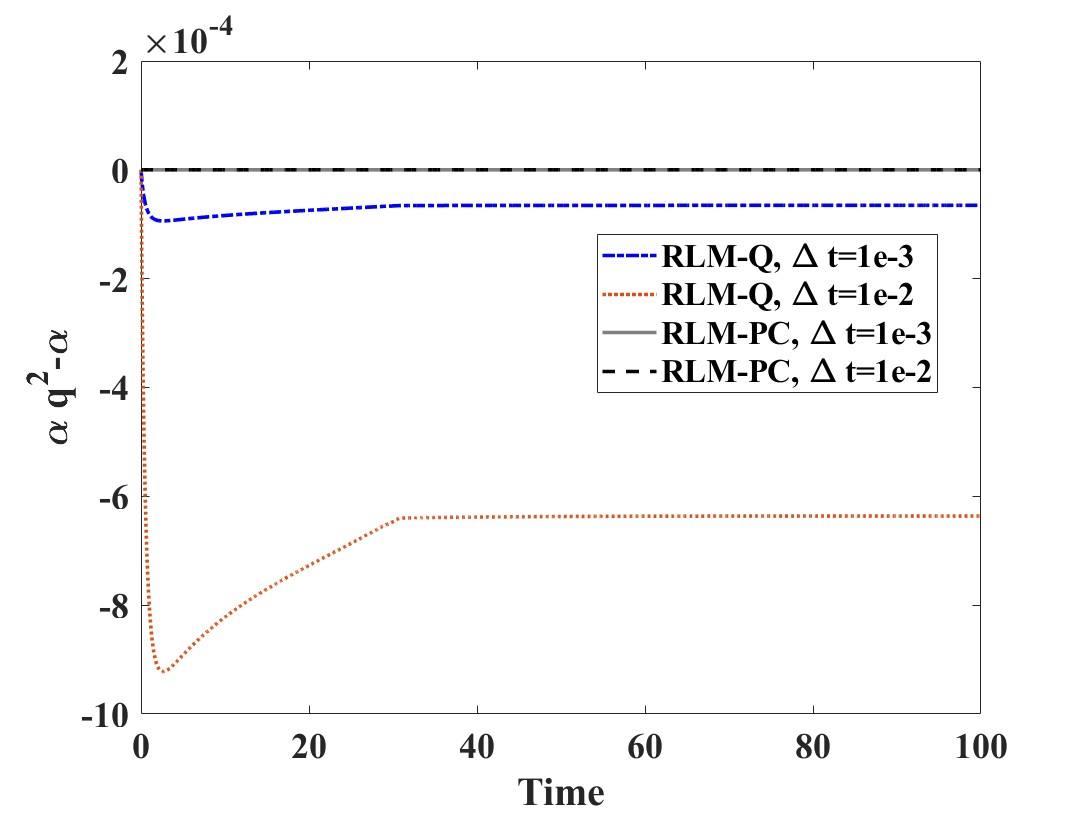}{$E_{\text{RLM}}-E$}{subfig:3D-original-energy}
\subfig{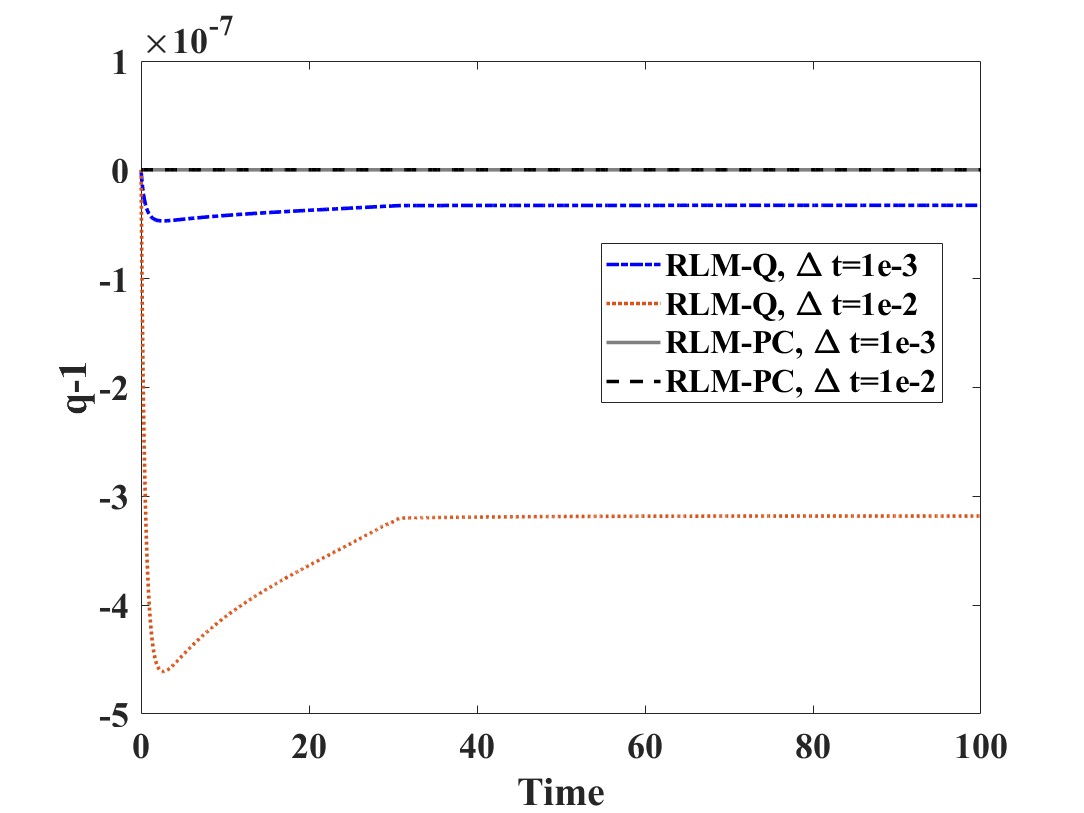}{Scaling factor $q$}{subfig:3D-q}
\subfig{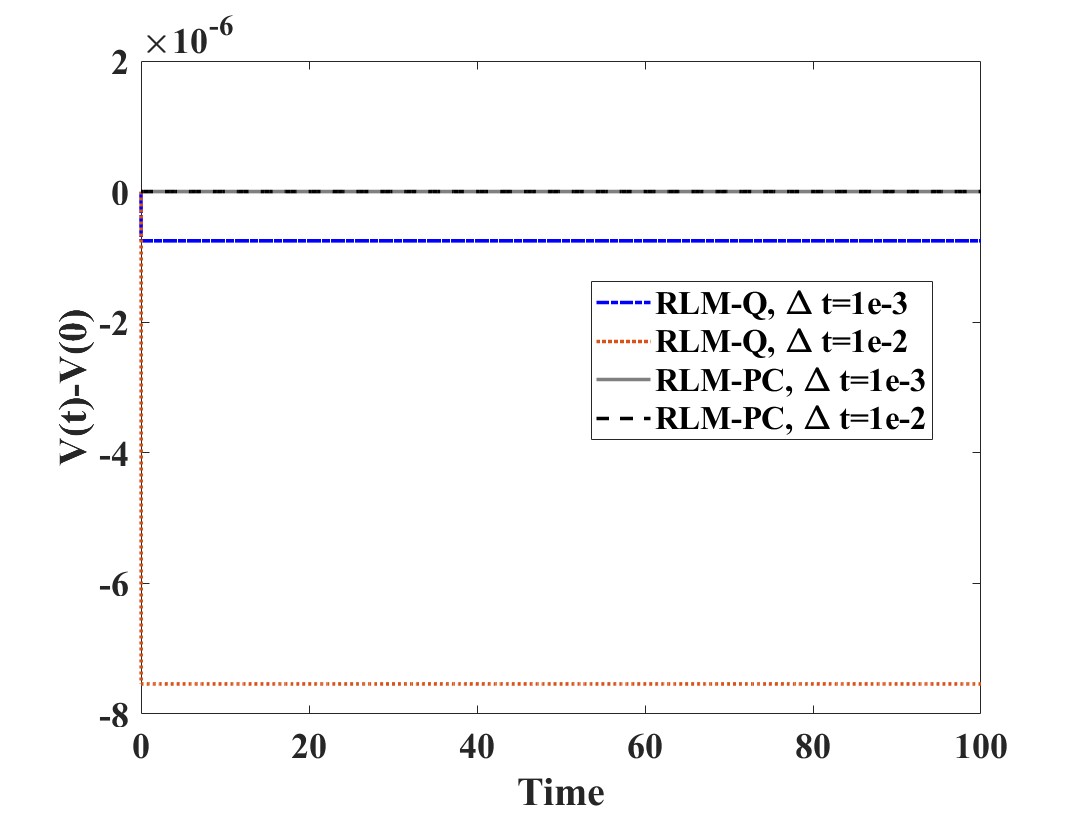}{Total volume}{subfig:3D-volume}
\caption{Time evolution of the energy, scaling factor $q$, and total volume for the conserved Allen--Cahn equation in 3D with $\Delta t=10^{-2}$ and $10^{-3}$. The curves for RLM-PC with different time step sizes coincide. (\subref{subfig:3D-modified-energy})~Original energy. (\subref{subfig:3D-original-energy})~Second class of energy differences. (\subref{subfig:3D-q}) and (\subref{subfig:3D-volume}) show the time evolutions of $q$ and the total volume  for the two RLM methods, where the total volume at $t$ is given by $V(t)=\int_\Omega \phi(\bx,t)d\bx$.}
\label{fig:3D}
\end{figure}
In Figure~\ref{fig:3D}(\subref{subfig:3D-modified-energy}) and (\subref{subfig:3D-original-energy}), we show the time evolution of the original energy $E_{\text{RLM}}-\alpha (q^2-1)$ and the difference of $E_{\text{RLM}}-E$ with various time step sizes by using the RLM-Q-CN and RLM-PC-CN schemes. The results show that RLM-PC-CN performs much better than RLM-Q-CN. In Figure~\ref{fig:3D}(\subref{subfig:3D-q}), we observe that $q$ approaches 1 for $\alpha=1\times 10^3$, and RLM-PC-CN performs better than RLM-Q-CN with the same time step size. In Figure~\ref{fig:3D}(\subref{subfig:3D-volume}), the total volume is conserved in all tests. We also find that a smaller $\Delta t$ introduces a smaller error in both $q$ and $V(t)$.

Next, we compare snapshots of the phase variable from the simulations above. Figure~\ref{fig:snapshotof3D} shows snapshots obtained by RLM-PC-CN with $\alpha=10^3$ at $T=0, 25, 40, 100$ to illustrate the coarsening dynamics. For all tested RLM variants, the snapshots appear visually identical.

\begin{figure}[H]
\centering
\renewcommand{\arraystretch}{0.9}
\setlength{\tabcolsep}{2pt}
\begin{tabular}{c c c c}
\toprule
$T=0$ &$T=25$ & $T=40$ & $T=100$  \\
\midrule
\snapshotthreeD{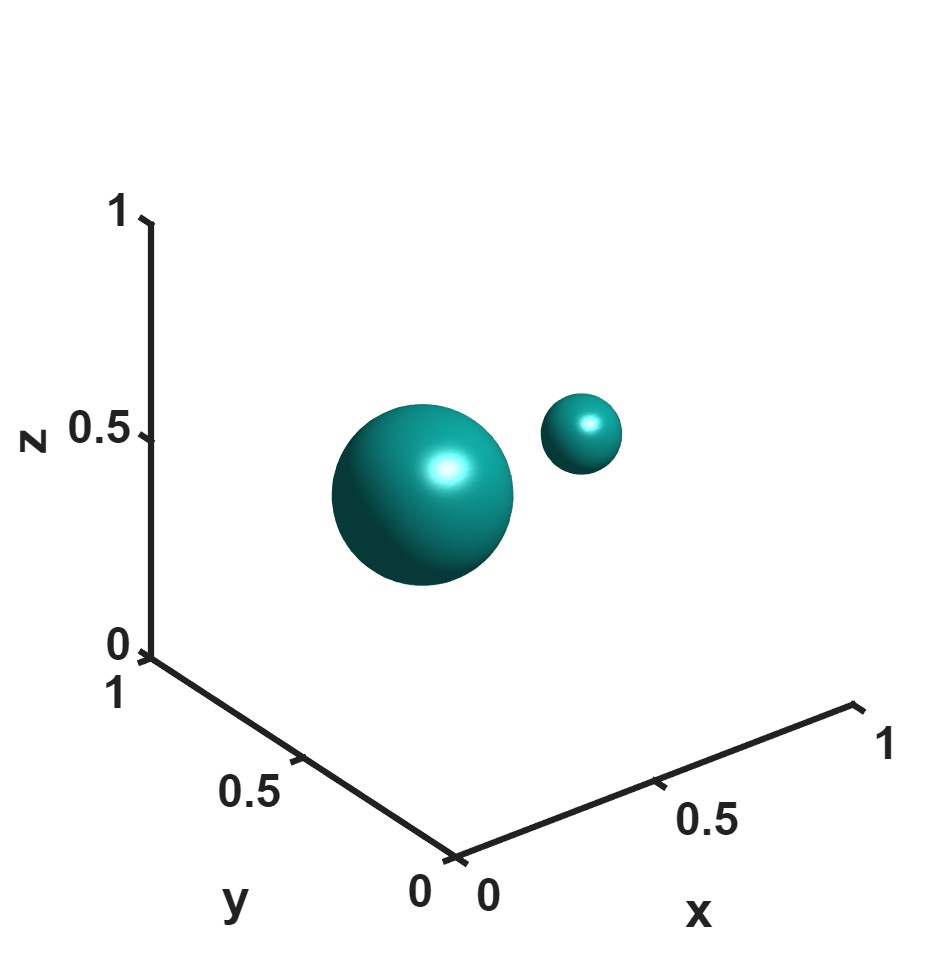} &
\snapshotthreeD{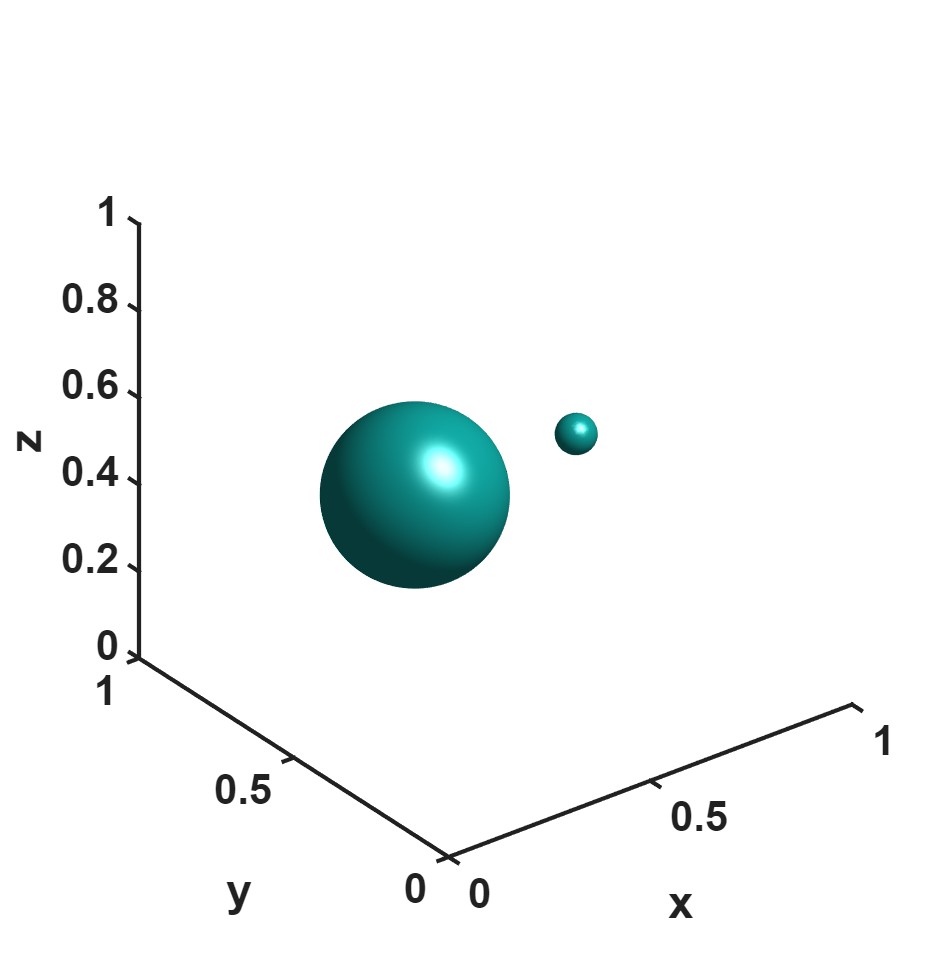} &
\snapshotthreeD{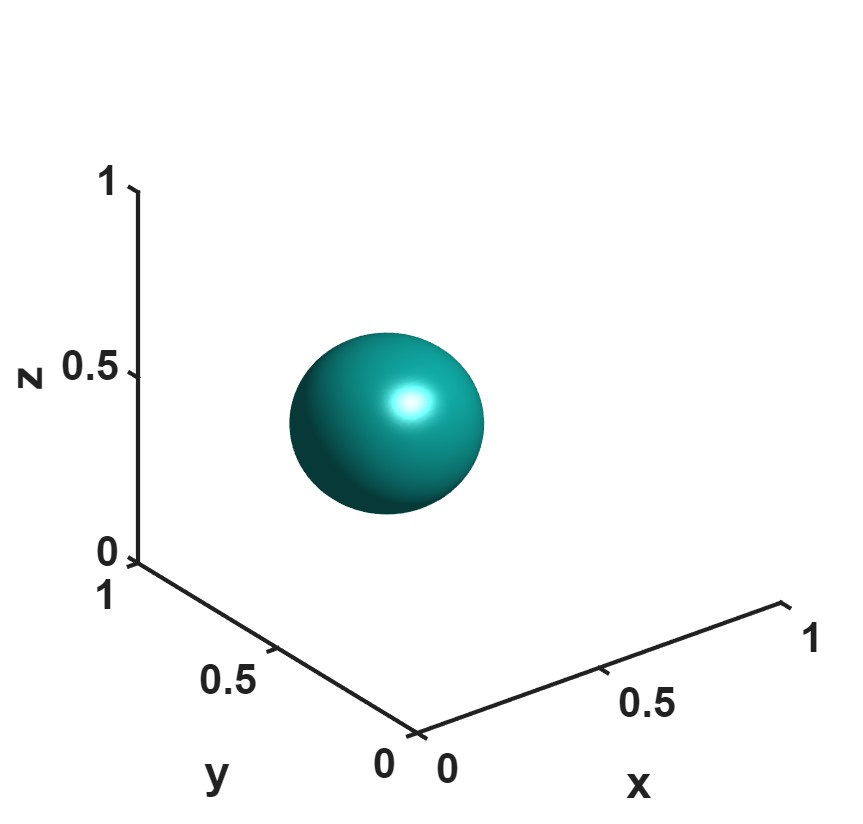} &
\snapshotthreeD{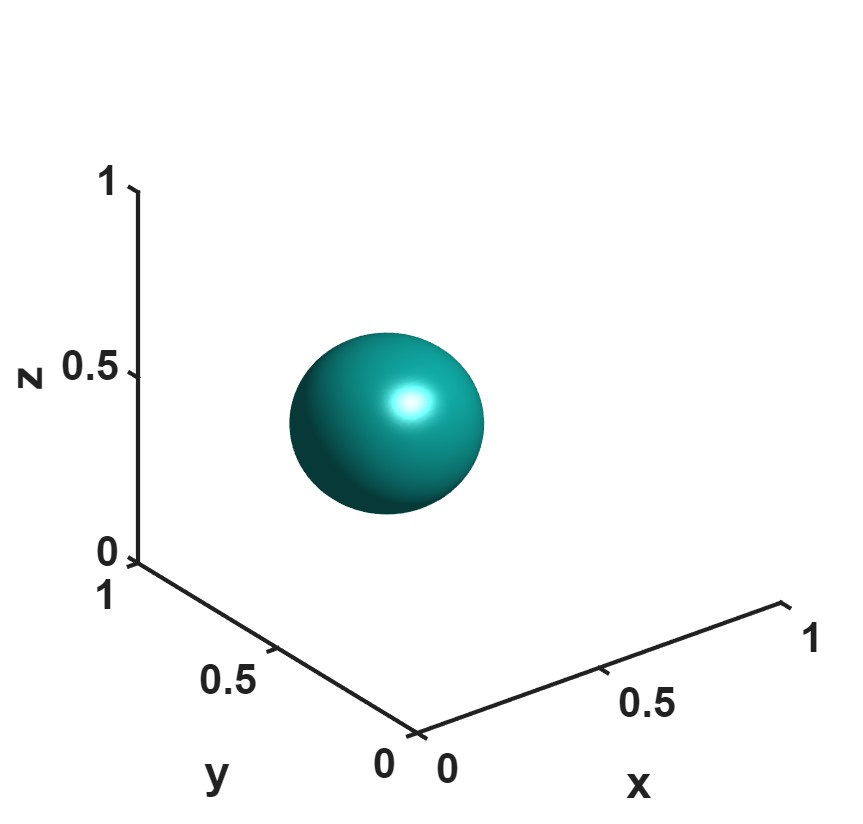}  \\
\bottomrule
\end{tabular}
\caption{Snapshots of the phase variable $\phi$ for the conserved Allen--Cahn equation in 3D at $T=0, 25, 40, 100$, computed by the RLM-PC-CN scheme with $\alpha=10^{3}$.}
\label{fig:snapshotof3D}
\end{figure}
\section{Conclusion}\label{sec:conclusion}

In this paper, we proposed a class of relaxed Lagrange multiplier (RLM) schemes for phase-field models that dissipate a relaxed original energy $\tilde{E}=E+\alpha (q^2-1)$ while tracking the original energy dissipation rate when $q\approx 1$, and remain linear and efficient. The central idea is to augment the classical Lagrange multiplier formulation with a relaxation term controlled by $\alpha\ge 0$, thereby converting the nonlinear scalar multiplier equation into a quadratic equation that admits a closed-form solution. We select the physically relevant root by continuity (i.e., the one closer to $q^n$).

We developed two practical variants. The RLM-Q approach uses a semi-implicit quadratization of the potential to obtain a quadratic equation for $q^{n+1}$, while the RLM-PC approach evaluates the potential at a predicted solution (obtained by setting $q=1$) and leads to an even simpler quadratic equation. The resulting schemes require only two linear solves with constant coefficients per time step, comparable to SAV-type methods, while tracking an energy much closer to the original free energy than the SAV or IEQ modified energies.
We also established unconditional energy stability and mass conservation for both first- and second-order RLM-Q variants. We provided explicit conditions on $\alpha$ under which the quadratic equation for the scaling factor is solvable and admits a unique positive root. The RLM-PC variant uses a predicted increment in the scalar equation. Numerical experiments confirm the expected temporal convergence orders and show that the RLM schemes accurately capture interface dynamics for polynomial, logarithmic (Flory--Huggins), and singular (Lennard--Jones-type) potentials in 2D or 3D. 

In future work, we plan to explore several directions. The first is higher-order time discretizations with rigorous stability and solvability analysis. The second is extensions to more complex phase-field models, such as phase-field crystal and multicomponent systems, as well as to coupled multiphysics problems. We also plan to develop adaptive, problem-dependent strategies for choosing $\alpha$ to balance accuracy and energy stability. Finally, we aim to derive rigorous error estimates, including quantitative bounds for $|q^n-1|$ in terms of $\alpha$ and $\Delta t$.

\section*{Acknowledgments}
{\sloppy
Xiaobo Jing's work is supported by the National Natural Science Foundation of China (No. 12147165) and the Jiangsu Provincial Scientific Research Center of Applied Mathematics (No.\ BK20233002). Additional support was provided by the Basic Research Program of Jiangsu (No.\ BK20252120) and the Start-up Research Fund of Southeast University (No.\ RF1028623369). Jia Zhao acknowledges support from the National Science Foundation under grant NSF-DMS-2513764.\par}

\end{document}